\documentclass[12pt]{article}

\usepackage{amsfonts}
\usepackage{amsmath}
\usepackage{amscd}
\usepackage{amsthm}

\date{}
         \newcommand{\MII}[2]{\mbox{$\int\limits^{\stackrel{#1}{\rightarrow}}_{#2}$}}
         \newcommand{\MIIL}[2]{\mbox{$\int\limits^{\stackrel{#1}{\leftarrow}}_{#2}$}}

         \newtheorem{thm}{Theorem}[section]

         \numberwithin{equation}{section}

         \newcommand{\be}{\begin{equation}}
         \newcommand{\ee}{\end{equation}}
         \newcommand{\ba}{\begin{array}}
         \newcommand{\ea}{\end{array}}
         \newcommand{\la}{\label}

         \newcommand{\rf}[1]{(\ref{#1})}

\begin{document}




\hoffset =-1.25truecm \voffset =-2.20truecm




        \begin{center}
 \textsc{\textbf{SOLITONIC COMBINATIONS, COMMUTING NONSELFADJOINT OPERATORS, AND APPLICATIONS}}%
        \footnote{Partially supported by Scientific Research Grant RD-08-61/2019 of University of Shumen and partially supported by the National Scientific Program "Information and Communication Technologies for a Single Digital Market in Science, Education and Security (ICTinSES)", financed by the Ministry of Education and Science.}
\vspace{0.9cm}
\\
        \textsc{\textbf{GALINA S. BORISOVA}}
        \\
        \vspace{0.5cm}
        Faculty of Mathematics and Computer Science,
        \\ Konstantin Preslavsky University of Shumen,
        \\ 9712 Shumen, Bulgaria
        \\ e-mail: \verb"g.borisova@shu.bg"
        \end{center}

%


        \begin{abstract}
In this paper, applications of the connection between the soliton theory and the commuting nonselfadjoint operator theory, established by M.S. Liv\v sic and Y. Avishai, are considered. An approach to the inverse scattering problem and to the wave equations is presented, based on the Liv\v{s}ic operator colligation theory (or vessel theory) in the case of commuting bounded nonselfadjoint operators in a Hilbert space, when one of the operators belongs to a larger class of nondissipative operators with asymptotics of the corresponding nondissipative curves. The generalized Gelfand-Levitan-Marchenko equation of the cases of different differential equations (the Korteweg-de Vries equation, the Schr\"{o}dinger equation, the Sine-Gordon equation, the Davey-Stewartson equation) are derived. Relations between the wave equations of the input and the output of the generalized open systems, corresponding to the Schr\"{o}dinger equation and the Korteweg-de Vries equation, are obtained. In these two cases, differential equations (the Sturm-Liouville equation and the 3-dimensional differential equation), satisfied by the components of the input and the output of the corresponding generalized open systems, are derived.
%
%
        \end{abstract}


%
%
\vspace{1cm}
        \hfill
        \textit{Dedicated to the memory of  prof. Kiril P. Kirchev}

%
%
        \par
        \section{Introduction}\label{s1}
        \par
        \par
M.S. Liv\v sic and Y. Avishai in their paper \cite{LIVSIC-AVISHAI} have established the connection between two mathematical theories: the theory of commuting nonselfadjoint operators \cite{LIVSIC-LECTURENOTES} and the theory of solitons \cite{SCOTT}. The essence of this connection is the Marchenko method for solving nonlinear differential equations (see \cite{MARCHENKO}). M.S. Liv\v sic and Y. Avishai have illustrated the idea behind this connection with the Korteweg-de Vries equation. They introduced an operator-valued function $\Gamma^{-1}(x,t)\Gamma_x(x,t)$ with
$\Gamma(x,t)=e^{i(tA+xB)}+M_1e^{i(tA^*+xB^*)}M$, considered as solitonic combinations and related them to the theory of couples $(A,B)$ of commuting nonselfadjoint operators in a Hilbert space $H$ when the operator $B$ is dissipative with a a zero limit
        \be\la{0-1}
        \lim\limits_{x\to\infty}(e^{ixB}f,e^{ixB}f)=0, \;\;\;\;\;\;\ f\in H
        \ee
where $( , )$ denotes an inner product in $H$. The scalar wave equations, generated by solitonic combinations, have been analyzed and the Gelfand-Levitan-Marchenko equation of the inverse scattering problem has been derived in this case of the operator $B$.
        \par
In \cite{JMAA} G.S. Borisova and K.P. Kirchev have obtained solutions  of different nonlinear differential equations (including new scalar solutions): the Schr\"{o}dinger equation, the Heisenberg equation, the Sine-Gordon equation, the Davey-Stewartson equation, using established connection between solitonic combinations and $n$-operator colligation theory, when one of the operators, $B$, belongs to a larger class of bounded nondissipative operators in a Hilbert space $H$, presented as couplings of dissipative and antidissipative operators with real absolutely continuous spectra. An essential property of the operators $B$ from this class nondissipative operators is the existence of the asymptotics of the corresponding continuous curves $e^{ixB}f$ (as $x\to\pm\infty$) and the storng limits
        $$s-\lim\limits_{x\to\pm\infty}e^{-ixB^*}e^{ixB}=\widetilde{S}^*_{\pm}\widetilde{S}_{\pm},$$
($\widetilde{S}_{\pm}$ are defined by \rf{21-2}) which have been obtained by the author in \cite{IEOT1}, \cite{SERDICA2}, \cite{SERDICA3} and the explicit form of these limits in terms of multiplicative integrals and finite dimensional analogue of the classical gamma function is presented in  \cite{IEOT1}.
The natural consideration of this class of operators follows from the system-theoretic significance of the colligation which is connected with the multiplication theorem of the corresponding ttransfer functions \cite{LIVSIC6}. The triangular model for this class of operators $B$ is introduced and investigated by the author in \cite{BAN2}, \cite{IEOT1}, \cite{SERDICA2}, \cite{SERDICA3}. The existence and the explicit form of the nonzero limits$\lim\limits_{x\to\pm\infty}(e^{ixB}f,e^{ixB}f)=(\widetilde{S}^*_{\pm}\widetilde{S}_{\pm}f,f)$, ($f\in H$)
obtained in \cite{IEOT1}
ensure the applications of the connection between solitonic combinations and $n$-tuples of commuting nonselfadjoint operators.
        \par
The theory of operator colligations in Hilbert spaces is deeply connected with the problem of wave dispersions, collective motions of spatio-temporal systems, classical scattering theory  (\cite{LIVSIC-OPERATORS-FIELDS}, \cite{LIVSIC-SYSTEM THEORY}, \cite{LIVSIC-VORTICES}, \cite{WAKSMAN}, and etc.). In series of papers (for example, \cite{LIVSIC-VORTICES}, \cite{MELNIKOV}, \cite{MELNIKOV1}, \cite{MELNIKOV2}, \cite{ALPAY}, and etc.) many common points of operator colligations (vessels) with the classical scattering theory of Sturm-Liouville operator, inverse scattering of linear differential equations are presented.
        \par
In this paper we continue with the representation of the connection between the commuting nonselfadjoint operator theory and the soliton theory and  applications of this connection, using the class of couplings of dissipative and antidissipative operators with real absolutely continuous spectra.
        \par
It turns out that in the case when the operator $B$ is a coupling of dissipative and antidissipative operators,
the mode $v_h(\xi)$ (connected with the output $v(\xi,\tau)$ of the corresponding generalized open system and satisfying the corresponding output matrix wave equation) of a vector function
        $$
        h(x,t)=S(x,t)g-iBg
        $$
where $g$ belongs to the nonhermitian subspace $G_B=(B-B^*)H$ of the operator $B$ and S(x,t) is a solitonic combination of different nonlinear differential equations (for example, the Schr\"{o}dinger equation, the Sine-Gordon equation, the Davey-Stewartson equation) satisfies the integral equation (which is the generalized form of the Gelfand-Levitan-Marchenko equation). This integral equation is connected with the transfer function of the operator $B$, with the complete characteristic function of the couples $(A,B)$ or triplets $(A,B,C)$, with the Fourier transform of a function from the Hilbert space. It is shown that the output $v(\xi,\tau)$ of the corresponding generalized open system is determined uniquely as a solution of the corresponding matrix wave equation, satisfying the condition $v(\xi,0)=v_h(\xi)$.
        \par
In Section \ref{s2} we remind some preliminary results concerning commuting nonselfadjoint operators, which generate generalized open systems, corresponding collective motions, and matrix wave equations (considered in \cite{JMAA}) in the case of $n$-tuples of commuting nonselfadjoint operators, when one of them is a coupling of dissipative and antidissipative operators with real absolutely continuous real spectra. We also present the results concerning a space of solutions of the matrix wave equations. These results are considered in \cite{JMAA} and \cite{BAN7}.
        \par
In Section \ref{s3} we derive the generalized Gelfand-Levitan-Marchenko equation of the inverse scattering problem using the connection between Korteweg-de Vries equation and appropriate couples of commuting nonselfadjoint operators, when one of them is a coupling of dissipative and antidissipative operators with real spectra. The properties of this integral equation are obtained, based on the results of commuting nonselfadjoint operator theory, concerning open systems, collective motions, characteristic function of the operator. At the end of this section an interesting connection with results of tha paper \cite{BAN6} for solutions of the Sturm-Liouville systems in the special case of the operator $B$ is presented.
        \par
Section \ref{s4} and Section \ref{s5} are dedicated to the similar problem as in Section \ref{s3}, but connected with other nonlinear differential equations: the Schr\"{o}dinger equation, the Sine-Gordon equation, the Davey-Stewartson equation and generalized Gelfand-Levitan-Marchenko equation using essentially the solitonic combinations, obtained in \cite{JMAA} and connected with appropriate couples and triplets of commuting nonselfadjoint operators.
        \par
In Section \ref{s6}, Section \ref{s7}, Section \ref{s8}, Section \ref{s9} we consider the generalized open systems, connected with obtaining of solutions of the nonlinear Schr\"{o}dinger equation and the Korteweg-de Vries equation in the special case of separated variables in the input, the state, and the output of the corresponding generalized open systems for appropriate couples $(A,B)$ of commuting nonselfadjoint operators. We consider separately two cases -- when operators $A$, $B$ do not depend on the variables $x$ and $t$ and the case when $A$, $B$ depend on the saptial variable $x$. We derive what kind of differential equations are satisfied by the components of the input
and the output of the corresponding generalized open systems. It turns out that the components of the input and the output satisfy the Sturm-Liouville differential equation from the form
        $$
        Ly=-y''+q(x)y=\lambda y
        $$
in the case of the nonlinear Schr\"{o}dinger equation and the 3-dimensional differential equation from the form
        $$
        Ly=\frac{d^3y}{dy^3}-p(x)\frac{dy}{dx}-q(x)y=\lambda y
        $$
in the case of the Korteweg-de Vries equation.
        \par
This paper is a continuation of the papers \cite{JMAA} and \cite{BAN7} of the author and presents other aspects of the connection between the soliton theory and the theory of commuting nonselfadjoint operators, established by M.S. Liv\v sic and Y. Avishai in \cite{LIVSIC-AVISHAI}.

        \section{Preliminary results}\la{s2}
        \par
    In this section we will present some preliminary results concerning commuting nonselfadjoint operators generating generalized open systems, corresponding collective motions, and matrix wave equations (considered by G.S. Borisova and K.P. Kirchev in \cite{JMAA}). These results are obtained in the case of $n$-tuples of commuting nonselfadjoint bounded operators when one of them belongs to a larger class of nondissipative operators, presented as couplings of dissipative and antidissipative operators and they are presented in \cite{BAN7}.
%
        \par
    Let us consider commuting nonselfadjoint bounded linear operators $A_1,A_2,\ldots,A_n$ ($A_kA_s=A_sA_k$, $k,s=1,2,\ldots,n$) in a Hilbert space $H$.
    Let these operators be embedded in a \textit{commutative regular colligation}
        \be\la{1}
        X=(A_1,A_2,\ldots,A_n;H,\Phi,E;\sigma_1,\sigma_2,\ldots,\sigma_n,\{\gamma_{ks}\},\{\widetilde{\gamma}_{ks}\},
        k,s=1,2,\ldots,n)
        \ee
    where $E$ is another Hilbert space, $\Phi$ is a bounded linear mapping of $H$ into $E$,
    $\sigma_1,\sigma_2,\ldots,\sigma_n$, $\{\gamma_{ks}\}$, $\{\widetilde{\gamma}_{ks}\}$, $(k,s=1,2,\ldots,n)$
are bounded linear selfadjoint operators in $E$ (where $\gamma_{ks}=\gamma_{ks}^*=-\gamma_{sk}$) and they satisfy the next conditions:
        \be\la{4}
        (A_k-A_k^*)/i=\Phi^*\sigma_k\Phi,
        \ee
        \be\la{5}
        \sigma_s\Phi A_k^*-\sigma_k\Phi A_s^*=\gamma_{ks}\Phi,
        \ee
        \be\la{6}
        \widetilde{\gamma}_{ks}=\gamma_{ks}+i(\sigma_k\Phi\Phi^*\sigma_s-\sigma_s\Phi\Phi^*\sigma_k)
        \ee
    for $k,s=1,2,\ldots,n$.
Instead of the term "regular colligation" one can use the term "vessel", that has been coined in \cite{LIVSIC-AT-AL}.
        \par
    In the most of important cases the space $E$ satisfies the condition $\dim E<\infty$ (which implies that the operators $A_1,A_2,\ldots,A_n$ have finite dimensional imaginary parts).
        \par
        In what follows, we assume that $\dim E<\infty$ and $\bigcap\limits_{k=1}^n\ker \sigma_k=\{0\}$. If $range\; \Phi=E$, the colligation $X$
is called a \textit{strict colligation}.
        \par
    The system-theoretic interpretation of $n$-operator colligation leads to an open $n$-dimensional system. We consider the generalized open system (introduced by G.S. Borisova and K.P. Kirchev in \cite{JMAA}) from the form
        \par
        \be\la{7}
        \left\{
        \ba{l}
        i\frac{1}{\varepsilon_k}\frac{\partial }{\partial x_k}f(x)+A_kf(x)=\Phi^*\sigma_ku(x), \;\;\; k=1,2,\ldots,n,
        \\
        v(x)=u(x)-i\Phi f(x),
        \ea
        \right.
        \ee
    where $x=(x_1,x_2,\ldots,x_n)$, $f(x)|_{\Gamma_+}=f_0(x)$ ($\Gamma_+=\partial \mathbb{R}^n_+$),
   $\varepsilon_1,\varepsilon_2,\ldots,\varepsilon_n\in\mathbb{C}$ are constants and the vector functions
        $u(x)=u(x_1,x_2,\ldots,x_n)$,
        $v(x)=v(x_1,x_2,\ldots,x_n)$,
        $f(x)=f(x_1,x_2,\ldots,x_n)$
are the input, the output, and the internal state of the open system \rf{7}.
(In the cases, considered by M.S. Liv\v sic, the open systems are when
    $\varepsilon_1=\varepsilon_2=\cdots=\varepsilon_n=1$.)
        \par
        Direct calculations (\cite{JMAA}, Theorem 6.1) show that the system \rf{7} has a solution if the function $f_0(x)$ on $\Gamma_+$ satisfies the equations from
\rf{7} and the vector function $u(x)$ is a solution of the system
        \be\la{9}
        \sigma_k\left(-i \frac{1}{\varepsilon_s}\frac{\partial u}{\partial x_s}\right)-
        \sigma_s\left(-i \frac{1}{\varepsilon_k}\frac{\partial u}{\partial x_k}\right)+
        \gamma_{sk}u=0,
        \ee
    ($k,s=1,2,\ldots,n$). In other words (following M.S. Liv\v sic and Y. Avishai \cite{LIVSIC-AVISHAI}) $u(x)$ satisfies the  matrix wave equations \rf{9}.
        \par
        The system \rf{7} is over determined in the case when $n\geq 3$ and to avoid this one has to consider the additional conditions for the operators $\{\sigma_k,\gamma_{ks}\}$, $k,s=1,2,\ldots,n$ when $\det\sigma_n\neq 0$. These conditions have been introduced by V. Zolotarev in the paper \cite{ZOLOTAREV-TIME CONES} and have the form:
        \be\la{a10}
        \sigma_n^{-1}\sigma_k\sigma_n^{-1}\sigma_s=\sigma_n^{-1}\sigma_s\sigma_n^{-1}\sigma_k,
        \ee
        \be\la{a11}
        \sigma_n^{-1}\sigma_k\sigma_n^{-1}\gamma_{sn}+\sigma_n^{-1}\gamma_{kn}\sigma_n^{-1}\sigma_s
        =
        \sigma_n^{-1}\gamma_{kn}\sigma_n^{-1}\gamma_{sn}+\sigma_n^{-1}\gamma_{sn}\sigma_n^{-1}\gamma_{kn},
        \ee
        \be\la{12}
        \sigma_n^{-1}\gamma_{kn}\sigma_n^{-1}\gamma_{sn}=
        \sigma_n^{-1}\gamma_{sn}\sigma_n^{-1}\gamma_{kn}
        \ee
    $k,s=1,2,\ldots,n-1$. The  conditions \rf{a10}, \rf{a11}, \rf{12} follow from the equalities of the mixed partial derivatives
        $
        \frac{\partial^2u}{\partial x_k\partial x_s}=\frac{\partial^2u}{\partial x_s\partial x_k},
        $
        $
        k,s=1,2,\ldots,n.
        $
        Then from \rf{9} it follows that
        \be\la{14}
        \gamma_{ks}=\sigma_s\sigma_n^{-1}\gamma_{kn}-\sigma_k\sigma_n^{-1}\gamma_{sn}, \;\;\; k,s=1,2,\ldots,n.
        \ee
    Consequently, when $\sigma_n$ is invertible matrix, the commutative regular colligation is determined by the matrices $\{\sigma_k\}_1^n$, $\{\gamma_{kn}\}_{k=1}^{n-1}$, satisfying the conditions \rf{a10}, \rf{a11}, \rf{12}, and other operators $\gamma_{ks}$, $k,s=1,2,\ldots,n-1$ are defined by the equalities \rf{14} (see \cite{JMAA}).
    The selfadjoint operators $\widetilde{\gamma}_{ks}$ ($k,s=1,2,\ldots,n$), defined by \rf{6}, satisfy analogous relations as $\gamma_{ks}$, i.e.
        $$
        \widetilde{\gamma}_{ks}=\widetilde{\gamma}_{ks}^*=-\widetilde{\gamma}_{sk}, \;\;\;\;\;\;
        \sigma_k\Phi A_s-\sigma_s\Phi A_k=\widetilde{\gamma}_{ks}\Phi,
        $$
        $$
        \sigma_n^{-1}\sigma_k\sigma_n^{-1}\widetilde{\gamma}_{sn}+
        \sigma_n^{-1}\widetilde{\gamma}_{kn}\sigma_n^{-1}\sigma_s
        =
        \sigma_n^{-1}\sigma_s\sigma_n^{-1}\widetilde{\gamma}_{kn}+
        \sigma_n^{-1}\widetilde{\gamma}_{sn}\sigma_n^{-1}\sigma_k,
        $$
        $$
        \sigma_n^{-1}\widetilde{\gamma}_{kn}\sigma_n^{-1}\widetilde{\gamma}_{sn}=
        \sigma_n^{-1}\widetilde{\gamma}_{sn}\sigma_n^{-1}\widetilde{\gamma}_{kn}, \;\;\;\;\;\;
        \widetilde{\gamma}_{ks}=
        \sigma_s\sigma_n^{-1}\widetilde{\gamma}_{kn}-
        \sigma_k\sigma_n^{-1}\widetilde{\gamma}_{sn}.
        $$
        \par
Now in the case when $n\geq 3$  we consider the colligation from the form
        \be\la{14-0}
        X=(A_1,A_2,\ldots,A_n;H,\Phi,E;\sigma_1,\sigma_2,\ldots,\sigma_n,\{\gamma_{kn}\},\{\widetilde{\gamma}_{kn}\},
        k=1,2,\ldots,n-1)
        \ee
instead of the commuting regular colligation \rf{1}.
        \par
Next, if the input $u(x)$ of the generalized open system \rf{7}, corresponding to the commutative regular colligation \rf{1}, satisfies the matrix wave equations \rf{9}, then the output $v(x)$ from \rf{7} satisfies the system (or matrix wave equations)
        \be\la{14-1}
        \sigma_k\left(-i\frac{1}{\varepsilon_s}\frac{\partial v}{\partial x_s}\right)-
        \sigma_s\left(-i\frac{1}{\varepsilon_k}\frac{\partial v}{\partial x_k}\right)+\widetilde{\gamma}_{sk}v=0,
        \ee
    $k,s=1,2,\ldots,n$ (see \cite{JMAA}, Theorem 6.2).
        Let us consider the collective motions
        \be\la{15}
        T(x_1,\ldots,x_n)=e^{i(\varepsilon_1x_1A_1+\cdots+\varepsilon_nx_nA_n)}, \;\;\;\;\;
        {T^*}^{-1}(x_1,\ldots,x_n)=
        e^{i(\overline{\varepsilon}_1x_1A_1^*+\cdots+\overline{\varepsilon}_nx_nA_n^*)},
        \ee
    where $T(x_1,x_2,\ldots,x_n)f$ ($f\in H$) is a solution of the corresponding open system \rf{7} with zero input $u(x)=0$. Then the operator functions
        $$
        V(x_1,x_2,\ldots,x_n)=\Phi T(x_1,x_2,\ldots,x_n), \;\;\;\;\;
        $$
        $$
        \widetilde{V}(x_1,x_2,\ldots,x_n)=\Phi {T^*}^{-1}(x_1,x_2,\ldots,x_n),
        $$
        satisfy the following systems of partial differential equations
        \be\la{17}
        \left(
        \sigma_n\left(-i\frac{1}{\varepsilon_k}\frac{\partial}{\partial x_k}\right)-
        \sigma_k\left(-i\frac{1}{\varepsilon_n}\frac{\partial}{\partial x_n}\right)+\widetilde{\gamma}_{kn}
        \right)
        V(x_1,x_2,\ldots,x_n)=0,
        \ee
        \be\la{18}
        \left(
        \sigma_n\left(-i\frac{1}{\overline{\varepsilon}_k}\frac{\partial}{\partial x_k}\right)-
        \sigma_k\left(-i\frac{1}{\overline{\varepsilon}_n}\frac{\partial}{\partial x_n}\right)+\gamma_{kn}
        \right)
        \widetilde{V}(x_1,x_2,\ldots,x_n)=0
        \ee
correspondingly, $k=1,2,\ldots,n-1$ (see Theorem 6.3 in \cite{JMAA}).
        \par
Let us consider the case when one of the the operators $A_1, A_2, \ldots,A_n$ is a coupling of dissipative and antidissipative operators with real absolutely continuous spectra (for example, $A_1$) and $\varepsilon_1=1$. Without loss of generality we can suppose that $A_1=B$, where $B$ is the triangular model of this coupling (introduced by G.S. Borisova in \cite{BAN2} and investigated by G.S. Borisova and K.P. Kirchev in \cite{IEOT1}, \cite{SERDICA3}):
        \be\la{19}
        \ba{c}
        Bf(w)=\alpha(w)f(w)-i\int\limits_{a'}^wf(\xi)\Pi(\xi)S^*\Pi^*(w)d\xi
        +\\+
        i\int\limits_w^{b'}f(\xi)\Pi(\xi)S\Pi^*(w)d\xi+
        i\int\limits_{a'}^wf(\xi)\Pi(\xi)L\Pi^*(w)d\xi,
        \ea
        \ee
    where $f=(f_1,f_2,\ldots,f_p)\in H={\bf L}^2(\Delta;\mathbb{C}^p)$, $\Delta=[a',b']$,
    $L:\mathbb{C}^m \longrightarrow\mathbb{C}^m$,
    $\det L\neq 0$,
    $L^*=L$,
    $L=J_1-J_2+S+S^*$,
            \be\la{20}
    J_{1}=\left(
    \begin{array}{cc}
    I_{r} & 0 \\ 0 & 0 \\
    \end{array} \right) , \;
            J_{2}=\left(
            \begin{array}{cc}
            0 & 0 \\ 0 & I_{m-r} \\
            \end{array} \right) , \;
                     S=\left(
                     \begin{array}{cc}
                     0 & 0 \\ \widehat {S} & 0 \\
                    \end{array} \right),
            \ee
    $r$ is the number of positive eigenvalues and $m-r$ is the
    number of negative  eigenvalues of the matrix $L$, $\Pi(w)$
    is a measurable $p\times m$ ($1\leq p\leq m$) matrix function
    on $\Delta$, whose rows are linearly independent at each point
    of a set of positive measure, the matrix function $\widetilde{\Pi}(w)=\Pi^*(w)\Pi(w)$ satisfies the conditions
        $tr\; \widetilde{\Pi}(w)=1$,  $\widetilde{\Pi}(w)J_1=J_1\widetilde{\Pi}(w)$,
    $||\widetilde{\Pi}(w_1)-\widetilde{\Pi}(w_2)||\leq C|w_1-w_2|^{\alpha_1}$ for all
    $w_1,w_2\in \Delta$ for some constant $C>0$, $\alpha_1$  is an appropriate constant with $0<\alpha_1\leq 1$ (see \cite{IEOT1}), (where $||\;\;||$ is the norm in $\mathbb{C}^m$)
    and the function $\alpha:\Delta\longrightarrow \mathbb{R}$ satisfies
    the conditions:
        \par
        \textbf{(i)}
        the function $\alpha(w)$ is continuous strictly increasing on $\Delta$;
        \par
        \textbf{(ii)}
        the inverse function $\sigma(u)$ of $\alpha(w)$ is
    absolutely continuous on $[a,b]$ ($a=\alpha(a')$,
    $b=\alpha(b')$);
        \par
        \textbf{(iii)}
        $\sigma'(u)$ is continuous and satisfies the relation
        $|\sigma'(u_1)-\sigma'(u_2)|\leq C|u_1-u_2|^{\alpha_2}$,
        $(0<\alpha_2\leq 1)$
    for all $u_1$, $u_2\in [a,b]$ and for some constant $C>0$.
        \par
    The imaginary part of the operator $B$ from \rf{19} satisfies the condition
        $(B-B^*)/i=\Phi^*L\Phi$,
    where the  operator $\Phi:H\longrightarrow H$ is defined by the equality
        \be\la{20-2}
        \Phi f(w)=\int\limits_{a'}^{b'}f(w)\Pi(w)dw.
        \ee
        \par
The existence of the wave operators
        $
        W_{\pm}(B^*,B)=s-\lim\limits_{x\to\pm\infty}e^{ixB^*}e^{-ixB}
        $
of the couple of operators $(B,B^*)$ as strong limits has been established and their explicit form has been obtained in \cite{IEOT1} and \cite{SERDICA2}, i.e.
        \be\la{21}
        W_{\pm}(B^*,B)=s-\lim\limits_{x\to\pm\infty}e^{ixB^*}e^{-ixB}=\widetilde{S}_{\mp}^*\widetilde{S}_{\mp}.
        \ee
The explicit form of the operators $\widetilde{S}_{\mp}$ on the right hand side of the relation \rf{21} for the operator $B$ with triangular model \rf{19} has been obtained in \cite{IEOT1} in the terms of the multiplicative integrals and the finite dimensional analogue of the classical gamma function (introduced in \cite{IEOT1}) and presented by \rf{21-2}.
        \par
To avoid complications of writing we consider the case when $\alpha(w)=w$, i.e. the operator $B$ has the form
        \be\la{21-1}
        \ba{c}
        Bf(w)=wf(w)-i\int\limits_{a'}^wf(\xi)\Pi(\xi)S^*\Pi^*(w)d\xi
        +\\+
        i\int\limits_w^{b'}f(\xi)\Pi(\xi)S\Pi^*(w)d\xi+
        i\int\limits_{a'}^wf(\xi)\Pi(\xi)L\Pi^*(w)d\xi,
        \ea
        \ee
 Then the operators $\widetilde{S}_{\mp}$ take the form (see, for example, \cite{IEOT1})
        \be\la{21-2}
        \widetilde{S}_{\pm}f(w)=(\widehat{S}_{\pm}f(w))T_{\pm},
        \;\;\;\;
        \widehat{S}_{\pm}f=\widetilde{S}_{11}f+\widetilde{S}_{22}f
        +\widetilde{S}_{12}^{\pm}f,
        \ee
        $$
        S_{\pm}f(w)=(\widehat{S}_{\pm}f(w))T_{\pm}\Pi(w)
        (J_1|t|^{i\widetilde{\Pi}_1(w)}J_1+J_2|t|^{-i\widetilde{\Pi}_2(w)}J_2)Q(w),
        $$
        $$
        \ba{c}
        T_{\pm}h=h(J_1U_{2a'}(w)w^{i\widetilde{\Pi}_1(w)}e^{\mp\frac{\pi}{2}\widetilde{\Pi}_1(w)}
        \mathbf{\Gamma}^{-1}(I+i\widetilde{\Pi}_1(w))J_1
        + \\ +
        J_2\widetilde{U}_{2a'}(w)w^{-i\widetilde{\Pi}_2(w)}e^{\pm\frac{\pi}{2}\widetilde{\Pi}_2(w)}
        \mathbf{\Gamma}^{-1}(I-i\widetilde{\Pi}_2(w))J_2)\Pi^*(w), \; \; \;
        (\forall h\in \mathbb{C}^m),
        \ea
        $$
        $$
        \widetilde{S}_{kk}f(w)=\int\limits_{a'}^x\widetilde{f}'(\xi)\MIIL{w}{a'}
        e^{\frac{(-1)^{k+1}i\widetilde{\Pi}_k(v)}{v-\xi}dv}d\xi J_k,  \; \;
        \widetilde{S}_{12}^{\pm}f(w)=-\int\limits_{a'}^{b'}\widetilde{f}'(\xi)
        \widetilde{F}_{\xi}^{\mp}(w,b')d\xi S,
        $$
        $$
        U_{2\xi}(w)=\lim\limits_{\delta\to 0}
        \MII{w-\delta}{\xi}e^{\frac{-i\widetilde{\Pi}_1(v)}{v-w}dv}
        e^{i\int\limits_{\xi}^{w-\delta}\frac{\widetilde{\Pi}_1(w)}{v-w}dv},\;\;
        \widetilde{U}_{2\xi}(w)=\lim\limits_{\delta\to 0}
        \MII{w-\delta}{\xi}e^{\frac{i\widetilde{\Pi}_2(v)}{v-w}dv}
        e^{-i\int\limits_{\xi}^{w-\delta}\frac{\widetilde{\Pi}_2(w)}{v-w}dv},
        $$
        $$
        \mathbf{\Gamma}(\varepsilon I-iT(u))=
        \int\limits_0^{\infty}e^{-x}
        e^{((\varepsilon-1)I-iT(u))\ln x}dx \; \; \;
        (\varepsilon>0).
        $$
        $$
        \widetilde{\Pi}_k(w)=J_k\widetilde{\Pi}(w)J_k=J_k\Pi^*(w)\Pi(w)J_k, \;\;\;\;\;
        \widetilde{f}(w)=f(w)Q^*(w),
        $$
        $k=1,2$. In the last equality, $m\times p$ matrix function $Q(w)$ is smooth on $\Delta$ and satisfies the condition
        $\Pi(w)Q(w)=I$.
        \par
        The existence  and the explicit form of the limits \rf{21} in the case of a coupling $A_1=B$ allow to introduce an appropriate scalar product in the space of solutions of the equations \rf{14-1} from the form
        $v_h(x_1,x_2,\ldots,x_n)=\Phi e^{i(\varepsilon_1 x_1A_1+\cdots+\varepsilon_n x_nA_n)}h$.
        \par
        \par
        Let $\widehat{H}$ be the principal subspace of the colligation $X$ from the form \rf{1} and $A_1=B$, where $B$ is the triangular model \rf{21-1}, i.e.
        \be\la{22}
        \widehat{H}=\overline{\textit{span}}\;\{A_1^{m_1}A_2^{m_2}\ldots A_n^{m_n}\Phi^*E, m_1,m_2,\ldots,m_n\in \mathbb{N}\cup\{0\}\}.
        \ee
Let $\widetilde{H}$ be the set of solutions
        \be\la{23}
        v_h(x_1,x_2,\ldots,x_n)=\Phi T(x_1,x_2,\ldots,x_n)h=\Phi e^{i(\varepsilon_1 x_1A_1+\cdots+\varepsilon_n x_nA_n)}h,  \;\;\; h\in \widehat{H}
        \ee
of the system \rf{14-1}. Let the operator
        $
        U:\widehat{H}\longrightarrow\widetilde{H}
        $
        be defined by the equality
        \be\la{24}
        Uh=\Phi e^{i(\varepsilon_1 x_1A_1+\cdots+\varepsilon_n x_nA_n)}h=v_h(x_1,x_2,\ldots,x_n), \;\;\; h\in \widehat{H}.
        \ee
Let the commuting nonselfadjoint operators $A_1,A_2,\ldots,A_n$ with $A_1=B$ and $B$ from the form \rf{21-1} be embedded in the colligation
        $$
        \ba{c}
        X=(A_1=B,A_2,\ldots,A_n;H=\mathbf{L}^2(\Delta,\mathbb{C}^p),\Phi,E=\mathbb{C}^m;
        \\
        \sigma_1,\ldots,\sigma_n,\{\gamma_{ks}\},
        \{\widetilde{\gamma}_{ks}\},k,s=1,2,\ldots,n).
        \ea
        $$
        \begin{thm}\la{t1} (see \cite{BAN7})
        The equality
        \be\la{25-1}
        \ba{c}
        \left\langle v_{h_1}(x_1,\ldots,x_n),v_{h_2}(x_1,\ldots,x_n)\right\rangle=
        \lim\limits_{x_1\to+\infty}(e^{ix_1A_1}h_1,e^{ix_1A_1}h_2)
        + \\ +
        \int\limits_0^{\infty}(\sigma_1 v_{h_1}(x_1,0,\ldots,0),v_{h_2}(x_1,0,\ldots,0)dx_1
        = \\ =
        (\widetilde{S}_+^*\widetilde{S}_+h_1,h_2)+\int\limits_0^{\infty}(\sigma_1 v_{h_1}(x_1,0,\ldots,0),v_{h_2}(x_1,0,\ldots,0))dx_1,
        \;\;\;\; h_1,h_2\in H
        \ea
        \ee
        defines a scalar product in the subspace $\widetilde{H}$ of solutions
        $$
        v_h(x_1,\ldots,x_n)=\Phi e^{i(\varepsilon_1 x_1A_1+\cdots+\varepsilon_n x_nA_n)}h,
        $$
        $h\in\widehat{H}$ 
        (with $\varepsilon_1=1$, $A_1=B$) of the equations \rf{14-1}.
        \par
        \end{thm}
        \par
        The equalities \rf{25-1} imply also that the operator $U$, defined by \rf{24}, is an isometric one and
        $$
        \left\langle v_{h_1},v_{h_2}\right\rangle=\left\langle Uh_1,Uh_2\right\rangle=(h_1,h_2), \;\;\;\; h_1,h_2\in\widehat{H}.
        $$
    It has to mention that the operator $U:\widehat{H}\longrightarrow\widetilde{H}$, defined by the equality \rf{24}, is also unitary.
       \par
        \par
        The case of two commuting nonselfadjoint operators $(A_1,A_2)$ where $A_1$ is a dissipative operator with zero limit
        $\lim\limits_{x_1\to\infty}(e^{ix_1A_1}h,e^{ix_1A_1}h)=0$ ($h\in H$, $\varepsilon_1=\varepsilon_2=1$),
considered by M.S. Liv\v sic in \cite{LIVSIC-MAPPING}, and the case of $n$ commuting nonselfadjoint operators $(A_1,\ldots,A_n)$, where $A_1$ is a dissipative operator with nonzero limit
        $\lim\limits_{x_1\to\infty}(e^{ix_1A_1}h,e^{ix_1A_1}h)\neq 0$ ($h\in H$, $\varepsilon_1=\cdots=\varepsilon_n=1$),
considered by G.S. Borisova and K.P. Kirchev in \cite{SERDICA1}, show that
$v_0(x_n)$ determines uniquely the output $v(x_1,\dots,x_n)$ by the equations \rf{14-1} ($s=1,k=1,2,\ldots,n-1)$ and the condition $v(0,\ldots,0,x_n)=v_0(x_n)$ in the region of an existence and uniqueness of the solutions (see \cite{MYSOHATA}).
        \par
Following the terminology by M.S. Liv\v sic in \cite{LIVSIC-MAPPING} the functions $v_h(x_1,\ldots,x_n)$ and $v_h(x_1,0,\ldots,0)$ are said to be the output representation and the mode of the element $h\in\widetilde{H}$ correspondingly.
        \par
The next theorem solves a similar problem for the output $v_h(x_1,\ldots,x_n)$ and the mode $v_h(0,\ldots,0,x_n)$ in the case of $n$ operators ($n\geq 3$) with nonzero constants $\varepsilon_1,\ldots,\varepsilon_n$, when $A_1=B$ is a coupling of dissipative and antidissipative operators with real absolutely continuous spectra, which ensures the existence of the limit
        $\lim\limits_{x\to+\infty}(e^{ixB}f,e^{ixB}f)$,
obtained explicitly in \cite{IEOT1}. In this case we essentialy use the conditions of V.A. Zolotarev \cite{ZOLOTAREV-TIME CONES}.
        \par
We consider now the boundary value problem for solutions of the partial differential equations
        \be\la{31-1}
        \left\{
        \ba{l}
        \sigma_n\left(-i\frac{1}{\varepsilon_k}\frac{\partial v}{\partial x_k}\right)-
        \sigma_k\left(-i\frac{1}{\varepsilon_n}\frac{\partial v}{\partial x_n}\right)+\widetilde{\gamma}_{kn}v=0, \;\;\;
        k=1,2,\ldots,n-1\\
        v(0,\ldots,0,x_n)=v_0(x_n)
        \ea
        \right.
        \ee
which are restrictions to $\mathbb{R}^n$ of entire functions on $\mathbb{C}^n$. We will denote by $(z_1,\ldots,z_n)$ the coordinates on $\mathbb{C}^n$ and by $(x_1,\ldots,x_n)$ the coordinates on $\mathbb{R}^n$.
%
        \begin{thm}\la{t3} (see \cite{BAN7})
        Let $\sigma_1,\sigma_2,\ldots,\sigma_n$, $\{\widetilde{\gamma}_{kn}\}$  ($k=1,2,\ldots,n-1$) be $m\times m$ hermitian matrices with $\det\sigma_n\neq 0$ and they satisfy the conditions of V.A. Zolotarev \rf{a10}, \rf{a11}, \rf{12}. Then
        \par
        \par
        1)
        if $v(x_1,\ldots,x_n)$ is a solution of
        \be\la{i75}
        \sigma_n\frac{1}{\varepsilon_k}\frac{\partial v}{\partial x_k}-\sigma_k\frac{1}{\varepsilon_n}\frac{\partial v}{\partial x_n}+
        i\widetilde{\gamma}_{kn}v=0, \;\;\;\;k=1,2,\ldots,n-1,
        \ee
        which is a restriction to $\mathbb{R}^n$ of entire function on $\mathbb{C}^n$, and $v(0,\ldots,0,x)=0$ for all $x\in \mathbb{R}$, then $v(x_1,\ldots,x_n)=0$;
        \par
        \par
        2)
        if $\{v_l(x_1,\ldots,x_n)\}$ is a sequence of solutions of the system \rf{i75} which are restrictions to $\mathbb{R}^n$ of entire function on $\mathbb{C}^n$, satisfying the condition
        $v_l(0,\ldots,0,x)\longrightarrow g(x)$ as $l\to \infty$ ($\forall x\in\mathbb{R}$)
        where $g(x)$ is a function on $\mathbb{R}$ which is infinitely differentiable in a neighbourhood of $0$ and there exists a constant $C$ such that
        \be\la{i77}
        \lim\limits_{l\to\infty}
        \left(\left(\frac{\partial v_l}{\partial x_n^k}(0,\ldots,0)-\frac{d^kg}{dx^k}(0)\right)/C^k\right)=0
        \ee
        uniformly according to $k$, then there exists a solution $v(x_1,\ldots,x_n)$ of \rf{i75} which is a restriction to $\mathbb{R}^n$ of an entire function on $\mathbb{C}^n$, such that $v(0,\ldots,0,x)=g(x)$ for all $x\in\mathbb{R}$ and
        $
        v_l(z_1,\ldots,z_n)\longrightarrow v(z_1,\ldots,z_n)
        $
        as $l\to \infty$ uniformly on compact subset on $\mathbb{C}^n$.
        \end{thm}
        \par
        \par
If the matrices $\sigma_1,\ldots,\sigma_n$, $\{\gamma_{kn}\}$ are selfadjoint $m\times m$ matrices, satisfying the conditions of V.A. Zolotarev \rf{a10}, \rf{a11}, \rf{12}, the matrices $\{\gamma_{ks}\}$ ($k,s=1,2,\ldots,n-1$) are defined by the equality \rf{14}, and the matrices
    $\{\widetilde{\gamma}_{ks}\}$ ($k,s=1,2,\ldots,n$)
are defined by \rf{6}, the Theorem \ref{t3} implies that the solution $v(x_1,\ldots,x_n)$ of the system \rf{i75} satisfies the system \rf{14-1}.
        \par
In the case when the selfadjoint operators $\{\sigma_k\}_1^n$ and $\{\gamma_{ks}\}$ satisfy the conditions \rf{a10}, \rf{a11}, \rf{12} and the operators $\{\widetilde{\gamma}_{ks}\}$ are defined by \rf{6}, then the system
        $$
        \left\{
        \ba{l}
        \sigma_k\left(-i\frac{1}{\varepsilon_s}\frac{\partial v}{\partial x_s}\right)-
        \sigma_s\left(-i\frac{1}{\varepsilon_k}\frac{\partial v}{\partial x_k}\right)+\widetilde{\gamma}_{sk}v=0
        \\
        v(0,\ldots,0,x)=g(x), \;\;\; x\in \mathbb{R}
        \ea
        \right.
        $$
    ($k,s=1,2,\ldots,n$), i.e. \rf{14-1} has a unique solution satisfying the condition $v(0,\ldots,0,x)=g(x)$, which is a restriction to $\mathbb{R}^n$ of an entire function on $\mathbb{C}^n$.
        \par
%
%
%
%
        \section{A generalized form of the Gelfand-Levitan-Marchenko equation and the Korteweg-de Vries equation}\la{s3}
        \par
In this section we will consider the connection between the Korteweg - de Vries equation and appropriate couples of commuting nonselfadjoint operators when one of them belongs to the wide class of couplings of dissipative and antidissipative operators with real spectra. We will apply this connection for deriving the generalized form of the Gelfand-Levitan-Marchenko equation of the inverse scattering problem. We will obtain the properties possessed by this integral equation, using the properties of the commuting nonselfadjoint operators.
        \par
At first we consider the case of two commuting bounded nonselfadjoint operators in a Hilbert space $H$ with finite dimensional imaginary parts, i.e $A_1=A$, $A_2=B$. Now the subspace $G=G_A+G_B$ (where $G_A=(A-A^*)H$, $G_B=(B-B^*)H$) is the non-Hermitian subspace of the pair $(A,B)$, $G_A$, $G_B$ are the non-Hermitian subspaces of $A$ and $B$ correspondingly. Let the operators $A$ and $B$ be embedded in a regular colligation
        \be\la{32-3}
        X=(A,B;H,\widetilde{\Phi},E;\sigma_A,\sigma_B,\gamma,\widetilde{\gamma})
        \ee
where
        \be\la{32-4}
        (A-A^*)/i=\widetilde{\Phi}^*\sigma_A\widetilde{\Phi}, \;\;\;\;\; (B-B^*)/i=\widetilde{\Phi}^*\sigma_B\widetilde{\Phi},
        \ee
        \be\la{32-5}
        \sigma_A\Phi B^*-\sigma_B\Phi A^*=\gamma\Phi,
        \ee
        \be\la{32-6}
        \sigma_A\widetilde{\Phi} B-\sigma_B\widetilde{\Phi} A=\widetilde{\gamma}\widetilde{\Phi},
        \ee
        \be\la{32-7}
        \widetilde{\gamma}-\gamma=i(\sigma_A\widetilde{\Phi}\widetilde{\Phi}^*\sigma_B-\sigma_B\widetilde{\Phi}\widetilde{\Phi}^*\sigma_A).
        \ee
Here $\widetilde{\Phi}:H\longrightarrow E$ is bounded linear operator, $\dim E<\infty$.
        \par
        If $\Phi H=E$ and $\ker\sigma_A\cap\ker\sigma_B=\{0\}$, the colligation is called a strict colligation. A colligation is said to be commutative if $AB=BA$. Strict commutative colligations are regular (see \cite{LIVSIC-LECTURENOTES,LIVSIC3}).
        \par
Here we consider the case of $\dim E<+\infty$ too.
        \par
        It has to be mentioned that every pair $(A,B)$ of commuting nonselfadjoint operators on $H$ with finite-dimensional imaginary parts can be always embedded in a commutative regular colligation \rf{32-3} (see, for example, \cite{JMAA}).
%
%
        \par
    To the given commutative regular colligation \rf{32-3} there corresponds the following generalized open system from the form \rf{7} for the operators $A$ and $B$ ($\varepsilon_1=\varepsilon$, $\varepsilon_2=\delta$, $\sigma_1=\sigma_A$, $\sigma_2=\sigma_B$)
        \be\la{32-8}
        \left\{
        \ba{l}
        i\frac{1}{\varepsilon}\frac{\partial}{\partial t}+Af=\widetilde{\Phi}^*\sigma_Au,
        \\
        i\frac{1}{\delta}\frac{\partial}{\partial x}f+Bf=\widetilde{\Phi}^*\sigma_Bu,
         \\
        v=u-i\widetilde{\Phi} f,
        \ea
        \right.
        \ee
where $\varepsilon$, $\delta$ are complex constant, the vector functions $u=u(x,t)$, $v=v(x,t)$ with values in $E$ and $f=f(x,t)$ with values in $H$ are the collective input, the collective output, and the collective state  correspondingly.
        \par
        For the commutative regular colligation $X$ the equations \rf{32-8} of open system are compatible if and only if the input $u=u(x,t)$ and the output $v=v(x,t)$ satisfy the following partial differential equations correspondingly (see Theorem 3.3, \cite{JMAA})
        \be\la{32-9}
        \sigma_B\left(-i \frac{1}{\varepsilon} \frac{\partial u}{\partial t}\right)
        -\sigma_A\left(-i \frac{1}{\delta} \frac{\partial u}{\partial x}\right)+\gamma u=0,
        \ee
        \be\la{32-10}
        \sigma_B\left(-i \frac{1}{\varepsilon} \frac{\partial v}{\partial t}\right)
        -\sigma_A\left(-i \frac{1}{\delta} \frac{\partial v}{\partial x}\right)+\widetilde{\gamma}v=0.
        \ee
The equations \rf{32-9}, \rf{32-10} are matrix wave equations. To the generalized open system \rf{32-8} there correspond the more general collective motions of the form
        \be\la{32-11}
        T(x,t)=e^{i(\varepsilon tA+\delta xB)}, \;\;\;\;
        {T^*}^{-1}(x,t)=e^{i(\overline{\varepsilon} tA^*+\overline{\delta} xB^*)}.
        \ee
It is evident that the vector function (or so-called open field, following the terminology of M.S. Liv\v sic) $f(x,t)=T(x,t)h$ ($h\in H$) satisfies the system \rf{32-8} with identically zero input and an arbitrary initial state $f(0,0)=h$ ($h\in H$).
        \par
        In the paper \cite{JMAA} using the connection between soliton theory and commuting nonselfadjoint operator theory
we obtain new scalar solutions of some nonlinear differential equations---the Schr\"{o}dinger equation, the Heisenberg equation, the Sine-Gordon equation, the Davey-Stewartson equation. These results are based on the generalized open systems and the corresponding matrix wave equations for appropriate pairs and triplets of commuting nonselfadjoint operators when one of them belongs to the larger class of nonselfadjoint nondissipative operators--couplings of dissipative and antidissipative operators with absolutely continuous real spectra (introduced and investigated by G.S. Borisova, K.P. kirchev in \cite{BAN2,IEOT1}). The preliminary results, concerning the application of the connection between the soliton theory and the commuting nonselfadjoint operator theory, are obtained in \cite{JMAA} in the case when one of the operators belongs to the mentioned above class of nondissipative operators. These preliminary results allow to expand the idea for solutions of the KdV equation (obtained by  M.S. Liv\v sic and Y. Avishai in \cite{LIVSIC-AVISHAI} for the dissipative operator $B$ with zero limit
        $
        \lim\limits_{x\to+\infty}(e^{ixB}f,e^{ixB}f)=0
        $
($f\in H$)) in the case of the considered larger class of nondissipative operators.
        \par
Let the operators $A$ and $B$ be commuting linear bounded nonselfadjoint operators in a separable Hilbert space $H$. Let us suppose that $A$ and $B$ satisfy the conditions:
        \par
        \textbf{(I)} the operators $A$ and $B$ have finite-dimensional imaginary parts (i.e. the so-called nonhermitian subspaces $G_A=(A-A^*)H$ and  $G_B=(B-B^*)H$ of the operators $A$ and $B$ are finite-dimensional subspaces);
        \par
        \textbf{(II)} the operator $B$ is a coupling of dissipative and antidissipative operators with absolutely continuous real spectra (and consequently, there exists $\lim\limits_{x\to+\infty}(e^{ixB}h,e^{ixB}h)\neq 0$ ($h\in H$), \cite{IEOT1}).
        \par
    Without loss of generality, we can assume that the operator $B$ is the triangular model \rf{21-1} when $\Delta=[-l,l]$ and $H=\textbf{L}^2(\Delta;\mathbb{C}^p)$. Let the operators $\Pi(x)$, $Q(x)$, $\widetilde{\Pi}(x)$, $L$, $\Phi$ be as in Section \ref{s2}.
        \par
In \cite{JMAA} (Theorem 2.1) it has been obtained that if a bounded linear operator
        $\rho:\textbf{L}^2(\Delta;\mathbb{C}^p)\longrightarrow \textbf{L}^2(\Delta;\mathbb{C}^p)$
commutes with the operator of multiplication with an independent variable in the space $\textbf{L}^2(\Delta;\mathbb{C}^p)$, then the operator $M$, defined in $\textbf{L}^2(\Delta;\mathbb{C}^p)$ by the equality
        \be\la{32-12}
        M=\int\limits_0^{\infty}e^{-ixB^*}\rho\frac{B-B^*}{i}e^{ixB}dx
        \ee
(as a strong limit), satisfies the relation
        \be\la{32-13}
        B^*M-MB=\rho(B^*-B)
        \ee
For the existence of the integral in \rf{32-12} and the equality \rf{32-13} we essentially use the existence and the explicit form of the limit
        $s-\lim\limits_{x\to+\infty}e^{-ixB^*}e^{ixB}=\widetilde{S}_+^*\widetilde{S}_+$,
which follows from \rf{21} and has been obtained in \cite{IEOT1}.
        \par
The next theorem which gives scalar solutions of the nonlinear KdV equation
        $$
        u_t+6u_x^2+u_{xxx}=0.
        $$
using couples of commuting linear bounded operators $(A,B)$
when the nondissipative operator $B$ is a coupling of dissipative and antidissipative operators with real absolutely continuous spectra. It has to be mentioned that analogous problem has been considered by M.S. Liv\v sic and Y. Avishai in \cite{LIVSIC-AVISHAI} in the case of dissipative operator $B$ with zero limit
        $
        s-\lim\limits_{x\to+\infty}e^{-ixB^*}e^{ixB}=0.
        $
        \begin{thm}\la{t4}
        Let the operator $B$ in a Hilbert space $H$ be a coupling of dissipative and antidissipative operators with absolutely continuous real spectra, let $B$  satisfy the condition
        \be\la{32-14}
        B^*=-UBU^* \;\;\;\;(U:H\longrightarrow H,\;\;U^*U=UU^*=I).
        \ee
        If $A=bB^{3}$ ($b\in\mathbb{R}$) and the operator function
        $
        \Gamma(x,t)=T(x,t)+M_1{T^*}^{-1}(x,t)M,
        $
        where $T(x,t)$ is the collective motion of the form \rf{32-11} with $\varepsilon=\delta=1$ and  $M_1$, $M$ are constant linear operators,
        then the solitonic combination
        \be\la{33}
        S(x,t)=\Gamma^{-1}(x,t)\Gamma_x(x,t)
        \ee
        satisfies the KdV equation
        \be\la{33-1}
        v=-2S_x(x,t), \;\;\;\;\; v_t-6vv_x+v_{xxx}=0.
        \ee
        If the operator $M$ has the form \rf{32-12} then
        $P_{G_B}S(x,t)\left|_{G_B}\right.$ is a scalar solution of the KdV equation \rf{33-1}, where $P_{G_B}$ is an orthogonal projector onto the nonhermitian subspace $G_B=(Im\;B)H$ of the operator $B$.
        \end{thm}
        \par
The proof of Theorem \ref{t4} is based on the idea of the proofs of Theorem 4.1, Theorem 5.1, Theorem 6.4 in \cite{JMAA}, where the scalar solutions of the Schr\"{o}dinger equation, the Heisenberg equation, the Sine-Gordon equation, and the Davey-Stewartson are obtained, using the conection between soliton theory and the theory of commuting nonselfadjoint operators.
        \par
Without loss of generality we can assume that the operator $B$ in Theorem \ref{t4} is the triangular model of the form \rf{21-1} in the space $\textbf{L}^2(\Delta;\mathbb{C}^p)$ with $\triangle=[-l,l]$. The operator $A=bB^3$ satisfies the condition
        \be\la{33-2}
        A^*=-UAU^*.
        \ee
which follows from \rf{32-14}. Now we use the well-known representation of the imaginary part of the operator $B$ from \rf{21-1} in terms of channel elements
        \be\la{33-3}
        \frac{B-B^*}{i}f=\sum\limits_{\alpha,\beta=1}^m(f,\Phi^*e_{\alpha})(Le_{\alpha},e_{\beta})\Phi^*e_{\beta}=
        \sum\limits_{\alpha,\beta=1}^m(f,g_{\alpha})(Le_{\alpha},e_{\beta})g_{\beta},
        \ee
where $\{e_{\alpha}\}_1^m$ is a basis in $\mathbb{C}^m$ and  $g_{\alpha}=\Phi^*e_{\alpha}$, $\alpha=1,\ldots,m$, are the channel elements (here $\Phi$ is defined by \rf{20-2} in the case when $\Delta=[-l,l]$). Then we present the imaginary parts of the operators $A$, $B$, $AB^*$ in the following form:
%
%
%
        \be\la{33-4}
        \ba{c}
        \frac{A-A^*}{i}f=b\frac{B^3-B^*}{i}f=b\left(\frac{B^2(B-B^*)}{i}f+\frac{B(B-B^*)B^*}{i}f+\frac{(B-B^*){B^*}^2}{i}f\right)
        = \\ =
        b\left(\sum\limits_{\beta=1}^m(\sum\limits_{\alpha=1}^m(f,g_{\alpha})(Le_{\alpha},e_{\beta})B^2g_{\beta})+
        \sum\limits_{\beta=1}^m(\sum\limits_{\alpha=1}^m(f,Bg_{\alpha})(Le_{\alpha},e_{\beta})Bg_{\beta})\right.
        + \\ +
        \left.\sum\limits_{\beta=1}^m(\sum\limits_{\alpha=1}^m(f,B^2g_{\alpha})(Le_{\alpha},e_{\beta})g_{\beta})\right)
        = \\ =
        ((f,g_1),\ldots,(f,g_m),(f,Bg_1),\ldots,(f,Bg_m),(f,g_m),(f,B^2g_1),\ldots,(f,B^2g_m)).
        \\
        \left(
            \ba{ccc}
            0   &  0   &   bL \\
            0   &  bL  &   0 \\
            bL  &  0  &   0
            \ea
        \right)
        \left(
        \ba{c}
        g_1 \\ \cdots \\g_m \\ Bg_1 \\ \cdots \\ Bg_m \\ B^2g_1 \\ \cdots \\ B^2g_m
        \ea
        \right)
        =
        \widetilde{\Phi}^*\sigma_A\widetilde{\Phi}f,
        \ea
        \ee
        \be\la{33-5}
        \frac{B-B^*}{i}f
        =\Phi^*L\Phi f
        = \widetilde{\Phi}^*\sigma_B\widetilde{\Phi}f,
        \ee
        \be\la{33-6}
        \frac{AB^*-BA^*}{i}f=bB\frac{B^2-{B^*}^2}{i}B^*f
        = \widetilde{\Phi}^*\gamma\widetilde{\Phi}f.
        \ee
Here we have denoted the following $3m\times 3m$ complex matrices $\sigma_A$, $\sigma_B$, $\gamma$ and the operator $\widetilde{\Phi}$:
        \be\la{33-7}
        \ba{c}
        \sigma_A=
        \left(
        \ba{ccc}
        0   &  0   &   bL \\
        0   &  bL  &   0  \\
        bL  &  0   &   0
        \ea
        \right), \;\;\;\;\;
        \sigma_B=
        \left(
        \ba{ccc}
        L  &  0   &   0\\
        0  &  0   &   0\\
        0  &  0   &   0
        \ea
        \right),  \;\;\;\;\;
        \gamma=
        \left(
        \ba{ccc}
        0  &  0   &   0  \\
        0  &  0   &   bL \\
        0  &  bL  &   0
        \ea
        \right),
        \\
        \widetilde{\Phi}=(\Phi \;\;\; \Phi B^* \;\;\; \Phi {B^*}^2 ).
        \ea
        \ee
Now we embed the operators $A=bB^3$ and $B$ in the following commutatative regular two-operator colligation
        \be\la{33-8}
        X=(A=bB^3,B;H=\textbf{L}^2(\Delta;\mathbb{C}^p),\widetilde{\Phi},E=\mathbb{C}^{3m};
        \sigma_A,\sigma_B,\gamma,\widetilde{\gamma}),
        \ee
$\widetilde{\Phi}:H\longrightarrow \mathbb{C}^{3m}$ is determined in \rf{33-7}, i.e.
        $$
        \widetilde{\Phi}f=((f,g_1),\ldots,(f,g_m),(f,Bg_1),\ldots,(f,Bg_m),(f,g_m),(f,B^2g_1),\ldots,(f,B^2g_m)).
        $$
        \par
We consider the corresponding open system \rf{32-8} in the case when $\varepsilon=\delta=1$. The corresponding collective motions \rf{32-11} have the form
        \be\la{33-9}
        T(x,t)=e^{i(btB^3+xB)}, \;\;\;\;\; {T^*}^{-1}(x,t)=e^{i(tA^*+xB^*)}=e^{i(bt{B^*}^3+xB^*)}=U^*T(-x,-t)U,
        \ee
using the conditions \rf{32-14} and \rf{33-2}. In the coure of proving Theorem \ref{t4} we consider the solitonic combination $S(x,t)$ from the form \rf{33},
where we have denoted the operator valued function
        \be\la{34}
        \Gamma(x,t)=T(x,t)+M_1{T^*}^{-1}(x,t)M,
        \ee
and $M_1$ and $M$ are constant bounded linear operators.
From the conditions \rf{32-14}, \rf{33-2}, \rf{33-9} it follows that \rf{34} takes the form
        \be\la{34-1}
        \Gamma(x,t)=T(x,t)+\widetilde{M}_1T(-x,-t)\widetilde{M}  \;\;\;\;\; (\widetilde{M}_1=M_1U, \;\;\; \widetilde{M}=U^*M).
        \ee
%
%
Theorem \ref{t4} shows that the solitonic combination $S(x,t)$ from \rf{33} is a solution of the KdV equation \rf{33-1}. If the operator $M$ has the form \rf{32-12} then the solitonic combination $S(x,t)$ takes the form
        \be\la{36}
        \ba{c}
        S(x,t)=\Gamma^{-1}(x,t)\Gamma_x(x,t)=iB-i\Gamma^{-1}(x,t)\widetilde{M_1}T(-x,-t)(B\widetilde{M}+\widetilde{M}B)
        = \\ =
        iB-i\Gamma^{-1}(x,t)\widetilde{M_1}T(-x,-t)\widetilde{\rho}(B-B^*),
        \ea
        \ee
where we have used the denotation $\widetilde{\rho}=U^*\rho$ and the representation
        \be\la{36-1}
        \ba{c}
        \widetilde{M}B+B\widetilde{M}=U^*MB-U^*B^*UU^*M=U^*(MB-B^*M)
        = \\ =
        U^*\rho(B-B^*)=\widetilde{\rho}(B-B^*).
        \ea
        \ee
        \par
Let us denote the vector function
        \be\la{37}
        h(x,t)=S(x,t)g-iBg, \;\;\;\;\; g\in G_B
        \ee
for $S(x,t)$ from the relation \rf{36}. In the case, when $\widetilde{M}_1=I$ (i.e. $M_1=U^*$) the equality \rf{37} takes the form
        \be\la{38}
        h(x,t)=-T(-2x,-2t)\widetilde{M}h(x,t)-iT(-2x,-2t)\widetilde{\rho}(B-B^*)g.
        \ee
        \par
Let $\widehat{H}$ be the principal subspace of the operator $B$, i.e.
        $
        \widehat{H}=\overline{\textit{span}}\;\{B^p\widetilde{\Phi}^*E, p\in \mathbb{N}\cup\{0\}\}.
        $
It has to be mentioned that in the considered case when $A=bB^3$ the principal subspace \rf{22} of the couple $(A,B)$ from the colligation \rf{33-8} coincides with the principal subspace of the operator $B$. Let us denote the set
        $
        \widetilde{H}=\left\{\widetilde{\Phi}e^{i\xi B}h,h\in\widehat{H}\right\}
        $
and $v_h(\xi)=\widetilde{\Phi} e^{i\xi B}h$. The next theorem determines a scalar product in the space $\widetilde{H}$.
        \begin{thm}\la{t5}
        The relation
        \be\la{38-1}
        \ba{c}
        \langle v_h(\xi),v_g(\xi)\rangle=\lim\limits_{\xi\to\infty}(e^{i\xi B}h,e^{i\xi B}g)+
        \int\limits_0^{\infty}(\sigma_B v_h(\xi),v_g(\xi))d\xi
        = \\ =
        (\widetilde{S}_+h,\widetilde{S}_+g)+\int\limits_0^{\infty}(\sigma_B v_h(\xi),v_g(\xi))d\xi
        \ea
        \ee
defines a scalar product in $\widetilde{H}$ and the operator
        $
        \widetilde{U} :\widehat{H}\longrightarrow\widetilde{H},
        $
which is defined by the equality
        $\widetilde{U} h=\widetilde{\Phi} e^{i\xi B}h=v_h(\xi)$,
is an isometric one.
        \end{thm}
        \par
The proof of Theorem \ref{t5} is analogous to the proof of Theorem \ref{t1} and it follows from the existence and the form of the limit \rf{21}
obtained in \cite{IEOT1} when the operator $B$ is the triangular model \rf{21-1} of couplings of dissipative and antidissipative operators.
        \par
On the other hand the relation \rf{38-1} and the equality
        $$
        \frac{d}{d\xi}\left(e^{i\xi B}h,e^{i\xi B}g\right)=
        -\left(\frac{B-B^*}{i}e^{i\xi B}h,e^{i\xi B}g\right)=
        -(\widetilde{\Phi}^*\sigma_B\widetilde{\Phi}e^{i\xi B}h,e^{i\xi B}g)
        $$
implies also that
        $$
        \langle v_h(\xi),v_g(\xi)\rangle=\langle \widetilde{U}h,\widetilde{U}g\rangle=(\widetilde{U}^*\widetilde{U}h,g)=(h,g),
        $$
i.e. $\widetilde{U}^*\widetilde{U}=I$, i.e. $\widetilde{U}$ is an isometric operator onto $\widehat{H}$.
        \par
        Following M.S. Liv\v sic \cite{LIVSIC-MAPPING} $v_h(\xi)$ is a mode of the element $h$.
        \begin{thm}\la{t6}
Let the operator $B$ in a Hilbert space $H$ be a coupling of dissipative and antidissipative operators with absolutely continuous real spectra. Let $B$ satisfy the condition
        $$
        B^*=-UBU^* \;\;\;\;(U:H\longrightarrow H,\;\;U^*U=UU^*=I),
        $$
and let the operators $A=bB^3$ and $B$ be embedded in the regular colligation
        \be\la{38-2-0}
        X=(A=bB^3,B;H=\textbf{L}^2(\Delta;\mathbb{C}^p),\widetilde{\Phi},E=\mathbb{C}^{3m};
        \sigma_A,\sigma_B,\gamma,\widetilde{\gamma}),
        \ee
where $\widetilde{\Phi}$, $\sigma_A$, $\sigma_B$, $\gamma$ are defined by \rf{33-7}. Then the mode $v_h(\xi)=\widetilde{\Phi}e^{i\xi B}h(x,t)$, where $h$ has the form \rf{37}, satisfies the integral equation
        \be\la{38-2}
        v_{h(x,t)}(\xi)=-\int\limits_0^{\infty}\Psi(-2x+\xi+\eta,-2t)v_h(\eta)d\eta
        +\Psi(-2x+\xi,-2t)v_g(0).
        \ee
where
        \be\la{38-3}
        \Psi(x,t)=\widetilde{\Phi} T(x,t)\widetilde{\rho}\widetilde{\Phi}^*\sigma_B.
        \ee
        \end{thm}
        \begin{proof}
Let us consider $h(x,t)$ from the equality \rf{37}, where $S(x,t)$ is a solitonic combination which is a solution of the KdV equation (see Theorem \ref{t4}). From the form \rf{32-12} of the operator $M$ and the condition \rf{32-14}, satisfied by the operator $B$, it follows that $\widetilde{M}$ takes the form
        \be\la{39}
        \ba{c}
        \widetilde{M}=U^*M=\int\limits_0^{\infty}U^*e^{-i\eta B^*}\rho\frac{B-B^*}{i}e^{i\eta B}d\eta
        =
        \int\limits_0^{\infty}U^*e^{-i\eta B^*}UU^*\rho\frac{B-B^*}{i}e^{i\eta B}d\eta
        = \\=
        \int\limits_0^{\infty}e^{i\eta B}\widetilde{\rho}\frac{B-B^*}{i}e^{i\eta B}d\eta
        =
        \int\limits_0^{\infty}e^{i\eta B}\widetilde{\rho}\widetilde{\Phi}^*\sigma_B\widetilde{\Phi} e^{i\eta B}d\eta,
        \ea
        \ee
where $\widetilde{\rho}=U^*\rho$. Now from \rf{38} and \rf{39} we obtain consecutively
        \be\la{40}
        \ba{c}
        v_{h(x,t)}(\xi)=\widetilde{\Phi} e^{i\xi B}h(x,t)
        =-\widetilde{\Phi} e^{i\xi B}T(-2x,-2t)
        \int\limits_0^{\infty}e^{i\eta B}\widetilde{\rho}\widetilde{\Phi}^*\sigma_B\widetilde{\Phi} e^{i\eta B}h(x,t)d\eta
        + \\ +
        \widetilde{\Phi} e^{i\xi B}T(-2x,-2t)\widetilde{\rho}\widetilde{\Phi}^*\sigma_B\widetilde{\Phi} g
        = \\ =
        -\int\limits_0^{\infty}\widetilde{\Phi} T(-2x+\xi+\eta,-2t)\widetilde{\rho}\widetilde{\Phi}^*\sigma_B\widetilde{\Phi} e^{i\eta B}h(x,t)d\eta
        + \\ +
        \widetilde{\Phi} T(-2x+\xi,-2t)\widetilde{\rho}\widetilde{\Phi}^*\sigma_B\widetilde{\Phi} g
        = \\ =
        -\int\limits_0^{\infty}\widetilde{\Phi} T(-2x+\xi+\eta,-2t)\widetilde{\rho}\widetilde{\Phi}^*\sigma_Bv_h(\eta)d\eta
        +
        \widetilde{\Phi} T(-2x+\xi,-2t)\widetilde{\rho}\widetilde{\Phi}^*\sigma_B v_g(0).
        \ea
        \ee
Consequently using the denotation \rf{38-3} we obtain that $v_h(\xi)$ satisfies the integral equation \rf{38-2}. The proof is complete.
        \end{proof}
        \par
The integral equation \rf{38-2} is the generalized Gelfand-Levitan-Marchenko equation, i.e. $v_h(\xi)$ satisfies a differential equation, $\Psi(x,t)$ can be presented as a characteristic operator function of the operator $A$ (which is the scattering function) and $\Psi(x,t)g$ can be presented as a Fourier transform of the function from the space $\mathbf{L}^2(\Delta;\mathbb{C}^m)$. More precisely we will obtain that the integral equation \rf{38-2} possesses the next properties:
        \par
        \textbf{(a)} the operator function $\Psi(x,t)$ can be presented as a characteristic operator function of the operator $A$ and  $\Psi(x,t)$  can be presented with the help of complete characteristic operator function of the colligation $X$ from \rf{33-8};
        \par
        \textbf{(b)} the vector function $\Psi(x,t)g$ is a Fourier transform of a function from the Hilbert space $\mathbf{L}^2(\Delta;\mathbb{C}^m)$;
        \par
        \textbf{(c)}  the vector function $v_h(\xi)$ is the so-called mode of $h$ which corresponds to the output representation $v_h(\xi,\tau)=\widetilde{\Phi}e^{i(\tau A+\xi B)}h$  of the element $h$ of the open system \rf{32-8} with $\varepsilon=\delta=1$ and $v_h(\xi,\tau)$ satisfies an appropriate matrix wave equation from the form \rf{32-10}, i.e
        \be\la{40-000}
        \sigma_B\left(-i \frac{\partial v}{\partial \tau}\right)
        -\sigma_A\left(-i \frac{\partial v}{\partial \xi}\right)+\widetilde{\gamma}v=0.
        \ee
From Theorem 2.4, \cite{LIVSIC-AT-AL} in the case of this wave equation it follows that the mode $v_h(\xi)$ determines the output $v(\xi,\tau)$ uniquely as a solution of the equation \rf{40-000}, satisfying the condition
        $
        v(\xi,0)=v_h(\xi)
        $
 in the case when the operator $B$ is a coupling of dissipative and antidissipative operators with real spectra which ensures the the existence of the limit
        $
        s-\lim\limits_{x\to+\infty}e^{-ixB^*}e^{ixB}.
        $
        \par
 In particular in Section \ref{s6}, Section \ref{s7}, Section \ref{s8}, Section \ref{s9} we will obtain what kind of differential equations are satisfied by the components of the mode $v_h(\xi)$ in the case when the open systems correspond to the KdV equation and the Schr\"{o}dinger equation.
        \par
 At first we will mention that the isometric operator $\widetilde{U}:\widehat{H}\longrightarrow\widetilde{H}\subset E$ maps $h(x,t)\in H=\textbf{L}^2(\Delta;\mathbb{C}^p)$ in a finite dimensional subspace.
 It also has to mention that  $\Gamma(x,t)$ from \rf{34} (when $M_1=U^*$ and $T(x,t)$, ${T^*}^{-1}(x,t)$ have the form \rf{33-9}) satisfies the equation
        \be\la{40-0}
        \frac{\partial^2}{\partial x^2}\Gamma(x,t)+B^2\Gamma(x,t)=0.
        \ee
        \par
It turns out that the integral equation \rf{38-2} can be presented in terms of the complete characteristic function of the pair $(A,B)$, i.e. there exists a connection between the generalized Gelfand-Levitan-Marchenko equation and the complete characteristic function of the operator colligation $X$.
        \par
Let the operators $A$, $B$, the colligation $X$ be like above stated. To avoid comlications of writing without loss of generality we consider the case when  $\widetilde{\rho}
=I$.
        \begin{thm}\la{t7}
        The matrix function $\Psi(x,t)$, defined by the equality \rf{38-3}, has the representation
        \be\la{40-1}
        \Psi(x,t)\widetilde{g}=\frac{1}{2\pi}\int\limits_{|\mu|=r}e^{i\mu}W(t,x,\mu)d\mu\widehat{g},
        \ee
where
        \be\la{40-2}
        W(t,x,\mu)=I+i\Phi(tA+xB-\mu I)^{-1}\Phi^*(t\sigma_A+x\sigma_B)
        \ee
is the complete characteristic function of the colligation $X$ from \rf{33-8}, $r>a\sqrt{|t|^2+|x|^2}$,
        $\widehat{g}=(x\sigma_B+t\sigma_A)^{-1}\sigma_B\widetilde{g}$, $\widetilde{g}\in G_B$
and the generalized Gelfand-Levitan-Marchenko equation has the form (in the case $\widetilde{\rho}=I$)
        \be\la{45-1}
        \ba{c}
        v_h(\xi)=\widetilde{\Phi}e^{i\xi B}h(x,t)
        =-\frac{1}{2\pi}\int\limits_0^{\infty}
        (\int\limits_{|\mu|=r}e^{i\mu}W(-2t,-2x+\xi+\eta,\mu)d\mu
        . \\ .
        \left.
       ((-2t)\sigma_A+(-2x+\xi+\eta)\sigma_B)^{-1}\sigma_Bv_h(\eta)\right)d\eta
        + \\ +
        \frac{1}{2\pi}\int\limits_{|\mu|=r}e^{i\mu}W(-2t,-2x+\xi,\mu)d\mu((-2t)\sigma_A+(-2x+\xi)\sigma_B)^{-1}\sigma_Bv_g(0).
        \ea
        \ee
        \end{thm}
        \begin{proof}
Let us consider $\Psi(x,t)$ using \rf{38-3} for $\widetilde{g}\in G_B$
        \be\la{46}
        \ba{c}
        \Psi(x,t)\widetilde{g}=\widetilde{\Phi}T(x,t)\widetilde{\Phi}^*\sigma_B\widetilde{g}
        = \\ =
        \widetilde{\Phi}e^{i(tA+xB)}\widetilde{\Phi}^*(x\sigma_B+t\sigma_A)(x\sigma_B+t\sigma_A)^{-1}\sigma_B\widetilde{g}
        = \\ =
        \widetilde{\Phi}e^{i(tA+xB)}\widetilde{\Phi}^*(x\sigma_B+t\sigma_A)\widehat{g}(w)
        = \\ =
        \widetilde{\Phi}\left(-\frac{1}{2\pi i}\int\limits_{|\mu|=r}e^{i\mu}(tA+xB-\mu I)^{-1}d\mu
        \widetilde{\Phi}^*(t\sigma_A+x\sigma_B)\widehat{g}\right)
        = \\ =
        \frac{1}{2\pi}\int\limits_{|\mu|=r}e^{i\mu}(W(t,x,\mu)-I)d\mu\widehat{g}=
        \frac{1}{2\pi}\int\limits_{|\mu|=r}e^{i\mu}W(t,x,\mu)d\mu\widehat{g}
        \ea
        \ee
where $W(t,x,\mu)=I+i\widetilde{\Phi}(tA+xB-\mu I)^{-1}\widetilde{\Phi}^*(x\sigma_B+t\sigma_A)$ is the complete characteristic function of the operator colligation $X$ from \rf{38-2-0} (of the couple $(A,B)$, embedded in the colligation $X$ from \rf{38-2-0}). In the equalities \rf{46} we have used the Dunford integral for the bounded linear operator $tA+xB$. Now the representation \rf{45-1} follows immediately from \rf{46}. The proof is complete.
        \end{proof}
It has to be mentioned that there exists a similar representation of $\Psi(x,t)$ in the case when $\widetilde{\rho}\neq I$ using the equality $\Pi(w)Q(w)=I$, where the matrix function $Q(x)$ is stated as in Section \ref{s2}.
%
%
        \par
The next theorem presents the matrix function $\Psi(x,t)$, defined by the equality \rf{38-3} in other form---with the help of the characterisic operator function of the operator $A$, i.e. in terms of the transfer function (or the scattering function).
        \begin{thm}\la{t7-1}
        The matrix function $\Psi(x,t)$, defined by the equality \rf{38-3}, has the representation
        $$
        \Psi(x,t)\widetilde{g}=\frac{1}{2\pi}\int\limits_{|\mu|=r}e^{it\mu}W_A(\mu)d\mu\widehat{g}(x),
        $$
where
        $
        W_A(\mu)=I+I\widetilde{\Phi}(A-\mu I)^{-1}\widetilde{\Phi}^*\sigma_A
        $
is the characteristic operator function of the operator $A=bB^3$ with
        $
        \widehat{g}(x)=((e^{ixB}\widetilde{\Phi}^*\sigma_B\widetilde{g})Q^*(w)\sigma_A^{-1},\underbrace{0,\ldots0}\limits_{2m})
        \in\mathbb{C}^{3m},
        $
($\widetilde{g}\in G_B$) and the generalized Gelfand-Levitan Marchenko equation (in the case $\widetilde{\rho}=I$) has the form (for $g\in G_B$)
        \be\la{46-1}
        \ba{c}
        v_h(\xi)=\widetilde{\Phi}e^{i\xi B}h(x,t)=
        -\frac{1}{2\pi}\int\limits_0^{\infty}
        \left(
        \int\limits_{|\mu|=r}e^{-2it\mu}W_A(\mu)d\mu
        \right)
        . \\ .
        ((e^{i(-2x+\xi+\eta)B}\widetilde{\Phi}^*\sigma_Bv_h(\eta))
        Q^*(w)\sigma_A^{-1},\underbrace{0,\ldots,0}\limits_{2m})d\eta
        + \\ +
        \frac{1}{2\pi}\int\limits_{|\mu|=r}e^{-2it\mu}W_A(\mu)d\mu
        .
        ((e^{i(-2x+\xi)B}\widetilde{\Phi}^*\sigma_Bv_g(0))
        Q^*(w)\sigma_A^{-1},\underbrace{0,\ldots,0}\limits_{2m})
        \ea
        \ee
        \end{thm}
        \begin{proof}
        Let $\widetilde{g}\in G_B$. We consider
        $$
        \ba{c}
        \widetilde{\Phi}T(x,t)\widetilde{\Phi}^*\sigma_B\widetilde{g}
        =\widetilde{\Phi}e^{itA}e^{ixB}\widetilde{\Phi}^*\sigma_B\widetilde{g}
        = \\ =
        \widetilde{\Phi}
        \left(
        -\frac{1}{2\pi i}\int\limits_{|\mu|=r}e^{it\mu}(A-\mu I)^{-1}Q^*(w)\Pi^*(w)e^{ixB}\widetilde{\Phi}^*\sigma_B\widetilde{g}d\mu
        \right)
        = \\ =
        -\frac{1}{2\pi i}\int\limits_{|\mu|=r}e^{it\mu}\widetilde{\Phi}
        (A-\mu I)^{-1}((e^{ixB}\widetilde{\Phi}^*\sigma_B\widetilde{g})Q^*(w)\Pi^*(w))d\mu
        = \\ =
        -\frac{1}{2\pi i}\int\limits_{|\mu|=r}e^{it\mu}\widetilde{\Phi}
        (A-\mu I)^{-1}\Phi^*((e^{ixB}\widetilde{\Phi}^*\sigma_B\widetilde{g})Q^*(w))d\mu
        = \\ =
        -\frac{1}{2\pi i}\int\limits_{|\mu|=r}e^{it\mu}\widetilde{\Phi}
        (A-\mu I)^{-1}\widetilde{\Phi}^*\sigma_A
        ((e^{ixB}\widetilde{\Phi}^*\sigma_B\widetilde{g})Q^*(w)\sigma_A^{-1},
        \underbrace{0,\ldots,0}\limits_{2m})d\mu
        = \\ =
        \frac{1}{2\pi i}\int\limits_{|\mu|=r}e^{it\mu}W_A(\mu)(\widehat{g},
        \underbrace{0,\ldots,0}\limits_{2m})d\mu.
        \ea
        $$
In the last equalities we have used $m\times p$ matrix function $Q(w)$ which satisfies the conditions from Section \ref{s2} and the equality
        $\Pi(w)Q(w)=I$.
Now the equality \rf{46-1} follows immediately using the representation \rf{38-2} for $v_h(\xi)$. The theorem is proved.
        \end{proof}
%
        \par
        \begin{thm}\la{t8}
        If the vector function $\varphi(\xi)\in\widetilde{H}$ (or $\varphi(\xi)=v_h(\xi)$ for all $h\in\widehat{H}$---the principal  subspace of the couple $(A,B)$ in the case when $A=bB^3$) satisfies the integral equation
        \be\la{52}
        \varphi(\xi)=v_h(\xi)=-\int\limits_0^{\infty}\Psi(-2x+\xi+\eta,-2t)\varphi(\eta)d\eta
        +\Psi(-2x+\xi,-2t)\Phi g, \;\;\; g\in G_B,
        \ee
then the function
        $
        h(x,t)+iBg
        $
is a solution of the nonlinear operator KdV equation
        $$
        u_t+6u_x^2+u_{xxx}=0,
        $$
where
        $
        h(x,t)=(\widetilde{U}^*\varphi)(x,t)
        $
and $\widetilde{U}$ is defined by $\widetilde{U} h=\widetilde{\Phi} e^{i\xi B}h=v_h(\xi)$.
        \end{thm}
        \begin{proof}
After straightforward calculations from \rf{52} we obtain
        \be\la{53}
        \ba{c}
        \varphi(\xi)=v_h(\xi)=-\int\limits_0^{\infty}\widetilde{\Phi} T(-2x+\xi+\eta,-2t)\rho_1\widetilde{\Phi}^*\sigma_Bv_h(\eta)d\eta
        + \\ +
        \widetilde{\Phi} T(-2x+\xi,-2t)\rho_1\widetilde{\Phi}^*\sigma_B v_g(0)
        = \\ =
        -\widetilde{\Phi}e^{i((-2x+\xi)B-2tA)}\int\limits_0^{\infty}e^{i\eta B}\frac{B-B^*}{i}e^{i\eta B}d\eta h
        +
        \widetilde{\Phi}e^{i((-2x+\xi)B-2tA)}\frac{B-B^*}{i}g
        = \\ =
        -\widetilde{\Phi}e^{i((-2x+\xi)B-2tA)}\widetilde{M}h-i\widetilde{\Phi}e^{i((-2x+\xi)B-2tA)}(\widetilde{M}B+B\widetilde{M})g,
        \ea
        \ee
where we have used the representation of the operator $M$ by \rf{32-12} (and $\widetilde{M}=U^*M$, $B^*=-UBU^*$ and the representation
$(B-B^*)/i=-i(\widetilde{M}B+B\widetilde{M})$). Consequently the relations \rf{53} take the form
        \be\la{54}
        \widetilde{\Phi}e^{i\xi B}(h+T(-2x,-2t)\widetilde{M}h+iT(-2x,-2t)(\widetilde{M}B+B\widetilde{M})g)=0.
        \ee
From \rf{54} it follows that
        \be\la{56}
        (T(x,t)+T(-x,-t)\widetilde{M})h=-iT(-x,-t)(\widetilde{M}B+B\widetilde{M})g.
        \ee
Hence
        $$
        h=-i\Gamma^{-1}(x,t)e^{-i(tA+xB)}(B-B^*)g+iBg-iBg=S(x,t)g-iBg
        $$
using the equality \rf{36-1} and the form of $\Gamma(x,t)$, $S(x,t)$ from \rf{34-1}, \rf{36}.
Consequently,
        $
        h(x,t)+iBg=S(x,t)g
        $
is a solution of the nonlinear KdV equation. The theorem is proved.
        \end{proof}
        \par
From Theorem \ref{t3} in the case $n=2$ and $\varepsilon_1=\varepsilon_2=1$ it follows that the vector function $\varphi(\xi)=v_h(\xi)$, satisfying the equality \rf{52}, determines uniquely the solution $v(\xi,\tau)$ of the matrix wave equation
        $$
        \sigma_B\left(-i\frac{\partial v}{\partial\tau}\right)-\sigma_A\left(-i\frac{\partial v}{\partial\xi}\right)+\widetilde{\gamma}v=0,
        $$
satisfying the condition $v(\xi,0)=\varphi(\xi)$ and corresponding to the open system \rf{32-8} with $\varepsilon=\delta=1$, generated by the colligation \rf{38-2-0} or from Proposition 10.4.7 in \cite{LIVSIC-AT-AL}.
        \par
Now we present an interesting connection with results of the paper \cite{BAN6} concerning solutions of the Sturm-liouville systems in the special case of the operator $B$ and the results, obined in this section.
        \par
        Let us consider the case of the operator $B$ from the form \rf{21-1} when $\Delta=[-l,l]$ with zero spectrum, i.e.
        \be\la{57-1}
        \ba{c}
        Bf(w)=-i\int\limits_{-l}^wf(\xi)\Pi(\xi)S^*\Pi^*(w)d\xi
        +\\+
        i\int\limits_w^{l}f(\xi)\Pi(\xi)S\Pi^*(w)d\xi+
        i\int\limits_{-l}^wf(\xi)\Pi(\xi)L\Pi^*(w)d\xi.
        \ea
        \ee
Let the matrix function $\Pi(w)$ have the form
        $
        \Pi(w)=V_1(w)
        $
where $V_1(w)$ is $r\times m$ matrix function from the block representation
        $
        V(w)=
        \left(
        \ba{c}
        V_1(w) \\  V_2(w)
        \ea
        \right)
        $
of the matrix function $V(w)$ from the form
        $
        V(w)=\widehat{V}(w)\mathcal{H}^*,
        $
where $m\times m$ matrix function $\widehat{V}(w)$ satisfies the equation
        \be\la{57-3}
        \left\{
        \ba{l}
        \frac{d\widehat{V}(w)}{dw}=P(w)\widehat{V}(w), \;\;\;\;\; (-l\leq w\leq l)
        \\
        \widehat{V}(-l)=I_m,
        \ea
        \right.
        \ee
with
        $$
        P(w)=
        \left(
        \ba{cc}
        0        &   i\widetilde{G}  \\
        -iRq(w)  &   0
        \ea
        \right),
        $$
$R=\widehat{S}\widehat{S}^*\widehat{S}$  ($\widehat{S}$ is from the representation \rf{20} of $L$, $\widetilde{G}$ is $r\times (m-r)$ matrix satisfying the relation $\widetilde{G}R=I_r$ (as multiplying by matrices), the matrix $\mathcal{H}$ is defined by the equality
        $$
        \mathcal{H}=
        \left(
        \ba{cc}
        \frac{1+i}{2}\widehat{S}^*\widehat{S}       &  \widehat{S}^* \\
        \frac{1+i}{2}U^*\widehat{S}^*\widehat{S}    &  -iU^*\widehat{S}^*
        \ea
        \right),
        $$
($\widehat{U}$ has the form $\widehat{U}=X(I_r \;\;0)Y$, $X$, $Y$ are unitary matrices, i.e. $XX^*=X^*X=I_r$, $YY^*=Y^*Y=I_{m-r}$) and $q(w)$ is a selfadjoint integrable $r\times r$ matrix function on $\Delta$.
        \par
        In terms of the multiplicative integrals the solution of the equation \rf{57-3} has the form
        $$
        \widehat{V}(w)=\MIIL{w}{-l}e^{P(\eta)d\eta}.
        $$
        \begin{thm}\la{t9}
        Let the operator $B$ have the form \rf{57-1}, let the matrix function $\Pi(w)=V_1(w)$
        and $V(w)$, $P(w)$, $q(w)$ be like above stated. Then the vector function
        $y(w,\lambda)=(I+\lambda B)^{-1}\widetilde{\Phi}^*\sigma_Bv_{h(x,t)}(\eta)$
satisfies the Sturm-Liouville system
        \be\la{57-5}
        -\frac{d^2}{dw^2}y(w,\lambda)+q(w)y(w,\lambda)=\lambda y(w,\lambda)
        \ee
with selfadjoint integrable $r\times r$ potential $q(w)$.
        \end{thm}
        \par
The proof of this theorem follows from the Theorem 1 in \cite{BAN6}.
        \par
It has to be mentioned that the solution $y(w,\lambda)$ of the Sturm-Liouville system can be written in the form
         $y(w,\lambda)=(I+\lambda B)^{-1}\frac{B-B^*}{i}e^{i\eta B}h(x,t)$.
Now we present the equality \rf{40} in the form
        \be\la{57-7}
        \ba{c}
        v_h(\xi)=\widetilde{\Phi} e^{i\xi B}h(x,t)
        =-\int\limits_0^{\infty}\widetilde{\Phi} e^{i\xi B}T(-2x+\eta,-2t)(I+\lambda B)y_h(w,\lambda,\eta)d\eta
        + \\ +
        \widetilde{\Phi} e^{i\xi B}T(-2x,-2t)(I+\lambda B)y_g(w,\lambda,0),
        \ea
        \ee
where we have denoted
        $
        y_h(w,\lambda,\eta)=(I+\lambda B)^{-1}\widetilde{\Phi}^*\sigma_Bv_h(\eta).
        $
        \par
If the representation \rf{57-7} of $v_h(\xi)$ (which belongs to the finite dimensional subspace $\mathbb{C}^{3m}$) we have that
        $\widetilde{\Phi}e^{i\xi B}f(w)$
is a Fourier transform (as a function of $\xi$) of the function from the space $\mathbf{L}^2(-\infty,+\infty;\mathbb{C}^{3m})$ (which follows from the course of proving of Theotem 6 in  \cite{IEOT1}) and $y_h(w,\lambda,\eta)$ satisfies the Sturm-Liouville system \rf{57-5} with a selfadjoint potential $q(w)$.
%
%
%
%
        \section{A generalized form of the Gelfand-Levitan-Marchenko equation and the Schr\"{o}dinger equation}\la{s4}
        \par
In this section we consider analogous problems as in Section \ref{s3} which are connected with the Schr\"{o}dinger equation. We will show that a vector function connected with the solitonic combination of the Schr\"{o}dinger equation, generated by an appropriate pair of commuting operators and  obtained in \cite{JMAA}, satisfies appropriate integral equation which is a generalized Gelfand-Levitan-Marchenko equation.
        \par
Let the operators $A$, $B$ be commuting linear bounded nonselfadjoint nondissipative operators in a separable Hilbert space $H$. Let us suppose that $A$ and $B$ satisfy the conditions {\textbf{(I)}}, {\textbf{(II)}}. Let the operator $B$ be a coupling of dissipative and antidissipative operators with real absolutely continuous spectra. Without loss of generality we can assume that the coupling $B$ is the triangular model \rf{21-1} when $\Delta=[-l,l]$, $H=\mathbf{L}^2(\Delta,\mathbb{C}^p)$. Let the operators $\Pi(w)$, $Q(w)$, $\widetilde{\Pi}(w)$, $L$, $\Phi$ be stated as in Section \ref{s2}.
        \par
Now we consider the open system \rf{32-8} in the case when $\varepsilon=i$, $\delta=1$
        \be\la{59}
        \left\{
        \ba{l}
        \frac{\partial f}{\partial t}+Af=\widetilde{\Phi}^*\sigma_Au
        \\
        i\frac{\partial f}{\partial x}+Bf=\widetilde{\Phi}^*\sigma_Bu
        \\
        v=u-i\widetilde{\Phi} f.
        \ea
        \right.
        \ee
Then the collective motions \rf{32-11} have the form
        \be\la{58}
        T(x,t)=e^{i(itA+xB)}, \;\;\;\;\; {T^*}^{-1}(x,t)=e^{i(-itA^*+xB^*)}.
        \ee
        \par
The next theorem gives the representation of the integral equation, which is satisfied by the solutions of the Schr\"{o}dinger equation using the connection between the soliton theory and commuting nonselfadjoint operators, obtained in \cite{JMAA}, Theorem 4.1. At first, we will remind some necessary results.
        \par
        Theorem 4.1 in \cite{JMAA} shows that the solitonic combination
        \be\la{61}
        S(x,t)=\Gamma^{-1}(x,t)\Gamma_x(x,t).
        \ee
where
        \be\la{60}
        \Gamma(x,t)=T(x,t)N+M_1{T^*}^{-1}(x,t)M,
        \ee
($N$, $M$, $M_1$ are constant bounded linear operators in $H$, $BN=NB$) satisfies the nonlinear Schr\"{o}dinger equation
        \be\la{67}
        \left\{
        \ba{l}
        u=[S,\widetilde{B}]=S\widetilde{B}-\widetilde{B}S
        \\
        -\frac{2}{b}\frac{\partial}{\partial t}u\widetilde{B}+\frac{\partial^2}{\partial x^2}u+2u^3=0,
        \ea
        \right.
        \ee
In the course of proving of Theorem 4.1 in \cite{JMAA} we have obtained that $S(x,t)$ from \rf{61} can be written in the form
        \be\la{62}
        S(x,t)=i\Gamma^{-1}(x,t)M_1{T^*}^{-1}(x,t)(B^*M-MB)+iB,
        \ee
where the operator $M$ is defined by the equality \rf{32-12} and
        $$(B^*M-MB)\left|_{G_B^{\bot}}=\rho(B^*-B)\right|_{G_B^{\bot}}=0.$$
From the additional condition
        \be\la{62-1}
        B^*=-UBU^*
        \ee
(where the operator $U:H\longrightarrow H$, $U^*U=UU^*=I$) and $A=bB^2$ ($b\in\mathbb{R}$), we have obtained that

        \be\la{63}
        \frac{A-A^*}{i}f
        = \widetilde{\Phi}^*\sigma_A\widetilde{\Phi}f, \;\;\;\;\;
%
%
        \frac{B-B^*}{i}f
        = \widetilde{\Phi}^*\sigma_B\widetilde{\Phi}f, \;\;\;\;\;
        \frac{AB^*-BA^*}{i}f=
        \widetilde{\Phi}^*\gamma\widetilde{\Phi}f.
        \ee

where
        \be\la{66}
        \sigma_A=
        \left(
        \ba{cc}
        0   &  bL \\
        bL  &  0
        \ea
        \right), \;\;\;\;\;
        \sigma_B=
        \left(
        \ba{cc}
        L  &  0 \\
        0  &  0
        \ea
        \right),  \;\;\;\;\;
        \gamma=
        \left(
        \ba{cc}
        0  &  0 \\
        0  &  bL
        \ea
        \right), \;\;\;\;\;
        \widetilde{\Phi}=(\Phi \;\;\;\; \Phi B^*).
        \ee
        \par
This implies that operators $A=bB^2$  and $B$ can be embedded in the commutative regular colligation $X$ from the form
        \be\la{66-1}
        X=(A=bB^2,B;H=\textbf{L}^2(\Delta;\mathbb{C}^p),\widetilde{\Phi},E=\mathbb{C}^{2m};
        \sigma_A,\sigma_B,\gamma,\widetilde{\gamma})
        \ee
where $\sigma_A$, $\sigma_B$, $\gamma$ are defined by the equalities \rf{66} and the matrix   $\widetilde{\gamma}$ is defined by \rf{32-7}.
        \par
Hence using the condition \rf{62-1} and the form of the operators $A=bB^2$ and $M$, we present $\Gamma(x,t)$ in the form
        \be\la{66-2}
        \Gamma(x,t)=T(x,t)N+\widetilde{M}_1T(-x,-t)\widetilde{M},
        \ee
where we have denoted $\widetilde{M}_1=M_1U$, $\widetilde{M}=U^*M$.
        If the operators $N$, $M$, $\widetilde{B}$ in $H$ satisfy the conditions
        \be\la{66-3}
        NB=BN, \;\;\; N\widetilde{B}=N,\;\;\;M\widetilde{B}=-M,\;\;\; \widetilde{B}^2=I,
        \ee
the Theorem 4.1 in \cite{JMAA} shows that the solitonic combinations $S(x,t)$ from the form \rf{61} satisfies the nonlinear Schr\"{o}dinger equation \rf{67}.
It is easy to see that
        \be\la{68}
        N=\widehat{N}\frac{1}{2}(I+\widetilde{B})=\widehat{N}P_1, \;\;\;\;\; M=U^*\widetilde{M}=U^*\widehat{M}\frac{1}{2}(I-\widetilde{B})=U^*\widehat{M}P_2
        \ee
($P_1$, $P_2$ - are orthoprojectors in $H$). If $P_{G_B}$ is an orthogonal projector onto the nonhermitian subspace $G_B=(Im B)H$ of the operator $B$, then $P_{G_B}S(x,t)\left|_{G_B}\right.$ is a scalar solution of the nonlinear Schr\"{o}dinger equation \rf{67}.
        \par
        Direct calculations show that the solitonic combination \rf{61} takes the form
        \be\la{70}
        \ba{c}
        S(x,t)g=iBg+\Gamma^{-1}(x,t)\widetilde{M}_1T(-x,-t)\widetilde{\rho}\widetilde{\Phi}^*\sigma_B\widetilde{\Phi}g
        = \\ =
        iBg-i\Gamma^{-1}(x,t)\widetilde{M}_1T(-x,-t)\widetilde{\rho}(B-B^*)g,
        \ea
        \ee
 where $\widetilde{\rho}=U^*\rho$ and $g\in G_B$.
        \begin{thm}\la{t10}
Let the operator $B$ in a Hilbert space $H$ be a coupling \rf{21-1} of dissipative and antidissipative operators with absolutely continuous real spectra. Let $B$ satisfy the condition
        $B^*=-UBU^*$, ($U:H\longrightarrow H$, $U^*U=UU^*=I$)
and let the operators $A=bB^2$ ($b\in\mathbb{R}$) and $B$ be embedded in the colligation
        \be\la{81-1}
                X=(A=bB^2,B;H=\textbf{L}^2(\Delta;\mathbb{C}^p),\widetilde{\Phi},E=\mathbb{C}^{2m};
        \sigma_A,\sigma_B,\gamma,\widetilde{\gamma}),
        \ee
where $\sigma_A$, $\sigma_B$, $\gamma$ are defined by the equalities \rf{66} and the matrix   $\widetilde{\gamma}$ is defined by \rf{32-7}.
Then the mode
        $
        v_h(\xi)=v_h(\xi,0)=\widetilde{\Phi}e^{i\xi B}h(x,t)
        $
of the
element $h(x,t)=S(x,t)g-iBg$ ($g\in G_B$, $S(x,t)$ is the solution \rf{61} (with conditions \rf{66-3} for operators $N$, $M$, $\widetilde{B}$ in $H$) of the nonlinear operator Schr\"{o}dinger equation \rf{67}) satisfies the integral equation
        \be\la{82}
        v_{P_1h}(\xi)=-\int\limits_0^{\infty}\Psi(-2x+\xi+\eta,-2t)v_{h}(\eta)d\eta+
        \Psi(-2x+\xi,-2t)v_{g}(0),
        \ee
where the operator $\Psi(x,t)$ is defined by the equality
        \be\la{81}
        \Psi(x,t)=\widetilde{\Phi}T(x,t)\widetilde{\rho}\widetilde{\Phi}^*\sigma_B
        \ee
and $P_1=(I+\widetilde{B})/2$ is the orthoprojector from \rf{68}.
        \end{thm}
        \begin{proof}
Let $S(x,t)$ be the solitonic combination \rf{61}. From \rf{70} it follows that $h(x,t)=S(x,t)g-iBg$ for an arbitrary $g\in G_B$ takes the form
        \be\la{71}
        \ba{c}
        h(x,t)=S(x,t)g-iBg=
        \Gamma^{-1}(x,t)\widetilde{M}_1T(-x,-t)\widetilde{\rho}\widetilde{\Phi}^*\sigma_B\widetilde{\Phi}g
        = \\ =
        i\Gamma^{-1}(x,t)\widetilde{M}_1T(-x,-t)\widetilde{\rho}(B-B^*)g.
        \ea
        \ee
Consequently
from \rf{71} it follows that
        \be\la{72}
        \Gamma(x,t)h(x,t)=-i\widetilde{M}_1T(-x,-t)\widetilde{\rho}(B-B^*)g.
        \ee
Let us suppose the additional condition
        $
        \widetilde{M}_1=M_1U=I.
        $
Now the equality \rf{72} can be written in the form
        \be\la{73}
        (T(x,t)N+T(-x,-t)\widetilde{M})h(x,t)=-iT(-x,-t)\widetilde{\rho}(B-B^*)g.
        \ee
From \rf{73} it follows that
        \be\la{74}
        Nh(x,t)=-T(-2x,-2t)\widetilde{M})h(x,t)-iT(-2x,-2t)\widetilde{\rho}(B-B^*)g.
        \ee
%
On the other hand with the help of the condition \rf{62-1} the operator $M$ from \rf{32-12} can be written in the form
        \be\la{76}
        M=\int\limits_0^{\infty}e^{-i\eta B^*}\rho\frac{B-B^*}{i}e^{i\eta B}d\eta=
        U\int\limits_0^{\infty}e^{i\eta B}U^*\rho\widetilde{\Phi}^*\sigma_B\widetilde{\Phi}e^{i\eta B}d\eta
        \ee
and consequently
        \be\la{77}
        \widetilde{M}=U^*M=\int\limits_0^{\infty}e^{i\eta B}U^*\rho\widetilde{\Phi}^*\sigma_B\widetilde{\Phi}e^{i\eta B}d\eta.
        \ee
The equality \rf{74} together with \rf{68} and \rf{77} implies that
        \be\la{78}
        \ba{c}
        \widehat{N}P_1h(x,t)=\widehat{N}h(x,t)=
        -T(-2x,-2t)\int\limits_0^{\infty}e^{i\eta B}U^*\rho\widetilde{\Phi}^*\sigma_B\widetilde{\Phi}e^{i\eta B}d\eta h(x,t)
        + \\ +
        T(-2x,-2t)U^*\rho\widetilde{\Phi}^*\sigma_B\widetilde{\Phi}g.
        \ea
        \ee
Without loss of generality we can suppose also $\widehat{N}=I$ (and $\widetilde{\rho}=I$, i.e. $\rho=U$).
        \par
As in Section \ref{s3} we consider the set
        $
        \widetilde{H}=\{\widetilde{\Phi}e^{i\xi B}h,h\in\widehat{H}\}
        $
and the operator $\widetilde{U}:\widehat{H}\longrightarrow\widetilde{H}$, defined by the equality
        $
        \widetilde{U}h=\widetilde{\Phi}e^{i\xi B}h=v_h(\xi).
        $
Now from \rf{78} we obtain, that
        \be\la{79}
        \ba{c}
        v_{P_1h}(\xi)=\widetilde{\Phi}e^{i\xi B}P_1h(x,t)
        =-\int\limits_0^{\infty}\widetilde{\Phi}e^{i\xi B}T(-2x,-2t)e^{i\eta B}
        \widetilde{\rho}\widetilde{\Phi}^*\sigma_B\widetilde{\Phi}e^{i\eta B}h(x,t)d\eta
        + \\ +
        \widetilde{\Phi}e^{i\xi B}T(-2x,-2t)\widetilde{\rho}\widetilde{\Phi}^*\sigma_B\widetilde{\Phi}g
        = \\ =
        -\int\limits_0^{\infty}\widetilde{\Phi}T(-2x+\xi+\eta,-2t)
        \widetilde{\rho}\widetilde{\Phi}^*\sigma_Bv_{h}(\eta)d\eta
        +\widetilde{\Phi}T(-2x+\xi,-2t)\widetilde{\rho}\widetilde{\Phi}^*\sigma_Bv_g(0).
        \ea
        \ee
The equality \rf{79} takes the form \rf{82}, using the denotation \rf{81}.
The theorem is proved.
        \end{proof}
        \par
The integral equation \rf{82} can be considered as a generalized form of the Gelfand-Levitan-Marchenko equation. The matrix function $\Psi(\xi,\tau)$ and the vector function $v_h(\xi)$, appearing in the integral equations \rf{82}, satisfy the conditions \textbf{(a)}, \textbf{(b)}, \textbf{(c)} in Section \ref{s3}.
%
%
        \par
As in Section \ref{s3} we can present the connection between the generalized form of the Gelfand-Levitan-Marchenko equation and the complete characteristic function of the colligation $X$ from the form \rf{81-1}.
        \begin{thm}\la{11}
        The matrix function $\Psi(x,t)$ from \rf{81} in the case of $\widetilde{\rho}=I$ has the representation
        \be\la{83}
        \Psi(\xi,\tau)\widetilde{g}=\frac{1}{2\pi}\int\limits_{|\lambda|=r}e^{i\lambda}W(i\tau,\xi,\lambda)d\lambda\widehat{g}(w),
        \ee
where
        \be\la{84}
        W(i\tau,\xi,\lambda)=I+i\widetilde{\Phi}(i\tau A+\xi B-\lambda)\widetilde{\Phi}^*(i\tau\sigma_A+\xi\sigma_B)
        \ee
($r>a\sqrt{|i\tau|^2+|\xi|^2}$, $a>0$), $\widehat{g}=(i\tau\sigma_A+\xi\sigma_B)^{-1}\sigma_B\widetilde{g}$, $\widetilde{g}\in G_B$, and the generalized Gelfand-Levitan-Marchenko equation \rf{82} has the form
        \be\la{84-1}
        \ba{c}
        v_{P_1h}(\xi)=\frac{1}{2\pi}\int\limits_0^{\infty}\left(\int\limits_{|\lambda|=r}e^{i\lambda}W(i(-2t),-2x+\xi+\eta,\lambda)d\lambda\right)
        . \\ .
        (i(-2t)\sigma_A+(-2x+\xi+\eta)\sigma_B)^{-1}\sigma_Bv_h(\eta)d\eta
        - \\ -
        \frac{1}{2\pi}\int\limits_{|\lambda|=r}e^{i\lambda}W(i(-2t),-2x+\xi,\lambda)d\lambda
        (i(-2t)\sigma_A+(-2x+\xi+\eta)\sigma_B)^{-1}\sigma_Bv_g(0).
        \ea
        \ee
        \end{thm}
        \par
The proof is analogous to the proof of Theorem \ref{t7}.
        \par
Analogous as in Theorem \ref{t7-1} we can obtain the representation of the matrix function $\Psi(x,t)$ from \rf{81} in the form
        $\Psi(x,t)\widetilde{g}=\frac{1}{2\pi}\int\limits_{|\mu|=r}e^{it\mu}W_A(\mu)d\mu\widehat{g}(x)$
with the help of the characteristic operator function
$W_A(\mu)=I+I\widetilde{\Phi}(A-\mu I)^{-1}\widetilde{\Phi}^*\sigma_A$ of the operator $A=bB^2$ for suitable $\widehat{g}(x)$.
        \par
As in Section \ref{s3} we can prove the inverse theorem of Theorem \ref{t10}.
        \begin{thm}\la{t12}
        If the function $\varphi(\xi)\in\widetilde{H}$ (or $\varphi(\xi)=v_h(\xi)$, $h\in\widehat{H}$) is a solution of the integral equation \rf{82} for an arbitrary $g\in G_B$, then the function
        $
        (\widetilde{U}^*\varphi)(x,t)+iBg=h(x,t)+iBg
        $
(where $\widetilde{U}:\widehat{H}\longrightarrow\widetilde{H}$ is defined by the equality $\widetilde{U}h=\widetilde{\Phi}e^{i\xi B}h=v_h(\xi)$)
        is a solution of the nonlinear Schr\"{o}dinger equation \rf{67}.
        \end{thm}
%
%
%
%
%
%
%
        \section{A generalized form of the Gelfand-Levitan-Marchenko equation, the Sine-Gordon equation, and the Davey-Stewartson equation}\la{s5}
        \par
In this section we will present the connection between the nonlinear differential equations: the Sine-Gordon equation and the Davey-Stewartson equation with the generalized Gelfand-Levitan-Marchenko equation. These consideration are based on the solutions of these equations, obtained in \cite{JMAA} with the help of generalized open systems and matrix wave equations for appropriate pairs and triplets of commuting nonselfadjoint operators.
        \par
Let the operators $A$ and $B$ be as in Section \ref{s3}, let $A$ and $B$ satisfy the conditions \textbf{(I)}, \textbf{(II)}, let $A$ and $B$ be embedded in a commutative two-operator colligation $X$ from the form \rf{32-3}.
        \par
At first we will consider the Sine-Gordon equation. In this case the Hilbert space is $H=\mathbf{L}^2(\Delta,\mathbb{C}^p)$, where
        $\Delta=[-l,-d]\cup[d,l]$, $0<d<l$.
Let the operator $B$ in a Hilbert space $H$ be a coupling of dissipative and antidissipative operators with absolutely continuous real spectra. Let $B$ be invertible and satisfy the condition
        \be\la{84-1-0}
        B^*=-UBU^*, \;\;\;\; (U:H\longrightarrow H,\;\; U^*U=UU^*=I).
        \ee
Let the operator $A=B^{-1}$. Without loss of generality we can suppose that the coupling $B$ is the triangular model
        \be\la{84-1-1-1}
        \ba{c}
        Bf(w)=wf(w)\chi_{\Delta}(w)-i\int\limits_{-l}^wf(\xi)\chi_{\Delta}(\xi)\Pi(\xi)S^*\Pi^*(w)d\xi
        +\\+
        i\int\limits_w^{l}f(\xi)\chi_{\Delta}(\xi)\Pi(\xi)S\Pi^*(w)d\xi+
        i\int\limits_{-l}^wf(\xi)\chi_{\Delta}(\xi)\Pi(\xi)L\Pi^*(w)d\xi,
        \ea
        \ee
with $\Delta=[-l,-d]\cup[d,l]$, $0<d<l$ (where $\chi_{\Delta}(w)$ is the characteristic function of the interval $\Delta$). Then the operators $A$ and $B$ are embedded in the regular colligation
        \be\la{84-1-1}
        X=(A=B^{-1},B;H=\mathbf{L}^2(\Delta,\mathbb{C}^p),\widetilde{\Phi},\mathbb{C}^{2m};\sigma_A,\sigma_B,\gamma,\widetilde{\gamma})
        \ee
where $2m\times 2m$ matrices $\sigma_A$, $\sigma_B$, $\gamma$ are defined by the equalities
        $$
        \sigma_A=
        \left(
        \ba{cc}
        0   &  0 \\
        0   &  L
        \ea
        \right), \;\;\;\;\;
        \sigma_B=
        \left(
        \ba{cc}
        L  &  0 \\
        0  &  0
        \ea
        \right),  \;\;\;\;\;
        \gamma=
        \left(
        \ba{cc}
        0  &  -L \\
        -L  &  0
        \ea
        \right)
        $$
and the operator $\widetilde{\Phi}:H\longrightarrow \mathbb{C}^{2m}$ is defined as
        $
        \widetilde{\Phi}=(\Phi \;\;\; \Phi{(B^{-1})}^*),
        $
i.e.
        $$
        \widetilde{\Phi}f=((\Phi f,e_1),\ldots,(\Phi f,e_m),(\Phi{B^{-1}}^*f,e_1),\ldots,(\Phi{B^{-1}}^*f,e_m)),
        $$
where $\Phi$ is defined by the equality
        $
        \Phi f(w)=\int\limits_{\Delta}f(w)\Pi(w)\chi_{\Delta}(w)dw.
        $
The matrices $\sigma_A$, $\sigma_B$, $\gamma$ and the operator $\widetilde{\Phi}$ are determined in \cite{JMAA} and $2m\times 2m$ matrix $\widetilde{\gamma}$ satisfies the equality \rf{32-7}.
        \par
Let the operator function
        $\Gamma(x,t)=T(x,t)N+M_1{T^*}^{-1}(x,t)M$,
where $T(x,t)$ is the collective motion from the form \rf{32-11} with $\varepsilon=-1$, $\delta=1$ and constant linear bounded operators $N$, $M_1$, $M$, $\widetilde{B}$ in $H$ satisfy the conditions
        \be\la{84-1-3}
        NB=BN, \;\;\; N\widetilde{B}=N, \;\;\; M\widetilde{B}=-M, \;\;\; \widetilde{B}^2=I.
        \ee
Then from Theorem 5.1 in \cite{JMAA} it follows that the solitonic combination
        \be\la{84-1-4}
        S(x,t)=\Gamma^{-1}(x,t)\Gamma_x(x,t)
        \ee
satisfies the nonlinear Sine-Gordon equation
        \be\la{84-1-5}
        \frac{\partial}{\partial t}(S^{-1}\frac{\partial}{\partial x}S)=S^{-1}\widetilde{B}S\widetilde{B}-\widetilde{B}S^{-1}\widetilde{B}S.
        \ee
If the operator $M$ has the form \rf{32-12} then $P_{G_B}S(x,t)\left|_{G_B}\right.$ is the scalar solution of the Sine-Gordon equation \rf{84-1-5}, where $P_{G_B} $is an orthogonal projector onto the nonhermotian subspace $G_B=(Im\;B)H$ of the operator $B$. The next theorem gives that the vector function
        $h(x,t)=S(x,t)g-iBg$, $g\in G_B$
satisfies an integral equation (which can be considered as a generalized  Gelfand-Levitan-Marchenko equation).
        \begin{thm}\la{t12-1}
Let the operator $B$ in a Hilbert space $H=\mathbf{L}^2(\Delta,\mathbb{C}^p)$ be the coupling \rf{84-1-0} of dissipative and antidissipative operators with absolutely continuous real spectra, let $B$ be invertible and satisfy the condition
        $$
        B^*=-UBU^*, \;\;\;\; (U:H\longrightarrow H,\;\; U^*U=UU^*=I).
        $$
Let the operators $A=B^{-1}$ and $B$ be embedded in the colligation $X$ from \rf{84-1-1}. Then the mode
        $$
        v_h(\xi)=\widetilde{\Phi}e^{i\xi B}h(x,t)=v_h(\xi,0)
        $$
of the element
        $
        h(x,t)=S(x,t)g-iBg 
        $
($g\in G_B$), where $S(x,t)$ is a solution \rf{84-1-4} (with operators $N$, $M$, $\widetilde{B}$, satisfying the conditions \rf{84-1-3}) of the nonlinear Sine-Gordon equation \rf{84-1-5}, satisfies the integral equation
        \be\la{84-1-6}
        v_{P_1h}(\xi)=-\int\limits_0^{\infty}\Psi(-2x+\xi+\eta,-2t)v_h(\eta)d\eta+
        \Psi(-2x+\xi,-2t)v_g(0),
        \ee
where the operator $\Psi(x,t)$ is defined by the equality
        \be\la{84-1-7}
        \Psi(x,t)=\widetilde{\Phi}T(x,t)\widetilde{\rho}\widetilde{\Phi}^*\sigma_B=
        \widetilde{\Phi}e^{i(-tA+xB)}\widetilde{\rho}\widetilde{\Phi}^*\sigma_B,
         \ee
$\widetilde{\rho}=U^*\rho$ and $P_1=\frac{1}{2}(I+\widetilde{B})$.
        \end{thm}
        \begin{proof}
Using the conditions \rf{84-1-0}, \rf{84-1-3}, and the form of the collective motions $T(x,t)$ and ${T^*}^{-1}(x,t)$ in the case when $\varepsilon=-1$, $\delta=1$ and $A=B^{-1}$ the operator function $\Gamma(x,t)=T(x,t)N+M_1{T^*}^{-1}(x,t)M$
takes the form
        \be\la{84-1-8}
        \Gamma(x,t)=e^{i(xB-tB^{-1})}N+M_1Ue^{-i(xB-tB^{-1})}U^*M=T(x,t)N+\widetilde{M}_1T(-x,-t)\widetilde{M}
        \ee
where $\widetilde{M}_1=M_1U$ and $\widetilde{M}=U^*M$. Then the solitonic combination $S(x,t)$ from \rf{84-1-4} takes the form
        \be\la{84-1-9}
        S(x,t)=i\Gamma^{-1}(x,t)M_1{T^*}^{-1}(x,t)(B^*M-MB)+iB.
        \ee
The relation \rf{84-1-9} together with the condition
        $
        B^*M-MB=-\rho(B-B^*)=-\rho(B+UBU^*)
        $
implies that
        \be\la{84-1-10}
        \ba{c}
        S(x,t)=iB+i\Gamma^{-1}(x,t)\widetilde{M}_1T(-x,-t)U^*\rho(B-B^*)
        = \\ =
        iB-i\Gamma^{-1}(x,t)T(-x,-t)\widetilde{\rho}(B-B^*)
        \ea
        \ee
where
we have used the additional assumption $\widetilde{M}_1=I$, i.e. $M_1=U^*$. Then
        \be\la{84-1-11}
        h(x,t)=S(x,t)g-iBg  \;\;\;\;\;\;\; (g\in G_B)
        \ee
and using \rf{84-1-10} we obtain
        $h(x,t)=-i\Gamma(x,t)T(-x,-t)\widetilde{\rho}(B-B^*)g$.
The last equality implies that
        $\Gamma(x,t)h(x,t)=-iT(-x,-t)\widetilde{\rho}(B-B^*)g$,
i.e.
        \be\la{84-1-12}
        \ba{c}
        Nh(x,t)=-T^{-1}(x,t)T(-x,-t)\widetilde{M}h(x,t)-iT^{-1}(x,t)T(-x,-t)\widetilde{\rho}(B-B^*)g
        = \\ =
        -T(-2x,-2t)\widetilde{M}h(x,t)-iT(-2x,-2t)\widetilde{\rho}(B-B^*)g.
        \ea
        \ee
On the other hand from the conditions \rf{84-1-3}
it follows that the constant operator $N$ and $M$ can be presented in the form
        \be\la{84-1-13}
        N=\widehat{N}(I+\widehat{B})/2=\widehat{N}P_1, \;\;\;\;
        M=\widehat{M}(I-\widehat{B})/2=\widehat{M}P_2
        \ee
where $P_1$, $P_2$ are orthoprojectors.
        \par
        Without loss of generality we can assume that $\widehat{N}=I$ and consequently the equality \rf{84-1-12} take the form
        $P_1h(x,t)=-T(-2x,-2t)\widetilde{M}h(x,t)-iT(-2x,-2t)\widetilde{\rho}(B-B^*)g$.
Now using the representation of the operator $M$ from \rf{32-12}
we obtain that
        $$
        \widetilde{M}=U^*M=U^*U\int\limits_0^{\infty}e^{i\eta B}\widetilde{\rho}\widetilde{\Phi}^*\sigma_B\widetilde{\Phi}e^{i\eta B}d\eta
        $$
and consequently
        $$
        \ba{c}
        v_{P_1h}(\xi)=\widetilde{\Phi}e^{i\xi B}P_1h(x,t)=
        -\int\limits_0^{\infty}\widetilde{\Phi}T(-2x+\xi+\eta,-2t)\widetilde{\rho}\widetilde{\Phi}^*\sigma_Bv_h(\eta)d\eta
        + \\ +
        \widetilde{\Phi}T(-2x+\xi,-2t)\widetilde{\rho}\widetilde{\Phi}^*\sigma_Bv_g(0)
        = \\ =
        -\int\limits_0^{\infty}\Psi(-2x+\xi+\eta,-2t)v_h(\eta)d\eta
        +\Psi(-2x+\xi ,-2t)v_g(0),
        \ea
        $$
where $\Psi(x,t)$ is defined by \rf{84-1-7}.
The theorem is proved.
        \end{proof}
        \par
Now we consider the Davey-Stewartson equation. In this case we consider again the operator $B$ presented as a coupling of dissipative and antidissipative operators with absolutely continuous real spectra. Without loss of generality we can assume that $B$ is the triangular model \rf{21-1} on the Hilbert space $H=\mathbf{L}^2(\Delta,\mathbb{C}^p)$ with $\Delta=[-l,l]$. Let $B^*=-UBU^*$, $U:H\longrightarrow H$, $U^*U=UU^*=I$. Let the operators $A_1=B$, $A_2=\alpha B$, $A_3=\beta B^2$ be embedded in the following commutative regular colligation from the form \rf{14-0} when $n=3$, i.e.
        \be\la{84-3}
        \ba{c}
        X=(A_1=B,A_2=\alpha B,A_3=\beta B^2;H=\mathbf{L}^2(\Delta;\mathbb{C}^p),\widetilde{\Phi},\mathbb{C}^{2m};
        \\
        \sigma_1,\sigma_2,\sigma_3,\gamma_{13},\gamma_{23},\widetilde{\gamma}_{13},\widetilde{\gamma}_{23})
        \ea
        \ee
where $\alpha, \beta\in\mathbb{R}$,
        $$
        \frac{A_1-A_1^*}{i}f=\frac{B-B^*}{i}f=\Phi^*L\Phi f=\widetilde{\Phi}^*\sigma_1\widetilde{\Phi}f, \;\;\;\;
        \frac{A_2-A_2^*}{i}f=\alpha\frac{B-B^*}{i}f
        =\widetilde{\Phi}^*\sigma_2\widetilde{\Phi}f,
        $$
        $$
        \frac{A_3-A_3^*}{i}f=\beta\frac{B^2-{B^2}^*}{i}f
        =\widetilde{\Phi}^*\sigma_3\widetilde{\Phi}f, \;\;\;\;
        \frac{A_1A_3^*-A_3A_1^*}{i}f=\beta B\frac{B^*-B}{i}B^*f
        =\widetilde{\Phi}^*\gamma_{13}\widetilde{\Phi}f,
        $$
        $$
        \frac{A_2A_3^*-A_3A_2^*}{i}f=\alpha\beta B\frac{B^*-B}{i}B^*f
        =\widetilde{\Phi}^*\gamma_{23}\widetilde{\Phi}f.
        $$
$2m\times 2m$ matrices $\sigma_1$, $\sigma_2$, $\sigma_3$, $\gamma_{13}$, $\gamma_{23}$ and the operator
$\widetilde{\Phi}:H\longrightarrow\mathbb{C}^{2m}$ have been obtained by the author in \cite{JMAA} and they have the representations
        \be\la{84-9}
        \ba{c}
        \widetilde{\Phi}=(\Phi \;\;\; \Phi B^*), \;\;\;
        \sigma_1=
        \left(
            \ba{cc}
            L   &  0 \\
            0  &  0
            \ea
        \right), \;\;\;
        \sigma_2=
        \left(
            \ba{cc}
            \alpha L   &  0 \\
            0  &  0
            \ea
        \right), \;\;\;
        \sigma_3=
        \left(
            \ba{cc}
            0        &  \beta L \\
            \beta L  &  0
            \ea
        \right),
        \\
        \gamma_{13}=
        \left(
            \ba{cc}
            0   &  0 \\
            0  &  -\beta L
            \ea
        \right), \;\;\;
        \gamma_{23}=
        \left(
            \ba{cc}
            0   &  0 \\
            0  &  -\alpha\beta L
            \ea
        \right).
        \ea
        \ee
        \par
Let us consider now the corresponding generalized  open system \rf{7} in the case when $n=3$ and $\varepsilon_1=1$,
$\varepsilon_2=\varepsilon_3=i$. The corresponding collective motions \rf{15} have the form
        \be\la{84-10}
        \ba{c}
        T(x,y,t)=e^{i(Bx+i\alpha By+i\beta B^2t)}
        \\
        {T^*}^{-1}(x,y,t)=e^{i(B^*x-i\alpha B^*y-i\beta {B^*}^2t)}
        = \\ =
        Ue^{i(-Bx+i\alpha By-i\beta B^2t)}U^*=
        UT(-x,y,-t)U^*,
        \ea
        \ee
using the condition $B^*=-UBU^*$.
From Theorem 6.1 in \cite{JMAA} it follows that the solitonic combination
        \be\la{84-11}
        S(x,y,t)=\Gamma^{-1}(x,y,t)\Gamma_x(x,y,t)
        \ee
with an operator function $\Gamma(x,y,t)$ from the form
        \be\la{84-12}
        \Gamma(x,y,t)=CT(x,y,t)N+M_1{T^*}^{-1}(x,y,t)M
        \ee
satisfies the nonlinear Davey-Stewartson equation
        \be\la{84-13}
        \left\{
        \ba{l}
        u=[S,\widetilde{B}], \;\;\;  v=\{S,\widetilde{B}\}
        \\
        -\frac{2}{\beta}\frac{\partial}{\partial t}u\widetilde{B}+
        \left(
        \left(-\frac{1}{i\alpha}\frac{\partial}{\partial y}\right)^2+\frac{\partial^2}{\partial x^2}
        \right)u
        +2u^3-2\{u,\frac{\partial}{\partial x}v\}=0
        \\
        -\frac{1}{i\alpha}\frac{\partial}{\partial y}v+\frac{\partial}{\partial x}v\widetilde{B}=u^2,
        \ea
        \right.
        \ee
where the operators $C$, $N$, $M_1$, $M$, $\widetilde{B}$ in \rf{84-12} are constant linear bounded operators on $H$ satisfying the conditions
        \be\la{84-14}
        N\widetilde{B}=N,\;\;\;M\widetilde{B}=-M,\;\;\; \widetilde{B}^2=I.
        \ee
If the operator $M$ has the representation \rf{32-12} then
        $P_{G_B}S(x,y,t)\left|_{G_B}\right.$ is a scalar solution of the Davey-Stewartson equation, where $P_{G_B}$ is an orthogonal projector onto the nonhermitian subspace $G_B=(Im\;B)H$ of the operator $B$.
        \par
Next from \rf{84-12}, \rf{84-10} and the form \rf{32-12} of the operator $M$ it follows that the solitonic combination $S(x,y,t)$ from \rf{84-11} can be written in the form
        $$
        S(x,y,t)=iB-i\Gamma^{-1}(x,y,t)\widetilde{M}_1T(-x,y,-t)\widetilde{\rho}(B-B^*)g
        $$
where we have denoted
        $\widetilde{M}_1=M_1U$, $\widetilde{\rho}=U^*\rho$.
The conditions \rf{84-14} imply that $N$, $M$ satisfy the relations \rf{84-1-13}, where $P_1$, $P_2$ are orthoprojectors. 
        \par
Let now suppose that $\widehat{N}=I$, $C=I$, and $\widetilde{M}_1=I$. It turns out that the mode of the vector function $h(x,y,t)$ defied by the equality
        $h(x,y,t)=S(x,y,t)g-iBg$
satisfies an integral equation, determined in the next theorem.
        \begin{thm}\la{t12-2}
        Let the operator $B$ in the Hilbert space $H=\mathbf{L}^2(\Delta,\mathbb{C}^p)$ be a coupling \rf{21-1} of dissipative and antidissipative operators with absolutely continuous real spectra, let the operator $B$ satisfy the condition
        $$
        B^*=-UBU^* \;\;\;\;\; (U:H\longrightarrow H, \;\; U*U=UU^*=I).
        $$
        Let the operators
        $A_1=B$, $A_2=\alpha B$, $A_3=\beta B^2$ $(\alpha,\beta\in\mathbb{R})$
        be embedded in the colligation $X$ from \rf{84-3}. Then the mode
        $$
        v_h(\xi)=\widetilde{\Phi}e^{i\xi B}h(x,y,t)
        $$
        of the element $h(x,y,t)=S(x,y,t)g-iBg$ ($g\in G_B$), where $S(x,y,t)$, is a solution \rf{84-12} (with $N$, $M$, $\widetilde{B}$ satisfying the conditions \rf{84-14}) of the nonlinear Davey-Stewartson equation \rf{84-13}
        satisfies the integral equation
        \be\la{84-15}
        v_{P_1h}(\xi)=
        -\int\limits_0^{\infty}\Psi(-2x+\xi+\eta,0,-2t)v_h(\eta)d\eta+
        \Psi(-2x+\xi,0,-2t)v_g(0),
        \ee
where
        \be\la{84-15-1}
        \Psi(x,y,t)=
        \widetilde{\Phi}e^{i(xB+i\alpha yB+i\beta tB^2)}\widetilde{\rho}\widetilde{\Phi}^*\sigma_B
        =\widetilde{\Phi}T(x,y,t)\widetilde{\rho}\widetilde{\Phi}^*\sigma_B,
        \ee
$P_1$ is defined by \rf{84-1-13} and $\widehat{N}=I$.
        \end{thm}
        \par
        The proof is analogous to the proof of Theorem \ref{t12-1}.
        \par
The operators $\Psi(x,t)$ and $\Psi(x,y,t)$ from \rf{84-1-7} and \rf{84-15-1} correspondingly have the representation as in Theorem \ref{t7} and Theorem \ref{t7-1} with the help of the complete characteristic function of the colligation \rf{84-1-1} and the colligation \rf{84-3} or with the help of the characteristic operator function of the operator $A$ and the operator $A_3$ correspondingly.
If the mode $\varphi(\xi)=v_{P_1h}(\xi)$ satisfies the integral equation \rf{84-1-6} or \rf{84-15} correspondingly then $h(x,t)+iBg$ or $h(x,y,t)+iBg$ are solutions of the corresponding nonlinear differential equations: the Sine-Gordon equation and the Davey-Stewartson equation.
        \par
Finally, it has to mention also that the mode $v_h(\xi)=\widetilde{\Phi}e^{i\xi B}h$ determines uniquely the output $v_h(\xi,\tau)$ and $v_h(\xi,\zeta,\tau)$ satisfying the differential equation (the matrix wave equations)
        $$
        \left\{
        \ba{l}
        \sigma_B\left(-i\frac{1}{\varepsilon}\frac{\partial v}{\partial \tau}\right)-
        \sigma_A\left(-i\frac{1}{\delta}\frac{\partial v}{\partial \xi}\right)+\widetilde{\gamma}v=0
        \\
        v(\xi,0)=v_h(\xi)
        \ea
        \right.
        $$
(when $\varepsilon=-1$, $\delta=1$) and
        $$
        \left\{
        \ba{l}
        \sigma_3\left(-i\frac{1}{\varepsilon_1}\frac{\partial v}{\partial \xi}\right)-
        \sigma_1\left(-i\frac{1}{\varepsilon_3}\frac{\partial v}{\partial \tau}\right)+\widetilde{\gamma}_{13}v=0
        \\
        \sigma_3\left(-i\frac{1}{\varepsilon_2}\frac{\partial v}{\partial \zeta}\right)-
        \sigma_2\left(-i\frac{1}{\varepsilon_3}\frac{\partial v}{\partial \tau}\right)+\widetilde{\gamma}_{23}v=0
        \\
        v(\xi,0,0)=v_h(\xi)
        \ea
        \right.
        $$
(when $\varepsilon_1=1$, $\varepsilon_2=i$, $\varepsilon_3=i$) using Theorem \ref{t3}.
%
%
%
%
        \section{A special case of the input and the output of the open system, corresponding to the Schr\"{o}dinger equation}\la{s6}
        \par
In this section we consider the generalized open system, connected with obtaining of solutions of the nonlinear Schr\"{o}dinger equation in the special case of separated variables in the input, the internal state and the output. We derive what kind of differential equations are satisfied by the components of the input and the output of the corresponding open  system.
        \par
Let the operators $A$, $B$, the collective motions \rf{58}, the colligation $X$ from \rf{81-1} be like in Section \ref{s4}. Let us consider the generalized open system, corresponding to $X$ (in the case when $\varepsilon=i$, $\delta=1$) from the form \rf{59}, i.e.
        \be\la{84-200}
        \left\{
        \ba{l}
        \frac{\partial}{\partial t}f+Af=\widetilde{\Phi}^*\sigma_Au   \\
        i\frac{\partial}{\partial x}f+Bf=\widetilde{\Phi}^*\sigma_Bu   \\
        v=u-i\widetilde{\Phi}f,
        \ea
        \right.
        \ee
where $u=u(x,t)$, $v=v(x,t)$, $f=f(x,t)$ are the input, the output and the state of the system.
        \par
From Theorem 3.3 in \cite{JMAA} it follows  that for the comutative regular colligation $X$ the collective motions are compatible if and only if the input and the output satisfy the following partial differential equations (or the matrix wave equations) correspondingly
        \be\la{85}
        \sigma_B\left(-i \frac{1}{\varepsilon} \frac{\partial u}{\partial t}\right)
        -\sigma_A\left(-i \frac{1}{\delta} \frac{\partial u}{\partial x}\right)+\gamma u=0,
        \ee
        \be\la{86}
        \sigma_B\left(-i \frac{1}{\varepsilon} \frac{\partial v}{\partial t}\right)
        -\sigma_A\left(-i \frac{1}{\delta} \frac{\partial v}{\partial x}\right)+\widetilde{\gamma}v=0.
        \ee
The equations \rf{85} and \rf{86} in the case when $\varepsilon=i$, $\delta=1$ have the form
        \be\la{87}
        -\sigma_B\frac{\partial u}{\partial t}
        +i\sigma_A\frac{\partial u}{\partial x}+\gamma u=0,
        \ee
        \be\la{88}
        -\sigma_B\frac{\partial v}{\partial t}
        +i\sigma_A\frac{\partial v}{\partial x}+\widetilde{\gamma} v=0.
        \ee
Let us consider now the special case of separated variables in the input, the output and the state when
        \be\la{89}
        u(x,t)=e^{-\lambda t}u_{\lambda}(x), \;\;\;\;v(x,t)=e^{-\lambda t}v_{\lambda}(x), \;\;\;\;f(x,t)=e^{-\lambda t}f_{\lambda}(x).
        \ee
Then the open system \rf{84-200} takes the form
        \be\la{90}
        \left\{
        \ba{l}
        -\lambda f_{\lambda}(x)+Af_{\lambda}(x)=\widetilde{\Phi}\sigma_Au_{\lambda}(x)   \\
        i\frac{df_{\lambda}}{dx}+Bf_{\lambda}(x)=\widetilde{\Phi}\sigma_Bu_{\lambda}(x)   \\
        v_{\lambda}(x)=u_{\lambda}(x)-i\widetilde{\Phi}f_{\lambda}(x).
        \ea
        \right.
        \ee
Consequently,
        $f_{\lambda}(x)=(A-\lambda I)^{-1}\widetilde{\Phi}\sigma_Au_{\lambda}(x)$
and
        \be\la{92}
        v_{\lambda}(x)=u_{\lambda}(x)-i\widetilde{\Phi}(A-\lambda I)^{-1}\widetilde{\Phi}\sigma_Au_{\lambda}(x)=W_A(\lambda)u_{\lambda}(x),
        \ee
where $W_A(\lambda)$ is the characteristic operator function of the operator $A$ ($\lambda$ does not belong to the spectrum of the operator $A$, i.e. $\lambda$ is a regular point of the characteristic operator function of $A$).
        \par
It has to be mentioned that the compatibility conditions \rf{87} and \rf{88} take the form
        \be\la{93}
        \lambda\sigma_Bu_{\lambda}+i\sigma_A\frac{du_{\lambda}}{dx}+\gamma u_{\lambda}=0,
        \ee
        \be\la{94}
        \lambda\sigma_Bv_{\lambda}+i\sigma_A\frac{dv_{\lambda}}{dx}+\widetilde{\gamma}v_{\lambda}=0.
        \ee
        \par
The form \rf{66} of the operators $\sigma_A$, $\sigma_B$, $\gamma$, $\widetilde{\Phi}$, the condition \rf{32-7},
and direct calculations show that the matrix $\widetilde{\gamma}$ has the representation
        \be\la{95}
        \widetilde{\gamma}=
        \left(
        \ba{cc}
        -ibL(\pi_{12}-\pi^*_{12})L       &    -ibL\pi_{11}L   \\
        ibL\pi_{11}L                   &    bL
        \ea
        \right),
        \ee
where we have used the block representation of the selfadjoint matrix
        \be\la{96}
        \widetilde{\Phi}\widetilde{\Phi}^*h=h
        \left(
        \ba{cc}
        \pi_{11}  &  \pi_{12}  \\
        \pi_{12}^*  &  \pi_{22}
        \ea
        \right),
        \ee
where the matrices $\pi_{11}$, $\pi_{22}$ satisfy
        $\pi_{11}^*=\pi_{11}$, $\pi_{22}^*=\pi_{22}$ and they do not depend on $x$.
Let us denote
        \be\la{98}
        u_{\lambda}(x)=(u_1(x),u_2(x)), \;\;\;\;\; v_{\lambda}(x)=(v_1(x),v_2(x)).
        \ee
Then the equality \rf{94} takes the form
        $$
        \ba{c}
        \lambda(v_1,v_2)
        \left(
        \ba{cc}
        L   &   0  \\
        0   &   0
        \ea
        \right)
        +i\left(\frac{dv_1}{dx},\frac{dv_2}{dx}\right)
        \left(
        \ba{cc}
        0   &   bL  \\
        bL  &   0
        \ea
        \right)
        + \\ +
        (v_1,v_2)
        \left(
        \ba{cc}
        -ibL(\pi_{12}-\pi^*_{12})L       &    -ibL\pi_{11}L   \\
        ibL\pi_{11}L                     &    bL
        \ea
        \right)=0,
        \ea
        $$
i.e.
        \be\la{99}
        \ba{c}
        \left(\lambda v_1L+iB\frac{dv_2}{dx}L-ibv_1L(\pi_{12}-\Pi_{12}^*)L+ibv_2L\pi_{11}L,
        \right.
        \\
        \left.
        ib\frac{dv_1}{dx}L-ibv_1L\pi_{11}L+bv_2L\right)=0,
        \ea
        \ee
which is equivalent to the equations
        \be\la{100}
        \left\{
        \ba{l}
        \lambda v_1+ib\frac{dv_2}{dx}-ibv_1L(\pi_{12}-\Pi_{12}^*)+ibv_2L\pi_{11}=0
        \\
        i\frac{dv_1}{dx}-iv_1L\pi_{11}+v_2=0.
        \ea
        \right.
        \ee
The equations \rf{100} and straightforward calculations show that $v_1(x)$ satisfies the equation
        \be\la{101}
        -\frac{d^2v_1}{dx^2}+v_1((L\pi_{11})^2+iL(\pi_{12}-\pi_{12}^*))=\frac{\lambda}{b}v_1
        \ee
If we suppose the additional condition 
        \be\la{102}
        \pi_{11}L(\pi_{12}-\pi_{12}^*)=(\pi_{12}-\pi_{12}^*)\pi_{11}
        \ee
(concerning the operator $\widetilde{\Phi}$ and the representation\rf{96}) analogously from the equations \rf{100} it follows that $v_2(x)$ satisfies the equation
        \be\la{103}
        -\frac{d^2v_2}{dx^2}+v_2((L\pi_{11})^2+iL(\pi_{12}-\pi_{12}^*))=\frac{\lambda}{b}v_2.
        \ee
        \par
Consequently in the case when the matrices $\pi_{11}$, $\pi_{12}$ satisfy the condition \rf{102}, the output of the system
        $v(x,t)=e^{-\lambda t}v_{\lambda}(x)=(e^{-\lambda t}v_{1}(x),e^{-\lambda t}v_{2}(x))$
satisfies the equations
        \be\la{104}
        \ba{c}
        -\frac{d^2v_1}{dx^2}+v_1((L\pi_{11})^2+iL(\pi_{12}-\pi_{12}^*))=\frac{\lambda}{b}v_1 \\
        -\frac{d^2v_2}{dx^2}+v_2((L\pi_{11})^2+iL(\pi_{12}-\pi_{12}^*))=\frac{\lambda}{b}v_2.
        \ea
        \ee
        Analogously for the vector function $u_{\lambda}(x)=(u_1(x),u_2(x))$ using the compatible condition (or matrix wave equation) \rf{93}, the form \rf{66} of the matrices $\sigma_A$, $\sigma_B$, $\gamma$, the representation \rf{96} of the operator $\widetilde{\Phi}$ we obtain
        $$
        \ba{c}
        \lambda(u_1,u_2)
        \left(
        \ba{cc}
        L   &   0  \\
        0   &   0
        \ea
        \right)
        +i\left(\frac{du_1}{dx},\frac{du_2}{dx}\right)
        \left(
        \ba{cc}
        0   &   bL  \\
        bL  &   0
        \ea
        \right)
        + \\ +
        (u_1,u_2)
        \left(
        \ba{cc}
        0       &    0   \\
        0       &    bL
        \ea
        \right)=0,
        \ea
        $$
Hence
        \be\la{105}
        (\lambda u_1L,0)+\left(i\frac{du_2}{dx}bL,i\frac{du_1}{dx}bL\right)+(0,bu_2L)=0,
        \ee
i.e. the vector functions $u_1$, $u_2$ satisfy the system
        \be\la{106}
        \left\{
        \ba{l}
        \lambda u_1+ib\frac{du_2}{dx}=0  \\
        ib\frac{du_1}{dx}+bu_2=0.
        \ea
        \right.
        \ee
From the equations \rf{106} it follows that the input
        $u(x,t)=e^{-\lambda t}u_{\lambda}(x)=(e^{-\lambda t}u_{1}(x),e^{-\lambda t}u_{2}(x))$
of the open system \rf{84} satisfies the equations
        \be\la{107}
        \left\{
        \ba{c}
        -\frac{d^2u_1}{dx^2}=\frac{\lambda}{b}u_1  \\
        -\frac{d^2u_2}{dx^2}=\frac{\lambda}{b}u_2.
        \ea
        \right.
        \ee
Consequently we prove the next theorem in the case of separated variables of the input, the internal state and the output  of the open system \rf{84} (when the operators $A$, $B$ do not depend on $x$ and $t$).
        \begin{thm}\la{t13}
        The input $u(x,t)=e^{-\lambda t}u_{\lambda}(x)=(e^{-\lambda t}u_{1}(x),e^{-\lambda t}u_{2}(x))$  and  the output
        $v(x,t)=e^{-\lambda t}v_{\lambda}(x)=(e^{-\lambda t}v_{1}(x),e^{-\lambda t}v_{2}(x))$
        of the open system \rf{84-200} satisfy the compatibility conditions
        $$
        \ba{c}
        \sigma_A\frac{du_{\lambda}}{dx}-i(\lambda\sigma_B+\gamma)u_{\lambda}=0  \\
        \sigma_A\frac{dv_{\lambda}}{dx}-i(\lambda\sigma_B+\widetilde{\gamma})v_{\lambda}=0
        \ea
        $$
and the characteristic operator function $W_A(\lambda)=I-i\widetilde{\Phi}(A-\lambda I)^{-1}\widetilde{\Phi}\sigma_A$ of the operator $A=bB^2$
maps the input $u(x,t)$, satisfying the equations
        $$
        \ba{c}
        -\frac{d^2u_1}{dx^2}=\frac{\lambda}{b}u_1  \\
        -\frac{d^2u_2}{dx^2}=\frac{\lambda}{b}u_2
        \ea
        $$
to the output $v(x,t)=W_A(\lambda)u(x,t)$ which components $v_1(x)$, $v_2(x)$ are solutions of the equations
        $$
        \ba{c}
        -\frac{d^2v_1}{dx^2}+v_1((L\pi_{11})^2+iL(\pi_{12}-\pi_{12}^*))=\frac{\lambda}{b}v_1 \\
        -\frac{d^2v_2}{dx^2}+v_2((L\pi_{11})^2+iL(\pi_{12}-\pi_{12}^*))=\frac{\lambda}{b}v_2.
        \ea
        $$
(when the operator $\widetilde{\Phi}$ satisfies the condition \rf{102}), and
        $
        v_{\lambda}(x)=W_A(\lambda)u_{\lambda}(x).
        $
        \end{thm}
%
%
%
%
        \section{The case when the operators $\mathbf{A}$ and $\mathbf{B}$ depend on the spatial variable $\mathbf{x}$ and the Schr\"{o}dinger equation}\la{s7}
The results, obtained in Section \ref{s6}, can be expanded in the case when the operators $A$ and $B$ depend on the spatial variable $x$. In other words we consider a special case of separated variables in the input, the state and the output  of the generalized open system, connected with obtaining of solutions of the nonlinear Schr\"{o}dinger equation when the operators $A$ and $B$ depend on the spatial variable $x$. We derive what kind of differential equations are satisfied by the components of the input and the output of the corresponding generalized open system.
        \par
Let us consider now regular colligations (or vessels) which depend on the spatial variable $x$. This is the case when the operators $A$, $B$, $\widetilde{\Phi}$ depend on the spatial variable $x$, i.e. the matrix function $\Pi(w)$ depends also on the variable $x$, i.e.
        $\Pi=\Pi(w,x)$.
        \par
Using the results of M.S. Liv\v sic in the article \cite{LIVSIC-VORTICES} (Section 3.5) it follows that the commuting operators $A(x)$, $B(x)$ are embedded in the strict colligation
        $$
        X=(A(x),B(x);H,\widetilde{\Phi},E=\mathbb{C}^{2m};
        \sigma_A,\sigma_B,\psi(x),\widetilde{\psi}(x)),
        $$
 with the corresponding colligation conditions (or vessel conditions)
        \be\la{108}
        (A(x)-A^*(x))/i=\widetilde{\Phi}^*(x)\sigma_A\widetilde{\Phi}(x), \;\;\;\;
        (B(x)-B^*(x))/i=\widetilde{\Phi}^*(x)\sigma_B\widetilde{\Phi}(x),
        \ee
        \be\la{110}
        \frac{1}{i}\sigma_A\frac{d\widetilde{\Phi}(x)}{dx}+\sigma_A\widetilde{\Phi}(x)B^*(x)-\sigma_B\widetilde{\Phi}(x)A^*(x)=\psi(x)\widetilde{\Phi}(x),
        \ee
        \be\la{111}
        \frac{1}{i}\sigma_A\frac{d\widetilde{\Phi}(x)}{dx}+\sigma_A\widetilde{\Phi}(x)B(x)-\sigma_B\widetilde{\Phi}(x)A(x)=
        \widetilde{\psi}(x)\widetilde{\Phi}(x),
        \ee
        \be\la{112}
        \widetilde{\psi}(x)=\psi(x)+i(\sigma_A\widetilde{\Phi}(x)\widetilde{\Phi}^*(x)\sigma_B-
        \sigma_B\widetilde{\Phi}(x)\widetilde{\Phi}^*(x)\sigma_A).
        \ee
The corresponding generalized open system have the form
        \be\la{113}
        \left\{
        \ba{l}
        \frac{\partial f(x,t)}{\partial t}+A(x)f(x,t)=\widetilde{\Phi}^*(x)\sigma_Au(x,t)
        \\
        i\frac{\partial f(x,t)}{\partial x}+B(x)f(x,t)=\widetilde{\Phi}^*(x)\sigma_Bu(x,t)
        \\
        v(x,t)=u(x,t)-i\widetilde{\Phi}(x)f(x,t)
        \ea
        \right.
        \ee
($t_0\leq t\leq t_1$, $x_0\leq x\leq x_1$). In this case the operators $\psi(x)$ and $\widetilde{\psi}(x)$ are selfadjoint.
        \par
Given the further consideration here we will present at first a general case. Let us consider a generalized open system with the form (according to the case when operators $A$, $B$ do not depend on the variables $x$ and $t$)
        \be\la{113-1}
        \left\{
        \ba{l}
        i\frac{1}{\varepsilon}\frac{\partial f(x,t)}{\partial t}+A(x)f(x,t)=\widetilde{\Phi}^*(x)\sigma_Au(x,t)
        \\
        i\frac{1}{\delta}\frac{\partial f(x,t)}{\partial x}+B(x)f(x,t)=\widetilde{\Phi}^*(x)\sigma_Bu(x,t)
        \\
        v(x,t)=u(x,t)-i\widetilde{\Phi}(x)f(x,t)
        \ea
        \right.
        \ee
($t_0\leq t\leq t_1$, $x_0\leq x\leq x_1$, $\varepsilon,\delta\in\mathbb{C}$). The colligation conditions now have the form
        \be\la{113-1-1}
        (A(x)-A^*(x))/i=\widetilde{\Phi}^*(x)\sigma_A\widetilde{\Phi}(x), \;\;\;\;
        (B(x)-B^*(x))/i=\widetilde{\Phi}^*(x)\sigma_B\widetilde{\Phi}(x),
        \ee
        \be\la{113-2}
        \frac{1}{i\overline{\delta}}\sigma_A\frac{d\widetilde{\Phi}(x)}{dx}+\sigma_A\widetilde{\Phi}(x)B^*(x)-\sigma_B\widetilde{\Phi}(x)A^*(x)=\psi(x)\widetilde{\Phi}(x),
        \ee
        \be\la{113-3}
        \frac{1}{i\delta}\sigma_A\frac{d\widetilde{\Phi}(x)}{dx}+\sigma_A\widetilde{\Phi}(x)B(x)-\sigma_B\widetilde{\Phi}(x)A(x)=
        \widetilde{\psi}(x)\widetilde{\Phi}(x),
        \ee
        \be\la{113-4}
        \widetilde{\psi}(x)=\psi(x)+i(\sigma_A\widetilde{\Phi}(x)\widetilde{\Phi}^*(x)\sigma_B-
        \sigma_B\widetilde{\Phi}(x)\widetilde{\Phi}^*(x)\sigma_A).
        \ee
Using the results of M.S. Liv\v sic in \cite{LIVSIC-VORTICES} we obtain that
\\
        \be\la{113-5}
        \frac{i}{\delta}\frac{dA(x)}{dx}f=A(x)B(x)f-B(x)A(x)f.
        \ee
        \begin{thm}\la{t14}
        The input $u(x,t)$ and the output $v(x,t)$ of the open system \rf{113-1} with $A(x)B(x)=B(x)A(x)$ satisfy the strong compatibility conditions (or matrix wave equations)
        \be\la{113-6}
        \sigma_B\left(-i\frac{1}{\varepsilon}\frac{\partial u}{\partial t}(x,t)\right)-
        \sigma_A\left(-i\frac{1}{\delta}\frac{\partial u}{\partial x}(x,t)\right)+\psi(x)u(x,t)=0,
        \ee
        \be\la{113-7}
        \sigma_B\left(-i\frac{1}{\varepsilon}\frac{\partial v}{\partial t}(x,t)\right)-
        \sigma_A\left(-i\frac{1}{\delta}\frac{\partial v}{\partial x}(x,t)\right)+\widetilde{\psi}(x)v(x,t)=0.
        \ee
        \end{thm}
        \begin{proof}
Let the input $u(x,t)$, the inner state $f(x,t)$, and the output $v(x,t)$ satisfy the equations of the open system \rf{113-1}. From the equations of the open system \rf{113-1} it follows that
        \be\la{113-8}
        \ba{c}
        i\frac{\partial^2f(x,t)}{\partial x\partial t}+\varepsilon\frac{dA(x)}{dx}f(x,t)+\varepsilon A(x)\frac{\partial f(x,t)}{\partial x}
        = \\ =
        \varepsilon\frac{d\widetilde{\Phi}^*(x)}{dx}\sigma_Au(x,t)+\varepsilon\widetilde{\Phi}^*(x)\sigma_A\frac{\partial u(x,t)}{dx}
        \ea
        \ee
and
        \be\la{113-9}
        i\frac{\partial^2f(x,t)}{\partial t\partial x}+\delta B(x)\frac{df(x,t)}{dt}=\delta\widetilde{\Phi}^*(x)\sigma_B\frac{\partial u(x,t)}{dt}.
        \ee
Now the equality of mixed partials
        $
        \frac{\partial^2f(x,t)}{\partial x\partial t}=\frac{\partial^2f(x,t)}{\partial t\partial x}
        $
together with the equalities \rf{113-8} and \rf{113-9} show that
        \be\la{118}
        \varepsilon\frac{dA}{dx}f+\varepsilon A\frac{\partial f}{\partial x}-\delta B\frac{df}{dt}=
        \varepsilon\frac{d\widetilde{\Phi}^*}{dx}\sigma_Au+\varepsilon\widetilde{\Phi}^*\sigma_A\frac{\partial u}{dx}
        -\delta\widetilde{\Phi}^*\sigma_B\frac{\partial u}{dt}.
        \ee
But from \rf{113-1} it follows that
        \be\la{119}
        \frac{\partial f}{\partial t}=i\varepsilon Af-i\varepsilon\widetilde{\Phi}^*\sigma_Au, \;\;\;\; \frac{\partial f}{\partial x}=i\delta Bf-i\delta\widetilde{\Phi}^*\sigma_Bu,
        \ee
which together with \rf{118} implies that
        \be\la{122}
        -\widetilde{\Phi}^*\left(-\delta\sigma_B\frac{\partial u}{\partial t}+
        \varepsilon\sigma_A\frac{\partial u}{\partial x}\right)-\varepsilon\frac{d\widetilde{\Phi}^*}{dx}\sigma_Au
        +i\varepsilon\delta(-A\widetilde{\Phi}^*\sigma_Bu+B\widetilde{\Phi}^*\sigma_Au)=0.
        \ee
The equality \rf{122} and the colligation condition \rf{113-2} imply that
        \be\la{123-1}
        \widetilde{\Phi}^*\left(-\delta\sigma_B\frac{\partial u}{\partial t}
        +\varepsilon\sigma_A\frac{\partial u}{\partial x}\right)+
        \varepsilon\frac{dA}{dx}f-\varepsilon\frac{d\widetilde{\Phi}^*}{dx}\sigma_Au+
        i\varepsilon\delta\left(\frac{1}{i\delta}\frac{d\widetilde{\Phi}^*}{dx}\sigma_Au+\widetilde{\Phi}^*\psi u\right)=0.
        \ee
Consequently,
        \be\la{125}
        \widetilde{\Phi}^*\left(-\delta\sigma_B\frac{\partial u}{\partial t}+\varepsilon\sigma_A\frac{\partial u}{\partial x}
        -i\varepsilon\delta\psi u\right)-\varepsilon\frac{dA}{dx}f=0.
        \ee
On the other hand from \rf{113-5} we have that
        $i\frac{dA}{dx}=A(x)B(x)-B(x)A(x)=0$.
Then the equality \rf{125} implies that
        \be\la{127-1}
        -\delta\sigma_B\frac{\partial u}{\partial t}+\varepsilon\sigma_A\frac{\partial u}{\partial x}-i\varepsilon\delta\psi u=0
        \ee
which proves the equality \rf{113-6}.
        \par
Now from \rf{113-1} we have $u=v+i\widetilde{\Phi}f$ and using \rf{127-1} we obtain
        \be\la{128}
        -\delta\sigma_B\left(\frac{dv}{dt}+i\widetilde{\Phi}\frac{\partial f}{\partial t}\right)
        +\varepsilon\sigma_A\left(\frac{dv}{dx}+i\frac{d\widetilde{\Phi}}{dx}f+i\widetilde{\Phi}\frac{\partial f}{\partial x}\right)
        -\varepsilon\delta\psi(v+i\widetilde{\Phi}f)=0.
        \ee
Next from \rf{119} it follows that \rf{128} takes the form
        \be\la{129}
        \ba{c}
        -\delta\sigma_B\frac{dv}{dt}+\varepsilon\sigma_A\frac{dv}{dx}+\varepsilon\delta(\sigma_B\widetilde{\Phi}Af-\sigma_A\widetilde{\Phi}Bf)-
        \varepsilon\delta(\sigma_B\widetilde{\Phi}\widetilde{\Phi}^*\sigma_Au-\sigma_A\widetilde{\Phi}\widetilde{\Phi}^*\sigma_Bu)
        + \\ +
        i\varepsilon\sigma_A\frac{d\widetilde{\Phi}}{dx}f-i\varepsilon\delta\psi u=0.
        \ea
        \ee
Now the equality \rf{129} together with \rf{113-3} and \rf{113-4} gives that
        $-\delta\sigma_B\frac{dv}{dt}+\varepsilon\sigma_A\frac{dv}{dx}-i\varepsilon\delta\widetilde{\psi}v=0$,
which proves the equality \rf{113-7}. The theorem is proved.
        \end{proof}
        \par
Let us consider the operator $B(x)$, presented as a coupling of dissipative and antidissipative operators with real absolutely continuous spectra. Without loss of generality we can assume that the operator $B(x)$ is the triangular model \rf{21-1} in the Hilbert space $H=\mathbf{L}^2(\Delta;\mathbb{C}^p)$ where $\Delta=[-l,l]$. Let the operators $\Pi(w,x)$, $Q(w,x)$, $\widetilde{\Pi}(w,x)$, $L$, $\Phi(x)$ be like in Section \ref{s2}, but depending on the other variable $x$ ($x_0\leq x\leq x_1$). Let the operator $B(x)$ satisfy the condition
        $B^*=-UBU^*$ ($U:H\longrightarrow H$, $U^*U=UU^*=I$).
        \par
        Let the operators $A(x)=bB^2(x)$ ($b\in\mathbb{R}$)  and $B(x)$ be embedded in the colligation
        $$
        X=(A(x)=bB^2(x),B(x);H=\mathbf{L}^2(\Delta;\mathbb{C}^p),\widetilde{\Phi},E=\mathbb{C}^{2m};\sigma_A,\sigma_B,\psi(x),\widetilde{\psi}(x)),
        $$
where $\sigma_A$, $\sigma_B$, $\widetilde{\Phi}$ are defined by the equalities \rf{66} and $\widetilde{\Phi}=\widetilde{\Phi}(x)$.
        \par
Let us consider the case of generalized open system \rf{113-1} when $\varepsilon=i$, $\delta=1$
        \be\la{129-1}
        \left\{
        \ba{l}
        \frac{\partial f(x,t)}{\partial t}+A(x)f(x,t)=\widetilde{\Phi}^*(x)\sigma_Au(x,t)
        \\
        i\frac{\partial f(x,t)}{\partial x}+B(x)f(x,t)=\widetilde{\Phi}^*(x)\sigma_Bu(x,t)
        \\
        v(x,t)=u(x,t)-i\widetilde{\Phi}(x)f(x,t)
        \ea
        \right.
        \ee
($t_0\leq t\leq t_1$, $x_0\leq x\leq x_1$). This case corresponds to the Schr\"{o}dinger equation as in Section \ref{s4}. Then the corresponding colligation conditions are as \rf{113-1-1}---\rf{113-4} (with $\delta=1$).
        \par
Next we determine the form of the matrix functions $\psi(x)$ and $\widetilde{\psi}(x)$. From the equality \rf{110} we have
        \be\la{130}
        \frac{1}{i}\widetilde{\Phi}^*\sigma_A\frac{d\widetilde{\Phi}}{dx}+\widetilde{\Phi}^*\sigma_A\widetilde{\Phi}B^*-
        \widetilde{\Phi}^*\sigma_B\widetilde{\Phi}A^*=\widetilde{\Phi}^*\psi\widetilde{\Phi}.
        \ee
But
        \be\la{131}
        \widetilde{\Phi}^*\sigma_A\widetilde{\Phi}B^*-\widetilde{\Phi}^*\sigma_B\widetilde{\Phi}A^*=
        \frac{A-A^*}{i}B^*-\frac{B-B^*}{i}A^*=\frac{1}{i}(AB^*-BA^*)=\widetilde{\Phi}^*\gamma\widetilde{\Phi},
        \ee
where the matrix $\gamma$ has the form \rf{66}. On the other hand the equality \rf{130} together with \rf{131} implies that
        \be\la{132}
        \frac{1}{i}\widetilde{\Phi}^*\sigma_A\frac{d\widetilde{\Phi}}{dx}+\widetilde{\Phi}^*\gamma\widetilde{\Phi}=\widetilde{\Phi}^*\psi\widetilde{\Phi}, \;\;\;\;\;
        -\frac{1}{i}\frac{d\widetilde{\Phi}^*}{dx}\sigma_A\widetilde{\Phi}+\widetilde{\Phi}^*\gamma\widetilde{\Phi}=\widetilde{\Phi}^*\psi\widetilde{\Phi}.
        \ee
From \rf{132} it follows that
        \be\la{134}
        \frac{1}{i}\widetilde{\Phi}^*\sigma_A\frac{d\widetilde{\Phi}}{dx}=-\frac{1}{i}\frac{d\widetilde{\Phi}^*}{dx}\sigma_A\widetilde{\Phi}
        =\left(\frac{1}{i}\widetilde{\Phi}^*\sigma_A\frac{d\widetilde{\Phi}}{dx}\right)^*.
        \ee
Consequently, $\frac{1}{i}\widetilde{\Phi}^*\sigma_A\frac{d\widetilde{\Phi}}{dx}$ can be presented in the form
        \be\la{135}
        \frac{1}{i}\widetilde{\Phi}^*\sigma_A\frac{d\widetilde{\Phi}}{dx}=\widetilde{\Phi}^*(x)\sigma_1(x)\widetilde{\Phi}(x),
        \ee
where $\sigma_1^*(x)=\sigma_1(x)$. Hence from \rf{132} and \rf{135} we obtain that
        \be\la{136}
        \widetilde{\Phi}^*(x)\sigma_1(x)\widetilde{\Phi}(x)+\widetilde{\Phi}^*(x)\gamma\widetilde{\Phi}(x)
        =\widetilde{\Phi}^*(x)\psi(x)\widetilde{\Phi}(x),
        \ee
i.e. we can consider $\psi(x)$ as a matrix function from the form $\psi(x)=\sigma_1(x)+\gamma$.
Let us present $\psi(x)$ in the form
        \be\la{138}
        \psi(x)=
        \left(
        \ba{cc}
        bL\psi_{11}(x)L     &   ibL\psi_{12}(x)L \\
        -ibL\psi_{12}^*(x)L &   bL
        \ea
        \right)
        =
        \left(
        \ba{cc}
        \widehat{\psi}_{11}   &   \widehat{\psi}_{12} \\
        \widehat{\psi}_{12}^*   &   bL
        \ea
        \right),
        \ee
where $\psi_{11}^*=\psi_{11}$.
Now using the block representation of the matrix function $\widetilde{\Phi}(x)\widetilde{\Phi}(x)^*$
        \be\la{139}
        \widetilde{\Phi}(x)\widetilde{\Phi}^*(x)=
        \left(
        \ba{cc}
        \pi_{11}(x)   &   \pi_{12}(x) \\
        \pi_{12}^*(x)   &   \pi_{22}(x)
        \ea
        \right)
        \ee
($\pi_{11}^*(x)=\pi_{11}(x)$, $\pi_{22}^*(x)=\pi_{22}(x)$), the form \rf{66} of the matrices $\sigma_A$, $\sigma_B$, from \rf{112} we obtain
        \be\la{140}
        \ba{c}
        \widetilde{\psi}(x)=\psi(x)+i
        \left(
        -\left(
        \ba{cc}
        L   &   0  \\
        0   &   0
        \ea
        \right)
        \left(
        \ba{cc}
        \pi_{11}(x)   &   \pi_{12}(x) \\
        \pi_{12}^*(x)   &   \pi_{22}(x)
        \ea
        \right)
        \left(
        \ba{cc}
        0   &   bL  \\
        bL   &   0
        \ea
        \right)
        \right.
        + \\ +
        \left.
        \left(
        \ba{cc}
        0   &   bL  \\
        bL   &   0
        \ea
        \right)
        \left(
        \ba{cc}
        \pi_{11}(x)   &   \pi_{12}(x) \\
        \pi_{12}^*(x)   &   \pi_{22}(x)
        \ea
        \right)
        \left(
        \ba{cc}
        L   &   0  \\
        0   &   0
        \ea
        \right)
        \right)
        = \\ =
        \psi-i
        \left(
        \ba{cc}
        bL(\pi_{12}-\pi_{12}^*)L   &   bL\pi_{11}L  \\
        -bL\pi_{11}L               &   0
        \ea
        \right)
        = \\ =
        \left(
        \ba{cc}
        \widehat{\psi}_{11}-ibL(\pi_{12}-\pi_{12}^*)L   &   \widehat{\psi}_{12}-ibL\pi_{11}L  \\
        \widehat{\psi}_{12}^*+ibL\pi_{11}L              &   \widehat{\psi}_{22}
        \ea
        \right).
        \ea
        \ee
Now the matrix function $\widetilde{\psi}(x)$ can be written in the form
        \be\la{141}
        \widetilde{\psi}(x)=
        \left(
        \ba{cc}
        bL\widetilde{\psi}_{11}(x)L  &  ibL\widetilde{\psi}_{12}(x)L  \\
        -ibL\widetilde{\psi}^*_{12}(x)L  & bL
        \ea
        \right),
        \ee
where $\widetilde{\psi}_{11}^*(x)=\widetilde{\psi}_{11}(x)$, $\widetilde{\psi}_{12}(x)$ are matrices, which depend on $x$.
        \par
Further as in Section \ref{s5} we consider the special case of separated variables (when $\varepsilon=i$, $\delta=1$)
        \be\la{142}
        u(x,t)=e^{-\lambda t}u_{\lambda}(x), \;\;\;\; f(x,t)=e^{-\lambda t}f_{\lambda}(x), \;\;\;\; v(x,t)=e^{-\lambda t}v_{\lambda}(x).
        \ee
Analogously as in Section \ref{s5} if $u(x,t)$, $f(x,t)$, $v(x,t)$ from \rf{142} satisfy the open system \rf{129-1}, then
        \be\la{143}
        \left\{
        \ba{l}
        -\lambda f_{\lambda}(x)+A(x)f_{\lambda}(x)=\widetilde{\Phi}^*(x)\sigma_Au_{\lambda}(x)  \\
        i\frac{df_{\lambda}(x)}{dx}+B(x)f_{\lambda}(x)=\widetilde{\Phi}^*(x)\sigma_Bu_{\lambda}(x)  \\
        v_{\lambda}(x)=u_{\lambda}(x)-i\widetilde{\Phi}f_{\lambda}(x)
        \ea
        \right.
        \ee
which implies that $f_{\lambda}(x)=(A-\lambda I)^{-1}\widetilde{\Phi}^*(x)\sigma_Au_{\lambda}(x)$ and
        \be\la{145}
        v_{\lambda}(x)=u_{\lambda}(x)-i\widetilde{\Phi}(A-\lambda I)^{-1}\widetilde{\Phi}^*(x)\sigma_Au_{\lambda}(x)=W_A(\lambda)u_{\lambda}(x)
        \ee
where $W_A(\lambda)$ is the characteristic operator function of the operator $A$ and $\lambda$ does not belong to the spectrum of the operator $A$. 
        \par
The compatibility conditions \rf{113-6} and \rf{113-7} (when $\varepsilon=i$, $\delta=1$) now take the form
        \be\la{146}
        \lambda \sigma_B u_{\lambda}(x)+i\sigma_A\frac{du_{\lambda}(x)}{dx}+\psi(x)u_{\lambda}(x)=0,
        \ee
        \be\la{147}
        \lambda \sigma_B v_{\lambda}(x)+i\sigma_A\frac{dv_{\lambda}(x)}{dx}+\widetilde{\psi}(x)v_{\lambda}(x)=0.
        \ee
Let us denote $u_{\lambda}(x)=(u_1(x),u_2(x))$, $v_{\lambda}(x)=(v_1(x),v_2(x))$.
        Then the equality \rf{136} with the help of \rf{66} and \rf{138} takes the form
        \be\la{149}
        \ba{c}
        \lambda(u_1(x),u_2(x))
        \left(
        \ba{cc}
        L   &   0 \\
        0   &   0
        \ea
        \right)
        +i\left(\frac{du_1}{dx},\frac{du_2}{dx}\right)
        \left(
        \ba{cc}
        0   &   bL \\
        bL   &   0
        \ea
        \right)
        + \\ +
        (u_1(x),u_2(x))
        \left(
        \ba{cc}
        bL\psi_{11}L  &  ibL\psi_{12} \\L
        -ibL\psi_{12}^*L  &  bL
        \ea
        \right)=0,
        \ea
        \ee
i.e.
        \be\la{150}
        \lambda u_1+ib\frac{du_2}{dx}+bu_1L\psi_{11}-ibu_2L\psi_{12}^*=0,
        \ee
        \be\la{151}
        \frac{du_1}{dx}+u_1L\psi_{12}-iu_2=0.
        \ee
        \par
Now if the vector functions $u_1$, $u_2$ satisfy the equations \rf{150}, \rf{151}, the straightforward calculations show that $u_1(x)$ satisfies the Sturm-Liouville equation
        \be\la{152}
        -\frac{d^2u_1}{dx^2}+u_1((L\psi_{12})^2-L\psi_{11}-L\frac{d\psi_{12}}{dx})=\frac{\lambda}{b}u_1
        \ee
and if the next additional conditions are satisfied
        \be\la{153}
        \psi_{12}=\psi_{12}^*,  \;\;\;\;\;\;
        -\psi_{12}L\psi_{11}+\psi_{11}L\psi_{12}^*+\frac{d\psi_{11}}{dx}=0
        \ee
the vector function $u_2(x)$ satisfies the next Sturm-Liouville equation
        \be\la{153-1}
         -\frac{d^2u_2}{dx^2}+u_2((L\psi_{12})^2-L\psi_{11}+L\frac{d\psi_{12}}{dx})=\frac{\lambda}{b}u_2
        \ee
        \par
Analogously the compatibility condition \rf{113-7} about the output $v(x)=(v_1(x),v_2(x))$ when $\varepsilon=i$, $\delta=1$ (using \rf{66} and \rf{141}) takes the form
        \be\la{154-1}
\ba{c}
        \lambda(v_1(x),v_2(x))
        \left(
        \ba{cc}
        L   &   0 \\
        0   &   0
        \ea
        \right)
        +i\left(\frac{dv_1}{dx},\frac{dv_2}{dx}\right)
        \left(
        \ba{cc}
        0   &   bL \\
        bL   &   0
        \ea
        \right)
        + \\ +
        (v_1(x),v_2(x))
        \left(
        \ba{cc}
        bL\widetilde{\psi}_{11}L  &  ibL\widetilde{\psi}_{12} \\L
        -ibL\widetilde{\psi}_{12}^*L  &  bL
        \ea
        \right)=0,
        \ea
        \ee
i.e.
        \be\la{155}
        \left\{
        \ba{l}
        \frac{dv_1}{dx}+v_1L\widetilde{\psi}_{12}-iv_2=0 \\
        \lambda v_1+ib\frac{dv_2}{dx}+bv_1L\widetilde{\psi}_{11}-ibv_2L\widetilde{\psi}_{12}^*=0.
        \ea
        \right.
        \ee
If $v_1$ and $v_2$ satisfy the system \rf{155}, starightforward calculations show that they satisfy the following Sturm-Liouville equations correspondingly
        \be\la{156}
        -\frac{d^2v_1}{dx^2}+v_1((L\widetilde{\psi}_{12})^2-L\widetilde{\psi}_{11}-L\frac{d\widetilde{\psi}_{12}}{dx})=\frac{\lambda}{b}v_1,
        \ee
        \be\la{157}
        -\frac{d^2v_2}{dx^2}+v_2((L\widetilde{\psi}_{12})^2-L\widetilde{\psi}_{11}+L\frac{d\widetilde{\psi}_{12}}{dx})=\frac{\lambda}{b}v_2
        \ee
when we suppose the additional conditions
        \be\la{158}
        \widetilde{\psi}_{12}=\widetilde{\psi}_{12}^*, \;\;\;\;
        -\widetilde{\psi}_{12}L\widetilde{\psi}_{11}+\widetilde{\psi}_{11}L\widetilde{\psi}_{12}^*+\frac{d\widetilde{\psi}_{11}}{dx}=0.
        \ee
Consequently we prove the next theorem:
        \begin{thm}\la{t16}
        The input $u(x,t)=e^{-\lambda t}u_{\lambda}(x)=e^{-\lambda t}(u_1(x),u_2(x))$ and the output
        $v(x,t)=e^{-\lambda t}v_{\lambda}(x)=e^{-\lambda t}(v_1(x),v_2(x))$ of the open system \rf{129-1} satisfy the compatibility conditions
        \be\la{159}
        \ba{l}
        \frac{du_{\lambda}(x)}{dx}-i\sigma_A^{-1}(\lambda\sigma_B+\psi(x))u_\lambda(x)=0,  \\
        \frac{dv_{\lambda}(x)}{dx}-i\sigma_A^{-1}(\lambda\sigma_B+\widetilde{\psi}(x))v_\lambda(x)=0
        \ea
        \ee
and the characteristic operator function
        $W_A(\lambda)=I-i\widetilde{\Phi}(A-\lambda I)^{-1}\widetilde{\Phi}^*\sigma_A$
maps the input $u_{\lambda}(x)=(u_1(x),u_2(x))$, satisfying the Sturm-Liouville equations
        $$
        -\frac{d^2u_1}{dx^2}+u_1((L\psi_{12})^2-L\psi_{11}-L\frac{d\psi_{12}}{dx})=\frac{\lambda}{b}u_1,
        $$
        $$
        -\frac{d^2u_2}{dx^2}+u_2((L\psi_{12})^2-L\psi_{11}+L\frac{d\psi_{12}}{dx})=\frac{\lambda}{b}u_2
        $$
to the output $v_{\lambda}(x)=(v_1(x),v_2(x))=W_A(\lambda)u_{\lambda}$,
 which are solutions of the following Sturm-Liouville equations
 Schr\"{o}dinger equations
        $$
        -\frac{d^2v_1}{dx^2}+v_1((L\widetilde{\psi}_{12})^2-L\widetilde{\psi}_{11}-L\frac{d\widetilde{\psi}_{12}}{dx})=\frac{\lambda}{b}v_1,
        $$
        $$
        -\frac{d^2v_2}{dx^2}+v_2((L\widetilde{\psi}_{12})^2-L\widetilde{\psi}_{11}+L\frac{d\widetilde{\psi}_{12}}{dx})=\frac{\lambda}{b}v_2
        $$
in the case when the operator functions $\widetilde{\Phi}(x)$, $\psi(x)$, $\widetilde{\psi}(x)$ satisfy the conditions \rf{153}, \rf{158} ($x\in[x_0,x_1]$).
        \end{thm}
%
%
%
%

        \section{A special case of the input and the output of the open system, corresponding to the Korteweg-de Vries equation}\la{s8}
        \par
In this section we consider an open system, connected with obtaining of solutions of the nonlinear KdV equation in the special case of the separated variables in the input, the internal state, and the output of the corresponding open system. We derive what kind of differential equations are satisfied by the components of the input and the output, at first in the case when the operators $A$ and $B$ do not depend on the variables $x$ and $t$.
        \par
Let the operators $A=bB^3$, $B$ ($b\in\mathbb{R}$), the collective motions \rf{58}, the colligation $X$ from \rf{81-1} be like in Section \ref{s4} stated.
        \par
Let us consider the open system, corresponding to the colligation $X$ in the case when $\varepsilon=\delta=1$, i.e. the system
        \be\la{160}
        \left\{
        \ba{c}
        i\frac{\partial f}{\partial t}+Af=\widetilde{\Phi}^*\sigma_Au \\
        i\frac{\partial f}{\partial x}+Bf=\widetilde{\Phi}^*\sigma_Bu \\
        v=u-i\widetilde{\Phi}f
        \ea
        \right.
        \ee
where $u=u(x,t)$, $v=v(x,t)$, $f=f(x,t)$ are the input, the output and the state of the system \rf{160}.
        \par
From Theorem 3.3 in \cite{JMAA} or Theorem 4 in \cite{LIVSIC-AVISHAI} (in the considered case $\varepsilon=\delta=1$ we obtain the case, considered by M.S. Liv\v sic and Y. Avishai) it follows that the collective motions are compatible if and only if the input and the output satisfy the following partial differential equations (or matrix wave equations)
(or compatibility conditions)
        \be\la{161}
        \sigma_B\frac{\partial u}{\partial t}-\sigma_A\frac{\partial u}{\partial x}+i\gamma u=0,
        \ee
        \be\la{162}
        \sigma_B\frac{\partial v}{\partial t}-\sigma_A\frac{\partial v}{\partial x}+i\widetilde{\gamma} v=0.
        \ee
        \par
As in Section \ref{s5} and Section \ref{s6} we consider the input, the output and the state in the special case of separated variables  when
        \be\la{163}
        u(x,t)=e^{i\lambda t}u_{\lambda}(x), \;\;\;\; f(x,t)=e^{i\lambda t}f_{\lambda}(x), \;\;\;\; v(x,t)=e^{i\lambda t}v_{\lambda}(x).
        \ee
        \par
Now the open system \rf{160} takes the form
        \be\la{164}
        \left\{
        \ba{l}
        -\lambda f_{\lambda}(x)+Af_{\lambda}(x)=\widetilde{\Phi}^*\sigma_A u_{\lambda}(x)  \\
        i\frac{df_{\lambda}(x)}{dx}+Bf_{\lambda}(x)=\widetilde{\Phi}^*\sigma_B u_{\lambda}(x)  \\
        v_{\lambda}(x)=u_{\lambda}(x)-i\widetilde{\Phi}f_{\lambda}(x).
        \ea
        \right.
        \ee
Consequently $f_{\lambda}(x)=(A-\lambda I)^{-1}\widetilde{\Phi}^*\sigma_A u_{\lambda}(x)$ and
        \be\la{166}
        v_{\lambda}(x)=u_{\lambda}(x)-i\widetilde{\Phi}(A-\lambda I)^{-1}\widetilde{\Phi}^*\sigma_A u_{\lambda}(x)=W_A(\lambda)u_{\lambda}(x),
        \ee
where
        \be\la{166-1}
        W_A(\lambda)=I-i\widetilde{\Phi}(A-\lambda I)^{-1}\widetilde{\Phi}^*\sigma_A
        \ee
is the characteristic operator function of the operator $A$ and $\lambda$ does not beling to the spectrum of the operator $A$.
        \par
Now the compatibility conditions \rf{161} and \rf{162} imply that $u_{\lambda}(x)$ and $v_{\lambda}(x)$ satisfy the following differential equations
        \be\la{167}
        i\lambda\sigma_Bu_{\lambda}(x)-\sigma_A\frac{du_{\lambda}(x)}{dx}+i\gamma u_{\lambda}(x)=0,
        \ee
        \be\la{168}
        i\lambda\sigma_Bv_{\lambda}(x)-\sigma_A\frac{dv_{\lambda}(x)}{dx}+i\widetilde{\gamma} v_{\lambda}(x)=0.
        \ee
The form \rf{33-7} of the operators $\sigma_A$, $\sigma_B$, $\gamma$, the colligation condition for regular colligation
        $$
        \widetilde{\gamma}-\gamma=i(\sigma_A\widetilde{\Phi}\widetilde{\Phi}^*\sigma_B-\sigma_B\widetilde{\Phi}\widetilde{\Phi}^*\sigma_A)
        $$
and straightforward calculations show that
        \be\la{169}
        \widetilde{\gamma}=
        \left(
        \ba{ccc}
        -ibL(\pi_{13}-\pi_{13}^*)L   &   ibL\pi_{12}L   &   ibL\pi_{11}L  \\
        -ibL\pi_{12}^*L   &   0   &   bL   \\
        -ibL\pi_{11}L   &   bL   &   0
        \ea
        \right)
        \ee
where $\widetilde{\Phi}$ is defined by \rf{33-7} and the next selfadjoint matrix $\widetilde{\Phi}\widetilde{\Phi}^*$ has the block representation
        \be\la{170}
        \widetilde{\Phi}\widetilde{\Phi}^*=
        \left(
        \ba{ccc}
        \pi_{11}  &  \pi_{12}  &  \pi_{13}  \\
        \pi_{12}^* & \pi_{22}  &  \pi_{23}  \\
        \pi_{13}^* & \pi_{23}^* & \pi_{33}
        \ea
        \right)
        \ee
(with $\pi_{11}^*=\pi_{11}$, $\pi_{22}^*=\pi_{22}$, $\pi_{33}^*=\pi_{33}$).
        \begin{thm}\la{t17}
        The input
        $u(x,t)=e^{i\lambda t}u_{\lambda}(x)=e^{i\lambda t}(u_1(x),u_2(x),u_3(x))$ and the output
        $v(x,t)=e^{i\lambda t}v_{\lambda}(x)=e^{i\lambda t}(v_1(x),v_2(x),v_3(x))$
of the open system \rf{160} satisfy the compatibility conditions
        \be\la{172}
        \sigma_A\frac{du_{\lambda}(x)}{dx}-i(\lambda\sigma_B+\gamma)u_{\lambda}(x)=0,
        \ee
        \be\la{173}
        \sigma_A\frac{dv_{\lambda}(x)}{dx}-i(\lambda\sigma_B+\widetilde{\gamma})v_{\lambda}(x)=0
        \ee
and the characteristic operator function $W_A(\lambda)$ of the operator $A=bB^3$ from \rf{166-1} maps the input
        $u(x,t)=e^{i\lambda t}u_{\lambda}(x)=e^{i\lambda t}(u_1(x),u_2(x),u_3(x))$, satisfying the equations
        $$
        -\frac{d^3u_k}{dx^3}=\frac{i\lambda}{b}u_k, \;\;\;\;\;\; k=1,2,3,
        $$
%
%
to the output
        $v(x,t)=e^{i\lambda t}v_{\lambda}(x)=e^{i\lambda t}(v_1(x),v_2(x),v_3(x))$
which are solutions of the differential equations
        $$
        \frac{d^3v_k}{dx^3}+\frac{dv_k}{dx}p_k+v_kq_k=-\frac{i\lambda}{b}v_k, \;\;\;\;\;\; k=1,2,3,
        $$
%
%
where the matrices $p_k$, $q_k$ ($k=1,2,3$) have the form
        $$
        p_1=iL\pi_{12}-iL\pi_{12}^*-(L\pi_{11})^2, \;\;\;\;\;
        q_1=L(\pi_{13}-\pi_{13}^*)-iL\pi_{11}L\pi_{12}^*-iL\pi_{12}L\pi_{11},
        $$
        $$
        p_2=iL(\pi_{12}-iL\pi_{12}^*)-(L\pi_{11})^2, \;\;\;\;\;
        q_2=L(\pi_{13}-\pi_{13}^*)-iL\pi_{11}L\pi_{12}-iL\pi_{12}^*L\pi_{11},
        $$
        $$
        p_3=(L\pi_{11})^2+iL\pi_{12}, \;\;\;\;\;
        q_3=iL\pi_{12}^*-iL\pi_{11}L\pi_{12},
        $$
in the case when the matrices $\pi_{11}$, $\pi_{12}$, $\pi_{13}$ satisfy the following additional conditions
        $$
        (L\pi_{11})^2L\pi_{12}-L\pi_{12}(L\pi_{11})^2+iL\pi_{11}L(\pi_{13}-\pi_{13}^*)-iL(\pi_{13}-\pi_{13}^*)L\pi_{11}=0,
        $$
        $$
        \ba{c}
        (L\pi_{11})^2L(\pi_{13}-\pi_{13}^*)-L\pi_{11}L(\pi_{13}-\pi_{13}^*)L\pi_{11}+
        iL(\pi_{13}-\pi_{13}^*)L\pi_{12}-iL\pi_{12}L(\pi_{13}-\pi_{13}^*)
        + \\ +
        L\pi_{11}L\pi_{12}L\pi_{12}^*-L\pi_{12}L\pi_{12}^*L\pi_{11}=0,
        \ea
        $$
        $$
        iL(\pi_{13}-\pi_{13}^*)L\pi_{11}-iL\pi_{11}L(\pi_{13}-\pi_{13}^*)-iL\pi_{12}L\pi_{12}^*-iL\pi_{12}^*L\pi_{12}=0.
        $$
        \end{thm}
        \begin{proof}
        In the considered case of separated variables \rf{163} the compatibility conditions \rf{167} and \rf{168} (using the form \rf{33-7}, \rf{169} of the matrices $\sigma_A$, $\sigma_B$, $\gamma$, $\widetilde{\gamma}$) take the form
        \be\la{180}
        \ba{c}
        i\lambda(u_1,u_2,u_3)
        \left(
        \ba{ccc}
        L   &   0 & 0 \\
        0   &   0 & 0 \\
        0 & 0 & 0
        \ea
        \right)
        -\left(\frac{du_1}{dx},\frac{du_2}{dx},\frac{du_3}{dx}\right)
        \left(
        \ba{ccc}
        0   & 0 &  bL  \\
        0  & bL & 0  \\
        bL & 0 & 0
        \ea
        \right)
        + \\ +
        i(u_1,u_2,u_3)
        \left(
        \ba{ccc}
        0   &  0  &    0   \\
        0  &  0  &    bL  \\
        0  & bL  &  0
        \ea
        \right)=0,
        \ea
        \ee
        \be\la{181}
        \ba{c}
        i\lambda(v_1,v_2,v_3)
        \left(
        \ba{ccc}
        L   &   0 & 0 \\
        0   &   0 & 0 \\
        0 & 0 & 0
        \ea
        \right)
        -\left(\frac{dv_1}{dx},\frac{dv_2}{dx},\frac{dv_3}{dx}\right)
        \left(
        \ba{ccc}
        0   & 0 &  bL  \\
        0  & bL & 0  \\
        bL & 0 & 0
        \ea
        \right)
        + \\ +
        i(v_1,v_2,v_3)
        \left(
        \ba{ccc}
        -ibL(\pi_{13}-\pi_{13}^*)L   &   ibL\pi_{12}L   &   ibL\pi_{11}L  \\
        -ibL\pi_{12}^*L   &   0   &   bL   \\
        -ibL\pi_{11}L   &   bL   &   0
        \ea
        \right)=0.
        \ea
        \ee
        \par
Straightforward calculations show that the equations \rf{180} and \rf{181} imply that $u_1$, $u_2$, $u_3$ and $v_1$, $v_2$, $v_3$ satisfy the next differential equations
        \be\la{182}
        \left\{
        \ba{l}
        -\frac{du_1}{dx}+iu_2=0 \\
        -\frac{du_2}{dx}+iu_3=0 \\
        -b\frac{du_3}{dx}+i\lambda u_1=0
        \ea
        \right.
        \ee
and
        \be\la{183}
        \left\{
        \ba{l}
        \frac{dv_1}{dx}+v_1L\pi_{11}-iv_2=0 \\
        \frac{dv_2}{dx}+v_1L\pi_{12}-iv_3=0 \\
        -b\frac{dv_3}{dx}+bv_1L(\pi_{13}-\pi_{13}^*)+bv_2L\pi_{12}^*+bv_3L\pi_{11}+i\lambda v_1=0
        \ea
        \right.
        \ee
Now direct calculations proves the theorem.

        \end{proof}
%
%
%
%
        \section{The case when the operators $\mathbf{A}$ and $\mathbf{B}$ depend on the variable $\mathbf{x}$ and the Korteweg-de Vries equation}\la{s9}
        \par
        In this section, we expand the results obtained in Section \ref{s8} in the case when the operators $A$ and $B$ depend on the spatial variable $x$ analogously as in Section \ref{s7}.
        \par
Let us consider now regular colligations (or vessels) which depend on the spatial variable, i.e. the operators $A$, $B$, $\widetilde{\Phi}$ depend on the spatial variable $x$. This implies that if the operator $B$ is the triangular model of couplings of dissipative and antidissipative operators with real spectra, i.e $B$ is the triangular model \rf{21-1} with $\Delta=[-l,l]$, $A=bB^3$ then the matrix function $\Pi(w)$ depends also on $x$, i.e. $\Pi=\Pi(w,x)$.
        \par
        We consider the case of $\delta=\varepsilon=1$.
Analogously as in Section \ref{s7} we embed the operators $A=bB^3$ and $B$ from \rf{21-1} ($b\in\mathbb{R}$) in the strict colligation
        $$
        X=(A(x)=bB^3(x),B(x);H=\mathbf{L}^2(\Delta;\mathbb{C}^p),\widetilde{\Phi},E=\mathbb{C}^{3m};
        \sigma_A,\sigma_B,\psi(x),\widetilde{\psi}(x)),
        $$
where $\sigma_A$, $\sigma_B$, $\widetilde{\Phi}$ have the form \rf{33-7}. The colligation conditions are as in Section \ref{s7}: \rf{113-1-1}, \rf{113-2}, \rf{113-3} and the matrices $\psi(x)$ and $\widetilde{\psi}(x)$ satisfy the condition \rf{113-4} (when $\delta=\varepsilon=1$).
        \par
Then the corresponding open system has the form
        \be\la{184}
        \left\{
        \ba{l}
        i\frac{\partial f(x,t)}{\partial t}+A(x)f(x,t)=\widetilde{\Phi}^*(x)\sigma_Au(x,t)
        \\
        i\frac{\partial f(x,t)}{\partial x}+B(x)f(x,t)=\widetilde{\Phi}^*(x)\sigma_Bu(x,t)
        \\
        v(x,t)=u(x,t)-i\widetilde{\Phi}(x)f(x,t)
        \ea
        \right.
        \ee
($t_0\leq t\leq t_1$, $x_0\leq x\leq x_1$). (In this case the operators $\psi(x)$ and $\widetilde{\psi}(x)$ are selfadjoint.)
        \par
From Theorem \ref{t14}
in the case when  $\delta=\varepsilon=1$ it follows that the strong compatibility conditions at the input and at the output (i.e. the matrix wave equations) are
        \be\la{185}
        \sigma_B\frac{\partial u(x,t)}{\partial t}-\sigma_A\frac{\partial u(x,t)}{\partial x}+i\psi(x)u(x,t)=0,
        \ee
        \be\la{186}
        \sigma_B\frac{\partial v}{\partial t}(x,t)-\sigma_A\frac{\partial v(x,t)}{\partial x}+i\widetilde{\psi}(x)v(x,t)=0.
        \ee
Analogously as in the case of the Schr\"{o}dinger equation in Section \ref{s6} we can present the matrix functions $\psi(x)$ and $\widetilde{\psi}(x)$ in the form
        \be\la{187}
        \psi(x)=
        \left(
        \ba{ccc}
        bL\psi_{11}(x)L     &   ibL\psi_{12}(x)L  &  ibL\psi_{13}(x)L \\
        -ibL\psi_{12}^*(x)L &   0   &   bL  \\
        -ibL\psi_{13}(x)L       &   bL  &   0
        \ea
        \right),
        \ee
where $\psi_{11}^*(x)=\psi_{11}(x)$, $\psi_{13}^*(x)=\psi_{13}(x)$ and
        \be\la{188}
        \widetilde{\psi}(x)=
        \left(
        \ba{ccc}
        bL\widetilde{\psi}_{11}(x)L     &   ibL\widetilde{\psi}_{12}(x)L  &  ibL\widetilde{\psi}_{13}(x)L \\
        -ibL\widetilde{\psi}_{12}^*(x)L &   0   &   bL  \\
        -ibL\widetilde{\psi}_{13}(x)L       &   bL  &   0
        \ea
        \right),
        \ee
where $\widetilde{\psi}_{11}^*(x)=\widetilde{\psi}_{11}(x)$, $\widetilde{\psi}_{13}^*(x)=\widetilde{\psi}_{13}(x)$ and the  matrices $\psi_{11}(x)$, $\psi_{12}(x)$, $\psi_{13}(x)$, $\widetilde{\psi}_{11}(x)$, $\widetilde{\psi}_{12}(x)$, $\widetilde{\psi}_{13}(x)$ are $m\times m$ matrices.
        \par
Next we consider the special case of separated variables for the input, the inner state, and the output
        \be\la{189}
        u(x,t)=e^{i\lambda t}u_{\lambda}(x), \;\;\;\; f(x,t)=e^{i\lambda t}f_{\lambda}(x), \;\;\;\; v(x,t)=e^{i\lambda t}v_{\lambda}(x).
        \ee
Then the open system \rf{184} takes the form
        \be\la{190}
        \left\{
        \ba{l}
        -\lambda f_{\lambda}(x)+A(x)f_{\lambda}(x)=\widetilde{\Phi}^*(x)\sigma_Au_{\lambda}(x)  \\
        i\frac{df_{\lambda}(x)}{dx}+B(x)f_{\lambda}(x)=\widetilde{\Phi}^*(x)\sigma_Bu_{\lambda}(x)  \\
        v_{\lambda}(x)=u_{\lambda}(x)-i\widetilde{\Phi}f_{\lambda}(x)
        \ea
        \right.
        \ee
which implies that $f_{\lambda}(x)=(A-\lambda I)^{-1}\widetilde{\Phi}^*(x)\sigma_Au_{\lambda}(x)$ and
%
        \be\la{192}
        v_{\lambda}(x)=u_{\lambda}(x)-i\widetilde{\Phi}(A-\lambda I)^{-1}\widetilde{\Phi}^*(x)\sigma_Au_{\lambda}(x)=W_A(\lambda)u_{\lambda}(x)
        \ee
where $W_A(\lambda)$ is the characteristic operator function of the operator $A$ and $\lambda$ does not belong to the spectrum of the operator $A$. 
        \par
On the other hand the matrix wave equations \rf{185} and \rf{186} take the form
        \be\la{193}
        i\lambda \sigma_B u_{\lambda}(x)-\sigma_A\frac{du_{\lambda}(x)}{dx}+i\psi(x)u_{\lambda}(x)=0,
        \ee
        \be\la{194}
        i\lambda \sigma_B v_{\lambda}(x)-\sigma_A\frac{dv_{\lambda}(x)}{dx}+i\widetilde{\psi}(x)v_{\lambda}(x)=0.
        \ee
Now the form \rf{33-7} of the matrices $\sigma_A$, $\sigma_B$, $\widetilde{\Phi}(x)$, the form \rf{187} and \rf{188} of the matrix function $\psi(x)$ and $\widetilde{\psi}(x)$, the presentation
        \be\la{195}
        u_{\lambda}(x)=(u_1(x),u_2(x),u_3(x)), \;\;\;\; v_{\lambda}(x)=(v_1(x),v_2(x),v_3(x)),
        \ee
together with the equalities \rf{193} and \rf{194} imply that $u_{\lambda}(x)$ and $v_{\lambda}(x)$ satisfy the following partial differential equations
        \be\la{196}
        \ba{c}
        i\lambda(u_1,u_2,u_3)
        \left(
        \ba{ccc}
        L   &   0 & 0 \\
        0   &   0 & 0 \\
        0 & 0 & 0
        \ea
        \right)
        -\left(\frac{du_1}{dx},\frac{du_2}{dx},\frac{du_3}{dx}\right)
        \left(
        \ba{ccc}
        0   & 0 &  bL  \\
        0  & bL & 0  \\
        bL & 0 & 0
        \ea
        \right)
        + \\ +
        i(u_1,u_2,u_3)
        \left(
        \ba{ccc}
        bL\psi_{11}(x)L   &  ibL\psi_{12}(x)L  &    ibL\psi_{13}(x)L   \\
        -ibL\psi_{12}^*(x)L &   0   &   bL  \\
        -ibL\psi_{13}(x)L       &   bL  &   0
        \ea
        \right)=0,
        \ea
        \ee
        \be\la{197}
        \ba{c}
        i\lambda(v_1,v_2,v_3)
        \left(
        \ba{ccc}
        L   &   0 & 0 \\
        0   &   0 & 0 \\
        0 & 0 & 0
        \ea
        \right)
        -\left(\frac{dv_1}{dx},\frac{dv_2}{dx},\frac{dv_3}{dx}\right)
        \left(
        \ba{ccc}
        0   & 0 &  bL  \\
        0  & bL & 0  \\
        bL & 0 & 0
        \ea
        \right)
        + \\ +
        i(v_1,v_2,v_3)
        \left(
        \ba{ccc}
        bL\widetilde{\psi}_{11}(x)L     &   ibL\widetilde{\psi}_{12}(x)L  &  ibL\widetilde{\psi}_{13}(x)L \\
        -ibL\widetilde{\psi}_{12}^*(x)L &   0   &   bL  \\
        -ibL\widetilde{\psi}_{13}(x)L       &   bL  &   0
        \ea
        \right)=0.
        \ea
        \ee
The equations \rf{196} and \rf{197} imply that $u_{\lambda}(x)$ and $v_{\lambda}(x)$ from \rf{195} satisfy the next systems correspondingly
        \be\la{200}
        \left\{
        \ba{l}
        \frac{du_1}{dx}+u_1L\psi_{13}-iu_2=0 \\
        \frac{du_2}{dx}+u_1L\psi_{12}-iu_3=0 \\
        -b\frac{du_3}{dx}+ibu_1L\psi_{11}+bu_2L\psi_{12}^*+bu_3L\psi_{13}+i\lambda u_1=0
        \ea
        \right.
        \ee
and
        \be\la{201}
        \left\{
        \ba{l}
        \frac{dv_1}{dx}+v_1L\widetilde{\psi}_{13}-iv_2=0 \\
        \frac{dv_2}{dx}+v_1L\widetilde{\psi}_{12}-iv_3=0 \\
        -b\frac{dv_3}{dx}+ibv_1L\widetilde{\psi}_{11}+bv_2L\widetilde{\psi}_{12}^*+bv_3L\widetilde{\psi}_{13}+i\lambda v_1=0.
        \ea
        \right.
        \ee
        \par
Straightforward calculations show that in the case when the matrix functions $\psi_{11}$, $\psi_{121}$, $\psi_{13}$, $\widetilde{\psi}_{11}$, $\widetilde{\psi}_{12}$, $\widetilde{\psi}_{13}$ satisfy additional conditions (mentioned below) the next theorem is true.
        \begin{thm}\la{t18}
        The input $u(x,t)$ and the output $v(x,t)$ from \rf{189} and \rf{195} of the open system \rf{184} satisfy the compatibility conditions
        $$
        \ba{l}
        \frac{du_{\lambda}(x)}{dx}-i\sigma_A^{-1}(\lambda \sigma_B +\psi(x))u_{\lambda}(x)=0,
        \\
        \frac{dv_{\lambda}(x)}{dx}-i\sigma_A^{-1}(\lambda\sigma_B+\widetilde{\psi}(x))v_{\lambda}(x)=0.
        \ea
        $$
and the characteristic operator function $W_A(\lambda)$ from \rf{166-1} of the operator $A=bB^3$ maps the input $u(x,t)=e^{it\lambda}u_{\lambda}(x)=e^{it\lambda}(u_1(x),u_2(x),u_3(x))$ satisfying the equations
        \be\la{204}
        \left\{
        \ba{l}
        \frac{d^3u_1}{dx^3}(x)+\frac{du_1}{dx}(x)p_1(x)+u_1(x)q_1(x)=-\frac{i\lambda}{b}u_1(x) \\
        \frac{d^3u_2}{dx^3}(x)+\frac{du_2}{dx}(x)p_2(x)+u_2(x)q_2(x)=-\frac{i\lambda}{b}u_2(x) \\
        \frac{d^3u_3}{dx^3}(x)+\frac{du_3}{dx}(x)p_3(x)+u_3(x)q_3(x)=-\frac{i\lambda}{b}u_3(x)
        \ea
        \right.
        \ee
to the output $v(x,t)=e^{it\lambda}v_{\lambda}(x)=e^{it\lambda}(v_1(x),v_2(x),v_3(x))$ which are solutions of the following differential equations
        \be\la{205}
        \left\{
        \ba{l}
        \frac{d^3v_1}{dx^3}(x)+\frac{dv_1}{dx}(x)\widetilde{p}_1(x)+v_1(x)\widetilde{q}_1(x)=-\frac{i\lambda}{b}v_1(x) \\
        \frac{d^3v_2}{dx^3}(x)+\frac{dv_2}{dx}(x)\widetilde{p}_2(x)+v_2(x)\widetilde{q}_2(x)=-\frac{i\lambda}{b}v_2(x) \\
        \frac{d^3v_3}{dx^3}(x)+\frac{dv_3}{dx}(x)\widetilde{p}_3(x)+v_3(x)\widetilde{q}_3(x)=-\frac{i\lambda}{b}v_3(x),
        \ea
        \right.
        \ee
where the matrix functions $p_k(x)$, $q_k(x)$, $\widetilde{p}_k(x)$, $\widetilde{q}_k(x)$ ($k=1,2,3$) have the form
        $$
        p_1(x)=iL\psi_{12}-iL\psi_{12}^*+2L\frac{d\psi_{13}}{dx}-(L\psi_{13})^2
        $$
        $$
        p_2(x)=iL\psi_{12}-iL\psi_{12}^*-L\frac{d\psi_{13}}{dx}-(L\psi_{13})^2
        $$
        $$
        p_3(x)=iL\psi_{12}-(L\psi_{13})^2-L\frac{d\psi_{13}}{dx}-iL\psi_{12}^*
        $$
        $$
        q_1(x)=L\frac{d^2\psi_{13}}{dx^2}+iL\frac{d\psi_{12}}{dx}+iL\psi_{11}-iL\psi_{12}L\psi_{13}-iL\psi_{13}L\psi_{12}^*
        -L\frac{d\psi_{13}}{dx}L\psi_{13}
        $$
        $$
        q_2(x)=-iL\psi_{13}L\psi_{12}+iL\psi_{11}+2iL\frac{d\psi_{12}}{dx}-iL\frac{d\psi_{12}^*}{dx}-iL\psi_{12}^*L\psi_{13}
        $$
        $$
        q_3(x)=-L\frac{d\psi_{13}}{dx}L\psi_{13}-L\frac{d^2\psi_{13}}{dx^2}+iL\psi_{11}-2iL\frac{d\psi_{12}^*}{dx}-iL\psi_{12}^*L\psi_{13}
        -L\psi_{13}L\frac{d\psi_{13}}{dx}-iL\psi_{13}L\psi_{12}
        $$
        $$
        \widetilde{p}_1(x)=iL\widetilde{\psi}_{12}-iL\widetilde{\psi}_{12}^*+2L\frac{d\widetilde{\psi}_{13}}{dx}-(L\widetilde{\psi}_{13})^2
        $$
        $$
        \widetilde{p}_2(x)=iL\widetilde{\psi}_{12}-iL\widetilde{\psi}_{12}^*-L\frac{d\widetilde{\psi}_{13}}{dx}-(L\widetilde{\psi}_{13})^2
        $$
        $$
        \widetilde{p}_3(x)=iL\widetilde{\psi}_{12}-(L\widetilde{\psi}_{13})^2-L\frac{d\widetilde{\psi}_{13}}{dx}-iL\widetilde{\psi}_{12}^*
        $$
        $$
        \widetilde{q}_1(x)=L\frac{d^2\widetilde{\psi}_{13}}{dx^2}+iL\frac{d\widetilde{\psi}_{12}}{dx}+iL\widetilde{\psi}_{11}
        -iL\widetilde{\psi}_{12}L\widetilde{\psi}_{13}-iL\widetilde{\psi}_{13}L\widetilde{\psi}_{12}^*
        -L\frac{d\widetilde{\psi}_{13}}{dx}L\widetilde{\psi}_{13}
        $$
        $$
        \widetilde{q}_2(x)=-iL\widetilde{\psi}_{13}L\widetilde{\psi}_{12}+iL\widetilde{\psi}_{11}+2iL\frac{d\widetilde{\psi}_{12}}{dx}
        -iL\frac{d\widetilde{\psi}_{12}^*}{dx}-iL\widetilde{\psi}_{12}^*L\widetilde{\psi}_{13}
        $$
        $$
        \widetilde{q}_3(x)=-L\frac{d\widetilde{\psi}_{13}}{dx}L\widetilde{\psi}_{11}-L\frac{d^2\widetilde{\psi}_{13}}{dx^2}
        +iL\widetilde{\psi}_{11}-2iL\frac{d\widetilde{\psi}_{12}^*}{dx}-iL\widetilde{\psi}_{12}^*L\widetilde{\psi}_{13}
        -L\widetilde{\psi}_{13}L\frac{d\widetilde{\psi}_{13}}{dx}-iL\widetilde{\psi}_{13}L\widetilde{\psi}_{12}.
        $$
        \end{thm}
        \par
        It has to mention that the equations \rf{204} and \rf{205} are satisfied when we suppose the next additional conditions
        $$
        \ba{c}
        \psi_{11}L\psi_{13}-\psi_{13}L\psi_{11}+\psi_{13}L\psi_{13}L\psi_{12}-\psi_{12}L\psi_{13}L\psi_{13}-
        \frac{d\psi_{13}}{dx}L\psi_{12}
        - \\ -
        \psi_{12}L\frac{d\psi_{13}}{dx}-2\psi-{13}L\frac{d\psi_{12}}{dx}+\frac{d\psi_{11}}{dx}+\frac{d^2\psi_{12}}{dx^2}=0,
        \ea
        $$
        $$
        \psi_{11}L\psi_{13}-\psi_{13}L\psi_{11}+i\psi_{12}L\psi_{12}^*-i\psi_{12}^*L\psi_{12}-
        \frac{d\psi_{12}^*}{dx}L\psi_{13}-\psi_{12}^*L\frac{d\psi_{13}}{dx}+2\frac{d\psi_{11}}{dx}-\frac{d^2\psi_{12}^*}{dx^2}=0,
        $$
        $$
        \ba{c}
        \psi_{12}L\psi_{12}^*L\psi_{13}-\psi_{13}L\psi_{12}L\psi_{12}^*+i\psi_{13}L\psi_{11}L\psi_{13}-i\psi_{13}L\psi_{13}L\psi_{11}+
        i\psi_{11}L\psi_{12}-i\psi_{12}L\psi_{11}
        - \\ -
        i\frac{d\psi_{11}}{dx}L\psi_{13}-i\psi_{11}L\frac{d\psi_{13}}{dx}+i\frac{d\psi_{13}}{dx}L\psi_{11}+2i\psi_{13}L\frac{d\psi_{11}}{dx}
        + \\ +
        2\psi_{12}L\frac{d\psi_{12}^*}{dx}+\frac{d\psi_{12}}{dx}L\psi_{12}^*-\frac{d^2\psi_{11}}{dx^2}=0,
        \ea
        $$
        $$
        \ba{c}
        \widetilde{\psi}_{11}L\widetilde{\psi}_{13}-\widetilde{\psi}_{13}L\widetilde{\psi}_{11}+
        \widetilde{\psi}_{13}L\widetilde{\psi}_{13}L\widetilde{\psi}_{12}-\widetilde{\psi}_{12}L\widetilde{\psi}_{13}L\widetilde{\psi}_{13}-
        \frac{d\widetilde{\psi}_{13}}{dx}L\widetilde{\psi}_{12}
        - \\ -
        \widetilde{\psi}_{12}L\frac{d\widetilde{\psi}_{13}}{dx}-2\widetilde{\psi}-{13}L\frac{d\widetilde{\psi}_{12}}{dx}+
        \frac{d\widetilde{\psi}_{11}}{dx}+\frac{d^2\widetilde{\psi}_{12}}{dx^2}=0,
        \ea
        $$
        $$
        \widetilde{\psi}_{11}L\widetilde{\psi}_{13}-\widetilde{\psi}_{13}L\widetilde{\psi}_{11}+
        i\widetilde{\psi}_{12}L\widetilde{\psi}_{12}^*-i\widetilde{\psi}_{12}^*L\widetilde{\psi}_{12}-
        \frac{d\widetilde{\psi}_{12}^*}{dx}L\widetilde{\psi}_{13}-\widetilde{\psi}_{12}^*L\frac{d\widetilde{\psi}_{13}}{dx}+
        2\frac{d\widetilde{\psi}_{11}}{dx}-\frac{d^2\widetilde{\psi}_{12}^*}{dx^2}=0,
        $$
        $$
        \ba{c}
        \widetilde{\psi}_{12}L\widetilde{\psi}_{12}^*L\widetilde{\psi}_{13}-\widetilde{\psi}_{13}L\widetilde{\psi}_{12}L\widetilde{\psi}_{12}^*+
        i\widetilde{\psi}_{13}L\widetilde{\psi}_{11}L\widetilde{\psi}_{13}-i\widetilde{\psi}_{13}L\widetilde{\psi}_{13}L\widetilde{\psi}_{11}+
        i\widetilde{\psi}_{11}L\widetilde{\psi}_{12}-i\widetilde{\psi}_{12}L\widetilde{\psi}_{11}
        - \\ -
        i\frac{d\widetilde{\psi}_{11}}{dx}L\widetilde{\psi}_{13}-i\widetilde{\psi}_{11}L\frac{d\widetilde{\psi}_{13}}{dx}+
        i\frac{d\widetilde{\psi}_{13}}{dx}L\widetilde{\psi}_{11}+2i\widetilde{\psi}_{13}L\frac{d\widetilde{\psi}_{11}}{dx}
        + \\ +
        2\widetilde{\psi}_{12}L\frac{d\widetilde{\psi}_{12}^*}{dx}+\frac{d\widetilde{\psi}_{12}}{dx}L\widetilde{\psi}_{12}^*-
        \frac{d^2\widetilde{\psi}_{11}}{dx^2}=0.
        \ea
        $$
        \par
These additional conditions will be simpler in the case of specific choice of the $m\times m$ matrix functions $\psi_{11}(x)$, $\psi_{12}(x)$, $\psi_{13}(x)$, $\widetilde{\psi}_{11}(x)$, $\widetilde{\psi}_{12}(x)$, $\widetilde{\psi}_{13}(x)$ or when they are scalar functions (i.e. $m=1$).
        \par
Similar differential equation can be obtained for the input and the output of the open systems for appropriate couples and triplets of operators corresponding to the Sine-Gordon equation, the Davey-Stewartson equation and etc.
The obtained connection between the generalized Gelfand-Levitan-Marchenko equation and complete characteristic function or characteristic operator function (transfer function) of one of the operators from the couples and the triplets when one of the operators is dissipative with spectrum consisting only on the eigenvalues can be applied for finding the connection with Blashke product, Bourgain algebras, Henkel operators, using results in \cite{HRISTOV1}, \cite{HRISTOV2}.
        \par
Finally, it is worth to mention that the obtained results in this paper can be expanded to the case n-tuples of commuting nonselfadjoint operators when one of the operators belongs to the class of nondissipative unbounded operators, presented as a regular coupling of dissipative and antidissipative $K^r$-operators with real spectra and with different domains of the operator and its adjoint (i.e. closed operators in a Hilbert space whose Hermitian part has deficiency index $(r, r)$ ($0 <r<\infty$) and a nonempty resolvent set), using the triangular model of these operators, introduced and investigated in \cite{BAN3}, \cite{IEOT2}.

\end{document}